\documentclass[12pt,letterpaper,oneside,reqno]{amsart}

\usepackage[utf8]{inputenc}
\usepackage[english]{babel}

\usepackage[margin=1in]{geometry}
\usepackage{mathtools}
\usepackage{enumerate}

\usepackage{amsthm}
\usepackage{amssymb}
\usepackage{amsfonts}
\usepackage[mathcal]{eucal}
\usepackage{bbm}
\usepackage{latexsym}
\usepackage{mathrsfs}
\usepackage{stmaryrd}

\swapnumbers
\theoremstyle{plain}
\newtheorem{theorem}{Theorem}[section]

\newtheorem{proposition}[theorem]{Proposition}

\newtheorem{corollary}[theorem]{Corollary}

\theoremstyle{definition}

\newtheorem{definition}[theorem]{Definition}

\newtheorem{notation}[theorem]{Notation}

\newtheorem{remark}[theorem]{Remark}
\newtheorem{remarks}[theorem]{Remarks}

\theoremstyle{remark}
\newtheorem*{claim}{Claim}

\makeatletter
\renewcommand\subsection{\@startsection{subsection}{2}%
  \z@{-.5\linespacing\@plus-.7\linespacing}{.5\linespacing}%
  {\normalfont\scshape}}
\renewcommand\subsubsection{\@startsection{subsubsection}{3}%
  \z@{.5\linespacing\@plus.7\linespacing}{-.5em}%
  {\normalfont\scshape}}
\makeatother

\DeclareMathOperator{\card}{card}

\DeclareMathOperator{\ev}{ev}

\DeclareMathOperator{\Mod}{Mod}

\DeclareMathOperator{\osc}{osc}

\DeclareMathOperator{\Smpl}{Smpl}

\DeclareMathOperator{\thry}{Th}
\DeclareMathOperator{\Th}{Th}
\DeclareMathOperator{\tp}{tp}

\DeclareMathOperator{\Ulim}{\mathcal{U}lim}

\DeclareMathOperator{\var}{var}

\DeclareMathOperator{\wneg}{\mathrel{\ooalign{\hss$\neg$\hss\cr\kern0.2ex\raise1.0ex\hbox{\scalebox{0.7}{w}}}}}

\renewcommand{\restriction}{\mathord{\upharpoonright}}

\def\Aapp{\mathrel\cA\joinrel\approx}
\def\appmod{\mathrel|\joinrel\approx}

\newcommand{\absR}[1]{\ensuremath{{\nrm{#1}}_{\RR}}}
\newcommand{\ab}{\ensuremath{a_{\bullet}}}
\newcommand{\appto}{\mathrel{\ooalign{\hss$\Leftarrow$\hss\cr\kern0.9ex\raise0.25ex\hbox{\scalebox{0.6}{$\sim$}}}}}
\newcommand{\appby}{\mathrel{\ooalign{\hss$\Rightarrow$\hss\cr\kern0.3ex\raise0.25ex\hbox{\scalebox{0.6}{$\sim$}}}}}
\newcommand{\AL}{\ensuremath{\cA_{\mathrm{L}}}}

\newcommand{\BA}{\mathbb{A}}
\newcommand{\bb}{\ensuremath{b_{\bullet}}}

\newcommand{\BS}{\mathbb{S}}
\newcommand{\bS}{\mathbf{S}}
\newcommand{\bU}{\mathbb{U}}
\newcommand{\cA}{\mathcal{A}}

\newcommand{\cC}{\mathcal{C}}
\newcommand{\cD}{\ensuremath{\mathcal{D}}}

\newcommand{\cL}{\ensuremath{\mathcal{L}}}
\newcommand{\cM}{\ensuremath{\mathcal{M}}}
\newcommand{\cN}{\mathcal{N}}
\newcommand{\compl}[1]{\ensuremath{{#1}^{\mathrm{c}}}}
\newcommand{\cP}{\mathcal{P}}

\newcommand{\cU}{\mathcal{U}}

\newcommand{\DD}{\ensuremath{\mathbb{D}}}
\newcommand{\dd}{\ensuremath{\mathrm{d}}}
\newcommand{\Eb}{E_{\bullet}}
\newcommand{\fb}{\ensuremath{\varphi_{\bullet}}}

\newcommand{\hcM}{\Hat{\cM}}
\newcommand{\hcN}{\Hat{\cN}}

\newcommand{\lamp}{\ensuremath{\lambda^+}}
\newcommand{\LL}{\ensuremath{\widetilde{L}}}
\newcommand{\Los}{{\L}o\'s}

\newcommand{\lixd}{\ensuremath{\ell^\infty(X_\lambda,d_\lambda)_{\Lambda}}}

\newcommand{\LO}{\ensuremath{\cL_{\Omega}^{\infty}}}

\newcommand{\Mb}{M_{\bullet}}
\newcommand{\muL}{\ensuremath{\mu_{\mathrm{L}}}}

\newcommand{\nLO}[1]{\normZ{#1}{\LO}}

\newcommand{\NN}{\mathbb{N}}
\newcommand{\normZ}[2]{\ensuremath{{\norm{#1}}_{#2}}}
\newcommand{\norm}[1]{\left\|#1\right\|}
\newcommand{\nrm}[1]{\left|#1\right|}
\newcommand{\omp}{\ensuremath{\omega^+}}

\newcommand{\Pfin}{\cP_{\!\mathrm{fin}}}

\newcommand{\PfinD}{\Pfin(\cD)}

\newcommand{\PP}{\mathbf{P}}
\newcommand{\QQ}{\mathbb{Q}}

\newcommand{\RR}{\mathbb{R}}
\newcommand{\sC}{\ensuremath{\mathscr{C}}}
\newcommand{\sD}{\ensuremath{\mathscr{D}}}

\newcommand{\ssb}{\ensuremath{s_{\bullet}}}

\newcommand{\TT}{\ensuremath{\widetilde{T}}}

\begin{document}
\title[Model theory and metric convergence I]
{Model theory and metric convergence I: \\
Metastability and dominated convergence}

\author[E. Dueñez]{Eduardo Dueñez}

\address{Department of Mathematics\\
  The University of Texas at San Antonio\\
  One UTSA Circle\\
  San Antonio, TX 78249-0664\\
  U.S.A.}

\email{eduardo.duenez@utsa.edu}

\author[J. N. Iovino]{José N. Iovino}

\address{Department of Mathematics\\
  The University of Texas at San Antonio\\
  One UTSA Circle\\
  San Antonio, TX 78249-0664\\
  U.S.A.}
  
\email{jose.iovino@utsa.edu}

\dedicatory{Dedicated to Ward Henson with gratitude on the occasion of his retirement.}

\date{\today}
\thanks{This research was partially funded by NSF grant DMS-1500615}


\begin{abstract}
We study Tao's finitary viewpoint of convergence in metric spaces, as captured by the notion of metastability.
We adopt the perspective of continuous model theory.
We show that, in essence, metastable convergence with a given rate is the only formulation of metric convergence that can be captured by a theory in continuous first-order logic, a result we call the Uniform Metastability Principle. 
Philosophically, this principle amounts to the following meta-theorem:
\emph{``If a classical statement about convergence in metric structures is refined to a statement about metastable convergence with some rate, then the validity of the original statement implies the validity of its metastable version.''} 
As an instance of this phenomenon, we formulate an abstract version of Tao's Metastable Dominated Convergence Theorem as a statement about axiomatizable classes of metric structures, and show that it is a direct consequence of the Uniform Metastability Principle. 
\end{abstract}

\maketitle

\section*{Introduction}
The concept of convergence in metric spaces is fundamental in analysis. 
The present article is the first of a series focusing on results, both classical and new, in which the convergence of some sequence(s)---or, more generally, some nets---follows from suitable hypotheses. 
We shall use the loose nomenclature ``convergence theorem'' for any result of this kind; 
the best known such results are the classical Monotone and Dominated Convergence theorems, as well as the ergodic convergence theorems of von Neumann and Birkhoff. 
For simplicity, we assume that all metric spaces under consideration are complete (an alternative perspective would be the study of theorems about Cauchy sequences and nets in metric spaces not necessarily complete).

Given a convergence theorem, it is natural to ask whether it admits refinements whereby the conclusion states a stricter mode of convergence. 
When the statement of a convergence theorem involves a collection of sequences, the classical refinement of the property of simultaneous convergence is that of uniform convergence. 
However, few convergence theorems in analysis admit a natural refinement implying uniform convergence. 
Furthermore, even in the rare cases when uniform convergence of a family of sequences is implied, the parameters of uniform convergence are rarely universal: 
Typically, uniform convergence will hold in every structure (consisting of the ambient metric space, sequences therein, plus any other necessary ingredients) satisfying suitable hypotheses, but not uniformly across all such structures.
In fact, even if a convergence theorem should refer to the convergence of a single sequence, our focus shall be on the entirety of structures to which the theorem applies, and hence on all instances of relevant sequences. 
Thus, even when a ``single'' sequence is under immediate discussion, we ask whether its mode of convergence admits parameters that are uniform over all instances.

Tao~\cite{Tao:2008} introduced the notion of \emph{metastable convergence} (Definition~\ref{def:metastable}), which is an equivalent formulation of the usual Cauchy property (Proposition~\ref{prop:metastable-Cauchy}).%
\footnote{Before Tao's paper, the concept of metastability had been used  by Avigad, Gerhardy, Kohlenbach, and Towsner, under different nomenclature, in mathematical logic, specifically, in the area of proof mining. See the appendix for details.}
The metastable-convergence viewpoint leads to the notion of \emph{uniformly metastable convergence} (Definition~\ref{def:unif-metastab}), which is not only a metastable analogue of the classical property of uniform convergence of a family of sequences, but also a generalization thereof (Remark~\ref{rem:unif-metastab-Cauchy}). 
Tao obtains a metastable version of the classical Dominated Convergence Theorem that holds with metastable rates that are universal; this result plays a crucial role in the proof his remarkable result on the convergence of ergodic averages for polynomial abelian group actions~\cite{Tao:2008}. 
Walsh's subsequent generalization~\cite{Walsh:2012} of Tao's theorem to polynomial nilpotent group actions relies on a similar convergence theorem.

In this article we prove that metastability with a given rate is the only formulation of metric convergence that can be captured by a theory in continuous first-order logic (Proposition~\ref{prop:metastab-1st-ord}). 
This is a precise statement of Tao's observation that metastable convergence with a prescribed rate is a ``finitary'' property~\cite{Tao:Metastability}.
The conceptual backbone of the manuscript is the Uniform Metastability Principle (Proposition~\ref{prop:U-metastab}), which may be formulated as the following meta-theorem: 
\emph{``If a classical statement about convergence in metric structures is refined to a statement about metastable convergence with some rate, then the validity of the original statement implies the validity of its metastable version.''} 
As an instance of this phenomenon, we obtain a soft proof of a version of the Metastable Dominated Convergence Theorem. 
Our proof depends on neither infinitary arguments \emph{\`a la} Tao~\cite{Tao:2008}, nor recursive arguments and constructive analysis \emph{\`a la} Avigad \emph{et al}~\cite{Avigad-Dean-Rute:2012}.
We show that the Uniform Metastability Principle follows directly from the fundamental theorem of model theory of metric structures, namely, the Compactness Theorem.
We believe that, in spite of its simplicity, this principle captures a certain philosophical view revealing the scope of applicability of model-theoretic methods to the study of convergence in metric spaces. 
Although anticipated by Avigad and the second author~\cite{Avigad-Iovino:2013}, we are not aware of a purely model-theoretical formulation of this principle hitherto.

In order to make the results of this paper accessible to readers with no prior background in logic, particularly to researchers in analysis, the last section (Section~\ref{sec:metric-structures}) is a self-contained tutorial on the basics of model theory of metric structures.
The literature contains several equivalent formulations of this theory; 
the one in current widespread use is the framework of first-order continuous logic developed by Ben Yaacov and Usvyatsov~\cite{Ben-Yaacov-Usvyatsov:2010,Ben-Yaacov-Berenstein-Henson-Usvyatsov:2008} building upon ideas of Chang-Keisler~\cite{Chang-Keisler:1961,Chang-Keisler:1966} and Henson~\cite{Henson-Iovino:2002}. 
We use the language of continuous approximations originally developed by Henson, as we feel that it is simpler and more natural for the applications at hand. 
The tutorial of Section~\ref{sec:metric-structures} parallels portions of earlier introductions given by Henson and the second author that emphasize Banach spaces~\cite{Henson-Iovino:2002};
however, the present exposition places greater emphasis on metric structures and on topics of direct interest for the study of convergence theorems, particularly structures that are hybrid in the sense that they include metric sorts alongside discrete sorts.
Readers already comfortable with model theory of metric structures from this perspective need to consult Section~\ref{sec:metric-structures} only for reference.

Section~\ref{sec:metastability}, which does not depend on model theory at all, is an introduction to the metastable viewpoint of convergence.
As an attempt to shed light into the finitary nature of uniformly metastable convergence, Subsection~\ref{sec:metastability}-\ref{sec:motivation} discusses the relation between the usual Cauchy criterion and Tao's notion of metastable convergence. 
In Subsection~\ref{sec:metastability}-\ref{sec:metastable-convergence} we define various notions of metastability, oscillation and uniform metastability for arbitrary sequences or nets in a metric space, and show how these notions relate to classical ones. 
Proposition~\ref{prop:metastable-Cauchy}, in particular, states that metastable convergence is equivalent to usual convergence (always in complete spaces).

In Section~\ref{sec:nets-in-structures}, we begin the discussion of sequences and nets from the viewpoint of model theory. 
We exhibit a (continuous) first-order axiomatization of the property of metastable convergence with a given uniform rate. 
Subsequently, we state and prove the Uniform Metastability Principle
and a useful corollary thereof (Proposition~\ref{prop:U-metastab} and Corollary~\ref{cor:UMP-paramd}).

In Section~\ref{sec:Loeb}, we study finite measure structures from a model-theoretic viewpoint.
We introduce a certain class of metric structures, which we call Loeb structures due to their close relation to Loeb probability spaces from nonstandard analysis~\cite{Loeb:1975}.%
\footnote{This manuscript makes no use of results or notions from nonstandard analysis.}
We start by defining the notion of pre-Loeb structure.
Roughly speaking, a pre-Loeb structure is a metric structure that satisfies all the first-order properties of probability spaces $(\Omega,\cA,\PP)$, where $\PP$ is a probability measure on a Boolean algebra $\cA$ of subsets of~$\Omega$.  
Loeb structures are then defined as pre-Loeb structures that satisfy a saturation hypothesis. 
By basic model theory (Section~\ref{sec:metric-structures}-\ref{sec:saturation}), every metric structure can be extended naturally to a saturated structure.
As a by-product of saturation, Loeb structures possess the infinitary properties (i.e., countable additivity) of \emph{bona fide} classical probability spaces (Proposition~\ref{prop:Loeb-measure}).%
\footnote{More precisely, every Loeb structure induces a classical probability space $(\Omega,\cA_{\mathrm{L}},\PP\!_{\mathrm{L}})$ on the same underlying sample space~$\Omega$  (Proposition~\ref{prop:Loeb-measure}).}

In Section~\ref{sec:prob-struct}, we study integration from a model-theoretic viewpoint.
We begin by introducing the class of pre-integration structures. 
True to their name, these are structures satisfying the first-order axioms of the space $\LO$ of essentially bounded real functions on, say, a probability space $(\Omega,\cA,\PP)$, endowed with the integration linear functional $I : \LO\to\RR$ mapping  $f$ to $\int_{\Omega}f(\omega)\dd\PP(\omega)$. 
(In particular, a pre-integration structure is a pre-Loeb structure.) 
Of course, most models of those axioms will not correspond to \emph{bona fide}, countably additive probability spaces, nor will $I$ correspond to the operation of integration with respect to a measure in the usual sense.  
Nevertheless, a Riesz Representation Theorem  (Theorem~\ref{thm:Riesz}) for integration structures (that is, for saturated pre-integration structures), asserts that the space $\LO$ (whose elements correspond to bounded functions on the induced Loeb probability space) is endowed with a positive functional extending~$I$ and given by the usual operation of integration with respect to Loeb measure.

In Section~\ref{sec:DCT}, we introduce directed pre-integration structures. 
In essence, these are pre-integration structures $(\Omega,\cA,\PP,\LO,I)$ further endowed, say, for simplicity, with a bounded sequence $\varphi = (\varphi_n : n\in\NN)$ of elements of $\LO$ (identified with bounded functions on~$\Omega$). 
Directed integration structures are defined as saturated directed pre-integration structures, as expected. 
By the Riesz Representation Theorem (Theorem~\ref{thm:Riesz}), under the saturation hypothesis, $I$~corresponds to a classical integration operator on~$\LO$; therefore, the usual proof of the Dominated Convergence Theorem applies verbatim (Proposition~\ref{prop:DCT}). 
Obviously, the conclusion of the Dominated Convergence Theorem continues to hold if an additional hypothesis of uniform metastable pointwise convergence is imposed on the family~$\varphi$; 
moreover, this hypothesis is axiomatizable. 
It follows from the Uniform Metastability Principle that, under the additional hypothesis of uniform metastable pointwise convergence, the conclusion of the Dominated Convergence Theorem must admit a strengthening to convergence with a certain metastable rate. 
This yields an immediate proof of Tao's Metastable Dominated Convergence Theorem (Corollary~\ref{prop:meta-int}).

In the Appendix to this manuscript we state the viewpoint that a property of metric spaces should be considered finitary when it is equivalent to the satisfaction of a collection of axioms in a first-order language for metric structures.

We conclude this introduction by mentioning that Tao has formulated a nonstandard-analysis version of his Metastable Dominated Convergence Theorem in a blog post on \emph{Walsh’s ergodic theorem, metastability, and external Cauchy convergence}~\cite{Tao:Metastability}. Tao's insightful post served as philosophical motivation for the model-theoretic perspective adopted in the current manuscript. 

We gratefully acknowledge support for this research from the National Science Foundation through grant DMS-1500615. 
We thank the administration and sponsors of the Banff International Research Station for providing an excellent research environment in June 2016 during the meeting of the focused research group on Topological Methods in Model Theory (16frg676). 
We thank Xavier Caicedo, Chris Eagle, and Frank Tall for valuable discussions and comments that helped refine some of the ideas and content of this manuscript.
We also wish to thank Ulrich Kohlenbach for sharing his knowledge of the origins and early literature on metastable convergence.

\section{A finitary formulation of convergence}
\label{sec:metastability}

\subsection{Motivation}
\label{sec:motivation}

Given $\epsilon\ge 0$, we say that a sequence $a=(a_n:n\in\NN)$ in some metric space~$X$ is \emph{$[\epsilon]$-Cauchy} if at least one of its tails $a_{\ge N}=(a_n:n\ge N)$ satisfies the inequality $\osc_{\ge N}(a) \coloneqq \sup_{m,n\ge N}\dd(a_m,a_n)\le\epsilon$.
The $[\epsilon]$-Cauchy condition is infinitary in the sense that $\osc_{\ge N}(a)$ depends on the values $a_n$ as $n$ varies over the infinite set $\{N,N+1,\dots\}$. 
Tao's criterion for ``metastable convergence''~\cite{Tao:2008} imposes small oscillation conditions only on \emph{finite} segments $a_{[N,N']} = (a_n:N\le n\le N')$ of~$a$. 
For every fixed choice of a strictly increasing function $F:\NN\to\NN$ we regard the collection $\eta=([N,F(N)]:N\in\NN)$ as a ``sampling'' of~$\NN$, one finite segment at a time; 
abusing nomenclature, we use the name \emph{sampling} to refer to either the collection $\eta$ or the function~$F$ defining it. 
In Tao's nomenclature, for $\epsilon>0$ and sampling~$F$, the sequence $a$ is \emph{$[\epsilon,F]$-metastable} if the inequality $\osc_{[N,F(N)]}(a)\coloneqq \sup_{N\le m,n\le F(N)}\dd(a_m,a_n)\le\epsilon$ holds for some~$N$. 
For fixed $\epsilon>0$ and sampling~$F$, ``\emph{$[\epsilon,F]$-metastable}'' is a weaker property than ``\emph{$[\epsilon]$-Cauchy}'' inasmuch as the former involves only the values of~$a$ on the subsets $a_{[N,F(N)]}$ of the tails $a_{\ge N}$. 
However, when $F$ varies over all samplings of~$\NN$, the conjunction of the corresponding properties of $[\epsilon,F]$-metastability of~$a$ implies that $a$ is $[\epsilon]$-Cauchy. 
This leads to Tao's characterization of convergence (i.e., of the Cauchy property in complete spaces) as $[\epsilon,F]$-metastability for all $\epsilon>0$ and all samplings~$F$ (Proposition~\ref{prop:metastable-Cauchy}).

The metastable characterization of convergence is still not quite finitary 
because the existential statement on $N\in\NN$ is infinitary. 
(We use the term ``finitary'' as a synonym for ``axiomatizable in first-order continuous logic'' per the Appendix to this manuscript.)
Tao's concept of \emph{metastability with a given rate} arises by restricting this existential statement to a bounded (finitary) one. 
After Tao, for fixed $\epsilon>0$ and sampling~$F$, we call $M\in\NN$ an \emph{(upper bound on the) rate of $[\epsilon,F]$-metastable convergence of~$a$} if $\osc_{[N,F(N)]}(a)\le\epsilon$ holds for some $N\le M$.
Any collection $\Mb = (M_{\epsilon,F})\subset\NN$ of natural numbers, one for each $\epsilon>0$ and sampling~$F$, is an \emph{(upper bound on the) rate of metastability of~$a$} if $M_{\epsilon,F}$ bounds the $[\epsilon,F]$-rate of metastable convergence of~$a$ for all $\epsilon,F$. 
Evidently, given $\epsilon,F$ and $M\in\NN$, the property ``$M$ is a rate of $[\epsilon,F]$-metastability for~$a$'' is finitary. 
 
Given arbitrary $\epsilon,F$ and any Cauchy sequence~$a$, it is clear that $a$ admits \emph{some} bound $\Mb$ on its rate of metastability---one may take $M_{\epsilon,F}$ to be any $N$ satisfying $\osc_{\ge N}(a)\le\epsilon$. 
However, no choice of $\Mb$ applies uniformly to \emph{all} Cauchy sequences.

In Subsection~\ref{sec:metastable-convergence} below, we define various notions of metastable convergence for nets in metric spaces. 
In particular, the notion of metastable convergence with a given rate (Definition~\ref{def:unif-metastab}) is the natural finitary notion of convergence of nets extending Tao's (for sequences).

\subsection{Metastable convergence of nets in metric spaces}
\label{sec:metastable-convergence}

Throughout this section we fix a nonempty directed set $(\cD,\le)$
(that is, $\le$ is a reflexive, antisymmetric and transitive binary relation on~$\cD$ such that every pair of elements has an upper bound).
We denote by $\cD_{\ge i}$ the final segment $\{j\in \cD : j\ge i\}\subset\cD$.

We recall that a $\cD$-net $\ab = (a_i:i\in \cD)$ in a topological space~$X$ \emph{converges to $b\in X$} if, for every neighborhood $B$ of $b$ in~$X$, there exists $i\in \cD$ such that $a_i\in B$ for all $j\ge i$. 
If $(X,\dd)$ is a metric space and $\epsilon\ge 0$, we will say that the $\cD$-net $\ab$ is \emph{$[\epsilon]$-Cauchy} if there exists $i\in \cD$ such that $\dd(a_j,a_{j'})\le\epsilon$ for all $j,j'\ge i$, and that $\ab$ is \emph{$(\epsilon)$-Cauchy} (or \emph{$\epsilon$-Cauchy}) if $\ab$ is $[\epsilon']$-Cauchy for all $\epsilon'>\epsilon$. 
The net $\ab$ is Cauchy in the usual sense when it is $0$-Cauchy;
in this case, $\ab$ converges to $b$ if for all $\epsilon>0$ there is $i\in \cD$ such that $\dd(a_i,b)\le \epsilon$ for all $j\ge i$. 
Every Cauchy net in a complete metric space $X$ converges to some (necessarily unique) element $b\in X$.
All metric spaces under consideration shall be complete, so we use the term ``convergent'' as a synonym of ``Cauchy''.

\begin{definition}
  A \emph{sampling} of the directed set~$(\cD,\le)$ is any collection $\eta = (\eta_i : i\in \cD)$ of finite subsets of~$\cD$ (indexed by $\cD$ itself), such that $\eta_i$ is a nonempty finite subset of~$\cD_{\ge i}$ for each $i\in \cD$.%
\footnote{The condition ``$\eta_i$ is nonempty'' is not logically necessary. 
We impose it for heuristic reasons only.}
The collection of all samplings of~$\cD$ will be denoted $\Smpl(\cD)$. 
\end{definition}

\begin{definition}\label{def:metastable}
Fix a metric space $(X,\dd)$ and directed set $(\cD,\le)$. 
For $\eta\in\Smpl(\cD)$ and $\epsilon>0$, a $\cD$-net $\ab = (a_i : i\in \cD)$ in~$X$ is:
\begin{itemize}
\item  \emph{strict $\epsilon,\eta$-metastable}, or \emph{$[\epsilon,\eta]$-metastable} (note the square brackets), if there exists $i\in \cD$ such that $\dd(a_j,a_{j'})\le\epsilon$ for all $j,j'\in\eta_i$.
Any such $i$ is a \emph{witness} of the (strict) $[\epsilon,\eta]$-metastability of~$\ab$.
\item \emph{lax $\epsilon,\eta$-metastable}, or \emph{$(\epsilon,\eta)$-metastable} (note the round parentheses), if $\ab$ is $[\epsilon',\eta]$-metastable for all $\epsilon'>\epsilon$, 
\item  \emph{$\eta$-metastable,} if $\ab$ is $(0,\eta)$-metastable,  
\item {$\epsilon$-metastable,} if $\ab$ is $(\epsilon,\eta')$-metastable for all $\eta'\in\Smpl(\cD)$, 
\item \emph{metastable} if $\ab$ is $0$-metastable.
\end{itemize}
\end{definition}
For fixed $\epsilon,\eta$, $(\epsilon,\eta)$-metastability (resp., $[\epsilon,\eta]$-metastability, $\epsilon$-metastability) implies $(\epsilon',\eta)$-metastability (resp., $[\epsilon',\eta]$-metastability, $\epsilon'$-metastability) for all $\epsilon'\ge\epsilon$. 
It is also clear that strict $[\epsilon,\eta]$-metastability implies lax $(\epsilon,\eta)$-metastability, that lax $(\epsilon',\eta)$-metastability for all $\epsilon'>\epsilon\ge 0$ implies lax $(\epsilon,\eta)$-metastability, and consequently that $\epsilon'$-metastability for all $\epsilon'>\epsilon\ge0$ implies $\epsilon$-metastability.
We also remark that any net $\ab$ is necessarily $C$-metastable if $C$ is an upper bound on the distances $\dd(a_i,a_j)$ ($i,j\in\cD$); 
this is the case, in particular, if $\dd(a_i,x_0)\le C/2$ for some fixed $x_0\in X$ and all $i\in\cD$.

\begin{definition}\label{def:osc}
Let $\ab = (a_i:i\in\cD)$ be an arbitrary bounded $\cD$-net on a metric space~$(X,\dd)$.

  For every $\eta\in\Smpl(\cD)$, the \emph{$\eta$-oscillation $\osc_{\eta}(\ab)$ of~$\ab$} is
  \begin{equation*}
      \osc_{\eta}(\ab) 
= \inf\{\epsilon\ge 0 : \text{$\ab$ is $[\epsilon,\eta]$-metastable}\}.
  \end{equation*}

 The \emph{oscillation $\osc(\ab)$ of~$\ab$} is
  \begin{equation*}
    \osc(\ab) = \sup\{\osc_{\eta}(\ab) : \eta\in\Smpl(\cD)\}.
  \end{equation*}

\end{definition}
Note that the notations ``$\osc$'', ``$\osc_{\eta}$'' fail to exhibit the dependence of $\osc(\ab)$, $\osc_{\eta}(\ab)$ on~$\cD$. 
However, when these notations are used, the directed set~$\cD$ will be fixed, precluding ambiguous interpretations. 

\begin{proposition} \label{lem:osc-metastab}
For any bounded $\cD$-net~$\ab$, $\eta\in\Smpl(\cD)$ and $\epsilon\ge 0$:
  \begin{enumerate}[\normalfont(1)]
  \item $\osc_{\eta}(\ab) \le \epsilon$ if and only if $\ab$ is $(\epsilon,\eta)$-metastable,
  \item $\osc_{\eta}(\ab) = \inf\{\epsilon\ge 0 : \text{$\osc_{\eta_i}(\ab)\le\epsilon$ for some $i\in\cD$}\}$,
  \item $\osc_{\eta}(\ab) 
= \min\{\epsilon\ge 0 : \text{$\ab$ is $(\epsilon,\eta)$-metastable}\}$,
  \item $\osc_{\eta}(\ab) 
= \inf\{\osc_{\eta_i}(\ab) : i\in\cD\}$,
  \item $\osc(\ab) 
= \min\{\epsilon\ge 0 : \text{$\ab$ is $\epsilon$-metastable}\}$,
  \item $\osc(\ab)
= \sup\{\epsilon\ge 0 : \text{for all $i\in\cD$ there exist $j,j'\ge i$ with $\dd(a_j,a_{j'})\ge\epsilon$}\}$,
  \item $\ab$ is $\epsilon$-Cauchy if and only if $\osc(\ab)\le \epsilon$, 
  \item $\ab$ is $\epsilon$-Cauchy if and only if $\ab$ is $\epsilon$-metastable.
  \end{enumerate}
\end{proposition}
It is customary to use~(6) above as the definition of~$\osc(\ab)$. 
\begin{proof}
  \begin{enumerate}
  \item Let $r=\osc_{\eta}(\ab)$. 
For any $t>r$ the definition of $\osc_{\eta}$ implies that $\ab$ is $[s,\eta]$-metastable for some $s\in(r,t)$, hence also $[t,\eta]$-metastable; therefore, $\ab$ is $(r,\eta)$-metastable, hence \emph{a fortiori} $(\epsilon,\eta)$-metastable for all $\epsilon\ge r$. 
Conversely, if $\ab$ is $(\epsilon,\eta)$-metastable, then it is $[\epsilon',\eta]$-metastable for all $\epsilon'\ge\epsilon$, so the  definition of $\osc_{\eta}$ implies $r \le\epsilon$. 
  \item Clearly, $\ab$ is $[\epsilon,\eta]$-metastable iff $\osc_{\eta_i}(\ab)\le\epsilon$ for some $i\in\cD$. 
  \item This follows immediately from~(1).
  \item Let $r=\osc_{\eta}(\ab)$ and $s = \inf\{\osc_{\eta_i}(\ab) : i\in\cD\}$. 
By part~(2), we have $r\le s$. 
Conversely, if $\epsilon\ge\osc_{\eta_i}(\ab)$ for some $i\in\cD$, then $\epsilon\ge s$, hence $s\le r$.
\item Clearly, the set $\{\epsilon\ge 0 :\text{$\ab$ is $\epsilon$-metastable}\}$ has a least element, say~$s$, since $\epsilon'$-metastability for all $\epsilon'>\epsilon$ is equivalent to $\epsilon$-metastability.
Let $r = \osc(\ab)$.
If $\ab$ is $\epsilon$-metastable, then it is $(\epsilon,\eta)$-metastable for all~$\eta\in\Smpl(\cD)$, hence $\osc_{\eta}(\ab)\le\epsilon$ by~(1), so $r\le s$ by~(3).
Conversely, for all $\eta\in\Smpl(\cD)$, $\ab$ is $(r,\eta)$-metastable by~(1) and the definition of~$\osc$, and thus $\ab$ is $r$-metastable, so $r\ge s$.
\item Let $r = \osc(\ab)$ and 
\[
s = \sup\{\epsilon\ge 0 : \text{for all $i\in\cD$ there exist $j,j'\ge i$ with $\dd(a_j,a_{j'})\ge\epsilon$}\}.
\]
If $0\le t<s$, then for all $i\in\cD$ there exists $\eta_i=\{j,j'\}\subset\cD_{\ge i}$ such that $\dd(a_j,a_{j'})>t$, hence $\ab$ is not $[t,\eta]$-metastable for $\eta = (\eta_i:i\in\cD)$. 
It follows that $t\le \osc_{\eta}(\ab)\le r$. 
As this holds for all positive $t<s$, we have $s\le r$. 
Conversely, if $0\le t<r$, then $\osc_{\eta}(\ab) > t$ for some $\eta\in\Smpl(\cD)$. 
By~(1), $\ab$ is not $(t,\eta)$-metastable. 
Hence, there is $t'>t$ such that for all $i\in\cD$ there exist $j,j'\in\eta_i\subset\cD_{\ge i}$ with $\dd(a_j,a_{j'}) > t' > t$; hence, $t\le s$. 
As this holds for all $t<r$, we have $r\le s$.
\item Let $r=\osc(\ab)$. 
It follows from (6) that $\ab$ is $r$-Cauchy, thus also $\epsilon$-Cauchy for all $\epsilon\ge r$, 
Conversely, if $0\le\epsilon < r$, then by~(6) there exists $\epsilon'\in(\epsilon,r)$ with the property that for all $i\in\cD$ there exist $j,j'\ge i$ with $\dd(a_j,a_{j'})\ge\epsilon'$. 
Then $\ab$ is not $\epsilon$-Cauchy.\qedhere
\item This follows immediately from~(5) and~(7), plus the remarks following Definition~\ref{def:metastable}.
\end{enumerate}
\end{proof}

\begin{proposition}[Metastable characterization of the Cauchy property]\label{prop:metastable-Cauchy}
  A net in a metric space~$X$ is Cauchy if and only if it is metastable.
\end{proposition}
\begin{proof}
This is the particular case $\epsilon=0$ of~(8) in Proposition~\ref{lem:osc-metastab}.
\end{proof}

\begin{remark}
  Propositions~\ref{lem:osc-metastab} and~\ref{prop:metastable-Cauchy} remain true if we consider only samplings $\eta$ such that $\eta_i$ consists of no more than two elements of~$\cD_{\ge i}$ for all $i\in\cD$. 
This more restrictive definition of sampling could be used in all further developments without any essential changes.
\end{remark}

\begin{definition}\label{def:unif-metastab}
  Fix a directed set~$\cD$.  
The collection of all finite subsets of $\cD$ will be denoted $\PfinD$.  
Let $\ab = (a_i : i\in\cD)$ be a $\cD$-net in a metric space~$X$.

Given $\epsilon>0$ and $\eta\in\Smpl(\cD)$, a set $E\in\PfinD$ is called \emph{a (bound on the) rate of $[\epsilon,\eta]$-metastability of~$\ab$}
if there exists a witness $i\in E$ of the $[\epsilon,\eta]$-metastability of~$\ab$. 
(No sequence $\ab$ has this property if $E$ is empty.)

For $\epsilon>0$, a collection $\Eb = (E_{\eta} : \eta\in\Smpl(\cD))\subset\PfinD$ is called a \emph{(bound on the) rate of $[\epsilon]$-metastability of~$\ab$}
if $E_{\eta}$ is a rate of $[\epsilon,\eta]$-metastability of~$\ab$ for all $\eta\in\Smpl(\cD)$.

For $r\ge 0$ and $\eta\in\Smpl(\cD)$, a collection $\Eb = (E_{\epsilon} : \epsilon>r)$ in $\PfinD$ is called a \emph{(bound on the) rate of $(r,\eta)$-metastability of~$\ab$}
if $E_{\epsilon}$ is a rate of $[\epsilon,\eta]$-metastability of~$\ab$ for all $\epsilon>r$. 

For $r\ge 0$, a collection $\Eb = (E_{\epsilon,\eta} : \eta\in\Smpl(\cD),\epsilon>r)$ in $\PfinD$ 
is called \emph{a (bound on the) rate of $r$-metastability of~$\ab$},
if $E_{\cdot,\eta}$ is a rate of $(r,\eta)$-metastability of~$\ab$ for all $\eta\in\Smpl(\cD)$, where $E_{\cdot,\eta} = (E_{\epsilon,\eta} : \epsilon>r)$. 
When $r=0$, we say simply that $\Eb$ is a (bound on the) rate of metastability of~$\ab$.

If $\cC$ is any collection of $\cD$-nets in~$X$, we say that:
\begin{itemize}
\item \emph{$E$ is a uniform (bound on the) rate of $[\epsilon,\eta]$-metastability for~$\cC$,} or \emph{$\cC$ is $E$-uniformly $[\epsilon,\eta]$-metastable,}
if $\ab$ is $E$-uniformly $[\epsilon,\eta]$-metastable for all $\ab\in\cC$;
\item \emph{$\Eb$ is a uniform (bound on the) rate of $[\epsilon]$-metastability for~$\cC$,} or \emph{$\cC$ is $\Eb$-uniformly $[\epsilon]$-metastable,} 
if $\ab$ is $\Eb$-uniformly $[\epsilon]$-metastable for all $\ab\in\cC$;
\item \emph{$\Eb$ is a uniform (bound on the) rate of $r$-metastability for~$\cC$,} or \emph{$\cC$ is $\Eb$-uniformly $r$-metastable,} 
if $\ab$ is $\Eb$-uniformly $r$-metastable for all $\ab\in\cC$.
(When $r=0$, we usually omit it.)
\end{itemize}
\end{definition}

\begin{remark}\label{rem:unif-metastab-Cauchy}
We show that the concept of uniform metastability generalizes that of uniform convergence. 
Let $\cC$ be any collection of $\cD$-nets.
Given $M\in\cD$ and $\epsilon>0$, the collection $\cC$ is \emph{$M$-uniformly $[\epsilon]$-Cauchy} if $\osc_{\ge M}(\ab) \coloneqq \sup_{j,j'\ge M}\dd(a_j,a_{j'}) \le \epsilon$ for each~$\epsilon>0$ and $\ab\in\cC$.
Given $\Mb = (M_\epsilon : \epsilon>0)\subset\cD$, the collection~$\cC$ is \emph{$\Mb$-uniformly Cauchy} if every $\ab\in\cC$ is $M_{\epsilon}$-uniformly $[\epsilon]$-Cauchy for each~$\epsilon>0$.
Corresponding to $\Mb$ there is a rate of metastability $\Eb = (E_{\epsilon,\eta})$ defined by $E_{\epsilon,\eta} = \{M_{\epsilon}\}\in\Pfin(D)$. 
Under this identification, the collection~$\cC$ is $\Mb$-uniformly Cauchy if and only if it is $\Eb$-metastable. 

Clearly, metastability rates $\Eb$ obtained from a collection $\Mb$ as above are very special. 
The following example, due to Avigad \emph{et al.}~\cite{Avigad-Dean-Rute:2012}, exhibits a family of uniformly metastable sequences that are not uniformly convergent.
Every monotonically increasing sequence in~$[0,1]$ is convergent. 
Let $\cC$ be the collection of all such sequences. 
Clearly, there is no rate $\Mb$ such that all sequences $\ab\in\cC$ are $\Mb$-uniformly Cauchy; in fact, for $\epsilon\in(0,1)$, the sequence $\ab$ with $a_m=0$ for $m\le M_{\epsilon}$ and $a_m=1$ for $m>M_{\epsilon}$ satisfies $\osc_{\ge E_{\epsilon}}(\ab) = 1 > \epsilon$.
On the other hand, for any $\epsilon>0$ and a sampling of~$\NN$ given as a strictly increasing function $F : \NN\to\NN$ (as per Section~\ref{sec:metastability}.\ref{sec:motivation}), let $E_{\epsilon,F} = \{m\in\NN : m\le F^{(k)}(0)\}$ where $k = \lceil\epsilon^{-1}\rceil$ is the smallest integer no smaller than~$\epsilon^{-1}$ and $F^{(k)}(0)$ is the $k$-fold iterate of~$F$ applied to~$0$.
Since $k\epsilon\ge 1$, at least one of the $k$ differences $a_{F^{(j+1)}(0)}-a_{F^{(j)}(0)}$ ($j=0,1,\dots,k-1$) must not exceed~$\epsilon$ whenever $\ab\in\cC$, hence $\cC$ is $E_{\epsilon,F}$-uniformly $[\epsilon]$-metastable. 
The collection $\Eb = (E_{\epsilon,F})$ is a uniform metastability rate for~$\cC$. 
\end{remark}

The concept of uniform metastability is crucial to our applications.
As discussed in Remark~\ref{rem:unif-metastab-Cauchy} above, uniform metastability is a proper generalization of the classical notion of uniform convergence. 
Moreover, uniform metastability with a given rate is axiomatizable in the logic of metric structures (Proposition~\ref{prop:metastab-1st-ord} below). 
This allows for powerful applications of model theory, particularly of compactness (e.g., the Uniform Metastability Principle, Proposition~\ref{prop:U-metastab}).
We believe that many convergence results in analysis follow from hypotheses captured by the semantics of first-order logic for metric structures; consequently, such results ought to admit refinements to convergence with a metastability rate---and moreover the rate ought to be universal, i.e., independent of the structure to which the theorem is applied. 
Tao's Metastable Dominated Convergence Theorem, as stated and proved in Section~\ref{sec:DCT} below, is but one example of this philosophy.

\section{Convergence of nets in metric structures and uniform metastability}
\label{sec:nets-in-structures}

Throughout Sections~\ref{sec:nets-in-structures} to~\ref{sec:DCT}, we will assume that the reader is familiar with the material presented in Section~\ref{sec:metric-structures}. 
In particular, we assume familiarity with the notions of Henson metric structure, signature, positive bounded formula, approximate satisfaction, the Compactness Theorem~\ref{T:compactness}, and saturated structures. 
We will deal with multi-sorted structures that contain discrete sorts alongside nondiscrete ones. 
Recall that $\mathbb{R}$, equipped with its ordered field structure and a constant for each rational number, occurs tacitly as a sort of every metric structure, and that discrete predicates in a structure are seen as $\{0,1\}$-valued functions. 
If $(M,d,a)$ is a pointed metric space and $C\ge0$, the set $\{\,x\in M\mid d(x,a)\le C\,\}$ will be denoted $M^{[C]}$.

Throughout the paper,  $L$ will be a many-sorted signature with sorts $(\BS_i : i<\alpha)$ (for some ordinal~$\alpha>1$), and $\BS_0$ will be the special sort designated for $\RR$. 
For notational convenience, we will identify $\BS_0$ with $\RR$. 

Hereafter, $(\cD,\le)$ will denote a directed set with least element~$j_0$. 
We will regard $(\cD,\le)$ as a discrete metric structure with sorts $\cD$ and $\mathbb{R}$, and the point $j_0$ will be regarded as the anchor for the sort $\cD$.
We will refer to $(\cD,\le,j_0)$ as a \emph{pointed directed set}.

\begin{definition}[Directed structure]\label{def:directed-structure}
Fix a pointed directed set $(\cD,\le,j_0)$. 
Let $L$ be a many-sorted signature with sorts $(\BS_i : i<\alpha)$ (for some ordinal~$\alpha>1$), where $\BS_0=\RR$.
The sort $\BS_1$ will be called the \emph{directed sort}; it will be denoted by~$\DD$ henceforth. 
Let $L$ include a symbol $\llbracket\cdot\le\cdot\rrbracket$ for a function $\DD\times\DD \to \{0,1\}$.
Let $L$ also include distinct constant symbols $\mathtt{c}_j$, one for each $j\in \cD$.
In addition to the function symbols for the sort metrics and the operations on~$\RR$ plus constants for rational numbers, $L$~may include other function and constant symbols, as well as any other sorts than those already mentioned.
A \emph{$\cD$-directed structure} is any metric $L$-structure $\cM$ such that
\begin{itemize}
\item the sort $\DD^{\cM}$ is discrete, 
\item the interpretation of $\llbracket\cdot\le\cdot\rrbracket$ induces an order on~$\DD^{\cM}$, denoted~$\le$ (by a slight abuse of notation), such that  $(\DD^{\cM},\le,\mathtt{c}_{j_0}^{\cM})$ is an anchored directed set and, for all $i,j\in\DD^{\cM}$, $i\le j$ holds precisely when $\llbracket i\le j\rrbracket = 1$,
\item the map $\cD\to\DD^{\cM}$  defined by $j\mapsto \mathtt{c}_j^{\cM}$ is an order-preserving injection;
\end{itemize}
\end{definition}

It should be clear that the class of $\cD$-directed structures is axiomatizable in the semantics of approximate satisfaction of positive bounded $L$-formulas. 
We note that the embedding $\cD\hookrightarrow\DD^{\cM}$ is usually not surjective.

In order to discuss nets in $\cD$-directed structures, we need to extend the language $L$ with a function symbol $\mathtt{s} : \DD\to\BS_{\iota}$ interpreted as a function from the directed sort~$\DD$ into some other sort~$\BS_{\iota}$ of~$L$. 
In this context, $T$ shall denote any fixed uniform $L[\mathtt{s}]$-theory extending the theory of $\cD$-directed $L$-structures. 
(Extensions of a Henson language are discussed in Section~\ref{sec:metric-structures}.\ref{sec:extensions}.)

\begin{definition}
\label{D: internal net}
Fix a $\cD$-directed $L$-structure~$\cM$, a new function symbol $\mathtt{s} : \DD\to\BS_{\iota}$, and a uniform $L[\mathtt{s}]$-theory~$T$. 
An \emph{(external) net} in~$\cM$ is any $\cD$-net $\ssb = (s_i : i\in\cD)$ taking values in some sort of~$L$. 
An \emph{internal net (modulo~$T$)} is a function $s : \DD^{\cM}\to\BS^{\cM}_{\iota}$ such that $(\cM,s)$ is a model of~$T$. 
\end{definition}
Note that an internal net $s$ yields an external $\cD$-net $\ssb = (s_i : i\in\cD)$, letting $s_i = s(i)$ for $i\in\cD$.

For any rational $\epsilon>0$, $\eta\in\Smpl(\cD)$ and internal net $s = \mathtt{s}^{\cM}$, the $\cD$-net $\ssb$ is strict $[\epsilon,\eta]$-metastable if and only if there exists $i\in \cD$ such that
\begin{equation*}
\cM\appmod \xi^{\eta}_i(\epsilon)\qquad\text{(equivalently,\quad $\cM\models \xi^{\eta}_i(\epsilon)$)},
\end{equation*}
where $\xi_i^{\eta}$ is the positive bounded $L$-formula
\begin{equation}\label{eq:xi}
  \xi^{\eta}_i(\mathtt{t}) : 
\bigwedge_{j,j'\in\eta_i} \big(\dd_{\iota}(\mathtt{s}(j),\mathtt{s}(j')) \le \mathtt{t}\big),
\end{equation}
as follows from the semantics of approximate satisfaction and part~(3) of Proposition~\ref{lem:osc-metastab}.%
\footnote{For notational convenience, in~\eqref{eq:xi} we write $\mathtt{s}_k$ for $\mathtt{s}(\mathtt{c}_k)$, where $\mathtt{c}_k$ is the constant denoting the element $k\in\cD$. 
Similar simplifications will be usually made without comment whenever the intended strict syntax is otherwise clear.}
We emphasize that $\xi_i^{\eta}$ is a \emph{bona fide} $L$-formula since it is a finite conjunction of atomic formulas. 
Nevertheless, the property ``\emph{$\ssb$ is $[\epsilon,\eta]$-metastable}'' is not $L$-axiomatizable since the asserted existence of the witness $i\in\cD$ amounts to an infinite disjunction of formulas when $\cD$ is itself infinite.

For fixed $\eta\in\Smpl(\cD)$ and~$\epsilon > 0$, call $\ssb$ an \emph{$[\epsilon,\eta]$-unstable} $\cD$-net if, for all $i\in\cD$, there exist $j,j'\in\eta_i$ with $\dd_{\iota}(s_j,s_{j'})\ge\epsilon$.
For rational $\epsilon\ge 0$, the assertion ``\emph{$\ssb$ is $[\epsilon,\eta]$-unstable}'' is equivalent to
\begin{equation*}
\cM \appmod \wneg\xi_i^{\eta}(\epsilon)\qquad\text{for all $i\in\cD$,\quad (alternatively,\quad $\cM \models \wneg\xi_i^{\eta}(\epsilon)$)}
\end{equation*}
where 
\begin{equation}\label{eq:notxi}
  \wneg\xi_i^{\eta}(\mathtt{t}) : \bigvee_{j,j'\in\eta_i}
\big(\dd_{\iota}(\mathtt{s}(j),\mathtt{s}(j')) \ge \mathtt{t}\big).
\end{equation}
Note that $[\epsilon,\eta]$-metastability is consistent with $[\epsilon',\eta]$-instability precisely when $\epsilon'\le\epsilon$. 
In contrast to the property of strict metastability, the property of $[\epsilon,\eta]$-instability for given $\eta\in\cD$ and $\epsilon\ge 0$ is axiomatized by the collection
\begin{equation*}
\{\wneg\xi^{\eta}_i(\epsilon'): \text{rational $\epsilon'<\epsilon$ and $i\in\cD$}\}
\end{equation*}
of $L$-sentences, in the semantics of approximate (or discrete) satisfaction.

Recall that if $\mathcal{M}$ is a structure, the complete theory $\Th(\cM)$ of $\mathcal{M}$ is the set of all sentences satisfied by $\mathcal{M}$.
It should be clear from the preceding observations that if $\ssb$ is a $\cD$-net  underlying an internal net~$s$, then classical properties of  $\ssb$, including: \emph{``$\ssb$ is $[\epsilon,\eta]$-metastable''}, \emph{``$\ssb$ is $(\epsilon,\eta)$-metastable''}, \emph{``$\ssb$ is $\eta$-metastable''}, ``$r\le \osc(\ssb)\le s$'', etc., depend only on the complete theory $T$ of~$\cM$, though some of these properties are not axiomatizable.

\begin{proposition}[Axiomatizability of uniform metastability]\label{prop:metastab-1st-ord}
Fix a pointed directed set~$(\cD,\le,j_0)$, a language $L$ for $\cD$-directed structures, a new function symbol $\mathtt{s} : \DD\to\BS_{\iota}$ and a uniform $L[\mathtt{s}]$-theory~$T$. 
For fixed reals $\epsilon > r\ge 0$ and sampling $\eta\in\Smpl(\cD)$, the following properties of the $\cD$-net~$\ssb$ induced by an internal net $s = \mathtt{s}^{\cM}$ are $L[\mathtt{s}]$-axiomatizable (modulo~$T$):
\begin{itemize}
\item ``\emph{$E$ is a rate of $[\epsilon,\eta]$-metastability for~$\ssb$}'', for any fixed $E\in\PfinD$;
\item ``\emph{$\Eb$ is a rate of $(r,\eta)$-metastability for~$\ssb$}'', for any family $\Eb = (E_{\epsilon}: \epsilon>r)$ in $\PfinD$; 
\item ``\emph{$E_{\bullet}$ is a rate of $r$-metastability for~$\ssb$}'', for any family $\Eb = (E_{\epsilon,\eta}: \eta\in\Smpl(\cD), \epsilon>r)$ in $\PfinD$.
\end{itemize}
Hence, each of the preceding properties characterizes an axiomatizable subclass of models of~$T$. 
In fact, in each case the collection of axioms is independent of $T$ and $L$: 
It depends only on the underlying directed set~$\cD$. 
\end{proposition}
\begin{proof}
For $E\in\PfinD$ and $\eta\in\cD$, let
\begin{equation*}
  \xi^{\eta}_E = \bigvee_{i\in E}\xi^{\eta}_i\quad
\text{with $\xi^{\eta}_i$ given by~\eqref{eq:xi}.}
\end{equation*}
The first stated property is axiomatized by the formulas $\xi^{\eta}_{E}(\epsilon')$ for rational $\epsilon' > \epsilon$, the second by all formulas $\xi^{\eta}_{E_{\epsilon}}\!(\epsilon')$ for rational $\epsilon' > \epsilon$, and the third by all formulas $\xi^{\eta}_{E_{\epsilon,\eta}}\!(\epsilon')$ for  $\epsilon > r$, rational $\epsilon' > \epsilon$, and all $\eta\in\Smpl(\cD)$.
\end{proof}

\begin{proposition}[Uniform Metastability Principle]
\label{prop:U-metastab}
Fix a pointed directed set~$(\cD,\le,j_0)$, a language $L$ for $\cD$-directed structures, a new function symbol $\mathtt{s} : \DD\to\BS_{\iota}$ and a uniform $L[\mathtt{s}]$-theory~$T$. 
Let $\eta\in\Smpl(\cD)$ and $r\ge 0$. 
If $\ssb$ is $(r,\eta)$-metastable whenever $(\cM,s)\appmod T$, then there exists a collection $\Eb = (E_{\epsilon} : \epsilon>r)$ in~$\PfinD$ such that $\Eb$ is a rate of $(r,\eta)$-metastability for~$\ssb$, uniformly over all models $(\cM,s)$ of~$T$.
\end{proposition}
Loosely speaking, if all internal sequences are metastable, then they are uniformly me\-tas\-table. 
(Of course the notion of internal sequence is in reference to a fixed uniform theory~$T$.)
On the other hand, the uniform rate $\Eb$ is certainly dependent on~$T$. 
The uniform theory $T$ implies an upper bound $C\ge 0$ such that $s(j)\in(\BS_{\iota})^{[C]}$ for all $j\in\DD^{\cM}$ whenever $(\cM,s)\appmod T$.

\begin{proof}
It is enough to show that a bound $E_{\epsilon}$ on the rate of $[\epsilon,\eta]$-metastability exists for each rational $\epsilon>r$, uniformly for all $\ssb$ arising from $\mathtt{s}^{\cM}$ for arbitrary~$\cM\appmod T$. 
Assume no such $E_{\epsilon}$ exists for some $\epsilon > r$.
With the notation of Proposition~\ref{prop:metastab-1st-ord}, given $E\in\PfinD$ there exists $(\cM,s)\appmod T$ and some rational $\epsilon'>\epsilon$ such that $(\cM,s)\not\appmod \xi_E^{\eta}(\epsilon')$. 
Consequently, $(\cM,s)\appmod \wneg\xi_E^{\eta}(\epsilon) : \bigwedge_{i\in E} \wneg\xi_i^{\eta}(\epsilon)$, with $\wneg\xi_i^{\eta}$ as in~(\ref{eq:notxi}). 
Thus, the collection 
  \begin{equation*}
X_{\epsilon}^{\eta} = \{\wneg\xi_i^{\eta}(\epsilon) : i\in\cD\}
\end{equation*}
is finitely jointly satisfiable with~$T$. 
By the Compactness Theorem~\ref{T:compactness}, there exists a model $(\cM,s)$ of  $T\cup X_{\epsilon}^{\eta}$. 
On the one hand, $\ssb$ is $[\epsilon,\eta]$-unstable because $\cM\appmod X_{\epsilon}^{\eta}$; 
on the other hand, $\ssb$ is  $(r,\eta)$-metastable by hypothesis, since $(\cM,s)\appmod T$. 
This is a contradiction since $r<\epsilon$. 
\end{proof}

The Uniform Metastability Principle may be formulated as the following dichotomy: 
For a fixed uniform $L[\mathtt{s}]$-theory~$T$, real $r\ge 0$ and $\eta\in\Smpl(\cD)$,
\begin{itemize}
\item Either: There exists a rate $\Eb = (E_{\epsilon} : \epsilon>r)$ of $(r,\eta)$-metastability for $\ssb$ valid uniformly for all models $(\cM,s)$ of~$T$;
\item Or else: There exists a model $(\cM,s)$ of~$T$ such that $\ssb$ fails to be $(r,\eta)$-metastable.
\end{itemize}

Below, we state a form of the Uniform Metastability Principle that applies when many internal nets are realized in the same structure. 
Fix a pointed directed set~$(\cD,\le,j_0)$ and a signature~$L$ for $\cD$-directed structures. 
Extend $L$ with function symbols $\mathtt{s} : \DD\to\BS_{\iota}$ $\boldsymbol{\sigma} : \BA\times\DD\to\BS_{\iota}$ where $\BA$ is of the form $\BS_{i_1}\times\dots\times\BS_{i_n}$ (a Cartesian product of sorts of~$L$) and $\BS_{\iota}$ is any sort of~$L$. 
We will treat $\BA$ itself as a sort, writing $\BA^{\cM}$ for $\BS_{i_1}^{\cM}\times\dots\times\BS_{i_n}^{\cM}$ in any $L$-structure~$\cM$.
If $(\cM,\sigma,s)$ is an $L[\boldsymbol{\sigma},\mathtt{s}]$-structure, we will say that~$s$ \emph{admits $\sigma$-parameters of size $\le C$} if there exists $\overline{a}\in(\BA^{\cM})^{[C]}$ such that $s(j) = \sigma(\overline{a},j)$ for all $j\in\DD^{\cM}$. 
The slightly more general notion that $s$~\emph{approximately admits $\sigma$-parameters of size $\lesssim C$} means that, for all $\epsilon > 0$, there exists $\overline{a}\in(\BA^{\cM})^{[C+\epsilon]}$ such that $\dd(s(j),\sigma(\overline{a},j))\le \epsilon$ for all $j\in\DD^{\cM}$. 
If $C$ is rational, the latter property of $L[\boldsymbol{\sigma},\mathtt{s}]$-structures~$(\cM,\sigma,s)$ is captured by a single axiom
\begin{equation}\label{eq:axiom-params-C}
\upsilon_C : 
(\exists_C\overline{\mathtt{a}})(\forall j)
(\dd(\mathtt{s}(j),\boldsymbol{\sigma}(\overline{\mathtt{a}},j))\le 0).
\end{equation}
In general, the property is axiomatized by the scheme $\{\upsilon_D : \text{rational $D>C$}\}$.

\begin{corollary}[Uniform metastability of parametrized sequences]\label{cor:UMP-paramd}
Fix a pointed directed set~$(\cD,\le,j_0)$ and a signature~$L$ for $\cD$-directed structures. 
Extend $L$ with function symbols $\mathtt{s} : \DD\to\BS_{\iota}$ and $\boldsymbol{\sigma} : \BA\times\DD\to\BS_{\iota}$ where $\BA$ is of the form $\BS_{i_1}\times\dots\times\BS_{i_n}$ (a Cartesian product of sorts of~$L$) and $\BS_{\iota}$ is any sort of~$L$.
Fix a rational $C\ge 0$ and a uniform $L[\boldsymbol{\sigma},\mathtt{s}]$-theory~$T$ such that $s$ approximately admits $\sigma$-parameters of size~$\lesssim C$ for every model $(\cM,\sigma,s)\appmod T$ (i.e., $T\appmod\upsilon_D$ for all rational $D>C$ with $\upsilon_D$ defined in~\eqref{eq:axiom-params-C} above).
For some $r\ge 0$ and $\eta\in\Smpl(\cD)$, assume that the external $\cD$-net $\ssb$ is $(r,\eta)$-metastable whenever $(\cM,\sigma,s)\appmod T$. 
Then there exists a collection $\Eb^r = (E_{\epsilon} : \epsilon>r)$ in~$\PfinD$ such that $\ssb$ is $\Eb$-uniformly $(r,\eta)$-metastable whenever $(\cM,\sigma,s)\appmod T$. 
Moreover, $\Eb$ depends only on~$T$.
\end{corollary}
As a simple application of saturation, the notions \emph{``$s$ admits $\sigma$-parameters of size $\le C$''} and \emph{``$s$ approximately admits $\sigma$-parameters of size $\lesssim C$''} are seen to be equivalent in any $\omp$-saturated $L[\boldsymbol{\sigma},\mathtt{s}]$-structure~$(\cM,\sigma,s)$.
Since every model of $T$ admits an $\omega^+$-saturated elementary extension, the hypotheses of Corollary~\ref{cor:UMP-paramd} may be relaxed to state that the external $\cD$-nets $\ssb$ are $(r,\eta)$-metastable whenever $s$ admits (exact) parameters of size~$\le C$ via~$\sigma$ (without changing the conclusion).

\begin{proof}
By the Uniform Metastability Principle~\ref{prop:U-metastab}, if the collection $\Eb$ did not exist, there would be a model $(\cM,\sigma,s)$ of~$T$ such that $\ssb$ is not $(r,\eta)$-metastable, contradicting the hypotheses. 
Hence, the metastability rate~$\Eb$ must exist, and depends only on~$T$.
(Note that $\Eb$ implicitly depends on~$C$.)
\end{proof}

\section{Loeb structures}
\label{sec:Loeb}

In this section, $L$ will be a many-sorted signature with sorts $(\BS_i : i<\alpha)$ (for some ordinal~$\alpha$), where $\BS_0 = \RR$ (as convened), $\BS_1 = \Omega$, and $\BS_2 = \cA$. 
The sort $\Omega$ will be interpreted as a set with anchor point~$\omega_0$, and $ \cA$ will be interpreted as a Boolean algebra with its least element as anchor point.

Recall that, by default, $L$ comes equipped with the following symbols:
\begin{itemize}
\item a binary function symbol $\dd_i:\BS_i\times\BS_i\to\RR$ ($i<\alpha$) for the metric of $\BS_i$ (we may dispense with $\dd_0$ since $\dd_0(x,y) = |y-x|$ for $x,y\in\RR$, and we shall denote $\dd_1$ by $\dd_{\Omega}$, $\dd_2$ by $\dd_{\cA}$);
\item binary symbols $+_{\RR}$, $-_{\RR}$, $\cdot_{\RR}$ for the operations of addition, subtraction, multiplication of~$\RR$ and a monadic function symbol $\absR{\cdot}$ for the absolute value of~$\RR$;
\item binary symbols $\wedge_{\RR}$ (``minimum'') and $\vee_{\RR}$ (``maximum'') for the lattice operations of~$\RR$;
\item an $\RR$-valued constant $\mathtt{c}_r$ for each rational number $r\in\RR$.
\end{itemize}
We will require that, in addition to these symbols, the signature $L$ include the following:
\begin{itemize}
\item a symbol for the \emph{measure}:
\begin{equation*}
  \boldsymbol{\mu} : \cA \to \RR;
\end{equation*}
\item function symbols
\[
\cup: \cA\times\cA \to \cA, \qquad \cap:\cA\times\cA\to \cA,\qquad\compl{}:\cA \to \cA
\]
for the union (join), intersection (meet) and complementation operations of the interpretation of~$\cA$;
\item a symbol $\llbracket \cdot \in \cdot \rrbracket:\Omega\times\cA \to \{0,1\}$;
\item $\cA$-valued constants $\varnothing$, $\mathtt{\Omega}$ for the zero (null element) and unity (universal element) of the interpretation of~$\cA$;
\item an $\Omega$-valued constant $\mathtt{c}_{\omega_0}$ for the distinguished element $\omega_0$ of the interpretation of $\Omega$.
\end{itemize}
The signature $L$ may include other function and constant symbols, as well as many other sorts than those already mentioned.

In order to simplify the notation, we normally omit the subscripts in the operations $+_{\RR}$, $-_{\RR}$, $\cdot_{\RR}$,  $\wedge_{\RR}$, $\vee_{\RR}$, and $\absR{\cdot}$  of~$\RR$. 
Also, for a rational number $r$, we denote the constant $\mathtt{c}_r$ simply as $r$.

If $\cM$ is a fixed $L$-structure,  we will denote the sorts of $\cM$ corresponding to $\cA$ and $\Omega$ also as $\cA$ and $\Omega$. 
Moreover, if the context allows it, we will identify the symbols of $L$ with their interpretation in $\cM$; thus, for instance, we denote $\cup^{\cM}$ and $\mathtt{\Omega}^{\cM}$ simply as $\cup$ and $\mathtt{\Omega}$ (respectively). 

\begin{definition}[Loeb structure]\label{def:Loeb-struct}
A signature $L$ as described above will henceforth be called a \emph{Loeb signature}. 
A \emph{pre-Loeb finite measure structure} is an $L$-structure~$\cM$ such that: 
\begin{itemize}
\item the metrics $\dd_{\Omega}$ on $\Omega$ and $\dd_{\cA}$ on $\cA$ are discrete;
\item $\llbracket\cdot\in\cdot\rrbracket$ is a $\{0,1\}$-valued function identifying ($\cA$, $\varnothing$, $\mathtt{\Omega}$, $\cup$, $\cap$, $\compl{\cdot}$) with an algebra of subsets of~$\Omega$, i.e., for all $A,B\in\cA$ and $\omega\in\Omega$:
  \begin{itemize}
  \item $\dd_{\cA}(A,B) = \sup_{x\in\Omega}\bigl|\llbracket x\in A\rrbracket - \llbracket x\in B\rrbracket\bigr|$;
  \item $\llbracket\omega\in\varnothing\rrbracket = 0$;
  \item $\llbracket\omega\in\mathtt{\Omega}\rrbracket = 1$;
  \item $\llbracket\omega\in A\cup B\rrbracket = \llbracket\omega\in A\rrbracket \vee \llbracket\omega\in B\rrbracket$;
  \item $\llbracket\omega\in A\cap B\rrbracket = \llbracket\omega\in A\rrbracket \wedge \llbracket\omega\in B\rrbracket$;
  \item $\llbracket\omega\in A\rrbracket + \llbracket\omega\in\compl{A}\rrbracket = 1$;
  \end{itemize}
\item the interpretation $\mu$ of $\boldsymbol{\mu}$ is a finitely additive measure on~$\cA$.
For some $C\ge 0$:
  \begin{itemize}
  \item $\mu(\varnothing) = 0$ and $0 \le \mu(A) \le \mu(\mathtt{\Omega}) \eqqcolon \|\mu\| \le C$ for all $A\in\cA$;
  \item $\mu(A\cup B) + \mu(A\cap B) = \mu(A) + \mu(B)$ for all $A,B\in\cA$.
  \end{itemize}  
\end{itemize}

In a \emph{pre-Loeb signed measure structure}~$\cM$, the last axioms are replaced by:
\begin{itemize}
\item For some $C\ge 0$, $\mu = \boldsymbol{\mu}^{\cM}$ is a signed measure on~$\cA$ with total variation at most~$C$:
  \begin{itemize}
  \item $\mu(\varnothing) = 0$ and $\mu(A\cup B) + \mu(A\cap B) = \mu(A)+\mu(B)$ for all $A,B\in\cA$;
  \item $\|\mu\| \coloneqq \sup\{|\mu(A)|+|\mu(B)|-|\mu(A\cap B)| : A, B \in \cA\} \le C$.
  \end{itemize}
\end{itemize}
A \emph{pre-Loeb probability structure} is a pre-Loeb finite measure structure $\cM$ with $\|\mu\| = 1$.
A \emph{Loeb probability (finite measure, signed measure) structure} is any $\lamp$-saturated pre-Loeb probability (finite measure, signed measure) structure, where $\lambda = \card(L)$.
\end{definition}
Clearly, a positive pre-Loeb structure is a signed pre-Loeb structure.
It can be easily verified that the equality $\|\mu\| =  \sup\{|\mu(A)|+|\mu(B)|-|\mu(A\cap B)| : A, B \in \cA\}$ holds in any pre-Loeb measure structure (probability, finite or positive). 
On the other hand, in signed measure structures only the inequality $|\mu(\Omega)| \le \|\mu\|$ holds in general.

Recall (Definition~\ref{D:model}) that if $L'$ is a signature, a class $\mathcal{C}$ of metric structures is said to be $L'$-axiomatizable if it consists of the models of a (positive bounded) $L'$-theory.

\begin{proposition}\label{prop:Loeb-1st-order}
Let~$L$ be a signature for Loeb structures. 
For every fixed $C\ge 0$, the class of all pre-Loeb probability (finite measure, signed measure) $L$-structures $\cM$ such that $\|\mu\|\le C$ is axiomatizable. 
\end{proposition}
(The assumption $C=1$ in the case of probability structures is tacit.)
\begin{proof}

Given $C\ge 0$, we have to show that the clauses of Definition~\ref{def:Loeb-struct} are equivalent to the (approximate) satisfaction of positive bounded $L$-sentences. 
Clearly, it suffices to do so under the assumption that $C$ is rational (if $C$ is irrational, simply take the union of all axiom schemes for rational $D>C$).
This is a routine exercise, so we give only one example. 
The condition $\|\mu\|\le C$ amounts to the approximate satisfaction of the sentence%
\footnote{We remark that the value $D = \|\boldsymbol{\mu}\|$ is itself uniquely characterized by the formula
$(\forall A,B)(|\boldsymbol{\mu}(A)| + |\boldsymbol{\mu}(B)| - |\boldsymbol{\mu}(A\cap B)| \le D)
\wedge
(\exists A,B)(|\boldsymbol{\mu}(A)| + |\boldsymbol{\mu}(B)| - |\boldsymbol{\mu}(A\cap B)| \ge D)$.}
\begin{equation*}
  (\forall A,B)(|\boldsymbol{\mu}(A)| + |\boldsymbol{\mu}(B)| - |\boldsymbol{\mu}(A\cap B)| \le C).
\end{equation*}
The reader is invited to write down formulas axiomatizing the remaining clauses. 
\end{proof}

\begin{definition}\label{def:ext-Loeb-alg}
  For any pre-Loeb structure~$\cM$, define
  \begin{equation*}
    \begin{split}
\cA^{\cM} &\to \mathcal{P}(\Omega^{\cM}) \\
    A &\mapsto [A]^{\cM} = \{\omega\in\Omega^{\cM} :  \llbracket\omega\in A\rrbracket = 1\}.
    \end{split}
  \end{equation*}
The collection
  \begin{equation*}
    [\cA]^{\cM} = \{[A]^{\cM} : A\in\cA\}
  \end{equation*}
of subsets of~$\Omega^{\cM}$ is the \emph{induced (external) algebra in~$\Omega^{\cM}$.}
It follows from the definition of pre-Loeb structure that $[\cA]^{\cM}$ is an algebra of subsets of~$\Omega^{\cM}$
(the map $A \mapsto [A]^{\cM}$ is an isomorphism between the Boolean algebras $\cA$ and $[\cA]^{\cM}$).

The \emph{(external) measure induced by~$\mu$ on~$\cM$} is the real-valued function
\begin{equation*}
  \begin{split}
[\mu]^{\cM} : [\cA]^{\cM} &\to \RR \\
    [A]^{\cM} &\mapsto \mu(A).
  \end{split}
\end{equation*}
$[\mu]^{\cM}$ is well defined because $A\mapsto[A]^{\cM}$ is injective.

When $\cM$ is fixed, we usually write $[\mu]$, $[\cA]$, $[A]$ for $[\mu]^{\cM}$, $[\cA]^{\cM}$, $[A]^{\cM}$. 
We may also write $[\Omega]$ instead of $\Omega^{\cM}$, for emphasis.
\end{definition}

\begin{proposition}\label{prop:Loeb-measure}
Fix a Loeb structure~$\cM$ and let $\AL$ be the $\sigma$-algebra of subsets of $\Omega$ generated by $[\cA]$.
  \begin{enumerate}[\normalfont(1)]
  \item The function $[\mu]$ extends to a unique countably additive positive [signed] measure~$\muL$ on~$\AL$ with total variation $\var(\muL) = \|\mu\|$. 
  \item For all $S\in\AL$ and $\epsilon>0$ there exist $A,B\in\cA$ such that $[A]\subset S\subset [B]$ and $\|\mu\restriction (B\setminus A)\| \le \epsilon$, where
\vspace{-0.5cm}
    \begin{equation*}
      \begin{split}
        \mu\restriction X : \cA &\to \RR\\
Y &\mapsto \mu(Y\cap X)
      \end{split}
    \end{equation*}
for any $X\in\cA$.
\item For every $S\in\AL$ there exists $A\in\cA$ such that $|\muL|(S\triangle[A])=0$, where $X\triangle Y = (X\setminus Y)\cup(Y\setminus X)$ is the symmetric set difference and $|\muL|$ is the (nonnegative) absolute measure of~$\muL$.
  \end{enumerate}
\end{proposition}

\begin{proof}
We prove the statements for $\mu$ a positive finite measure, leaving the case of a signed measure to the reader.
  \begin{enumerate}
  \item Clearly, $[\mu]$ is a finitely additive nonnegative measure on~$([\Omega],[\cA])$ with finite total variation $\var([\mu]) = \mu(\mathtt{\Omega}) = \|\mu\|$. 
Suppose that $([A_i] : i\in\NN)$ is a descending chain in~$[\cA]$ (for some descending chain $(A_i)$ in~$\cA$). 
Assume that, for every $i\in\NN$, $[A_i]$ is nonempty.
Choose $\omega_i\in[A_i]$. 
By saturation, since $(A_i)$ is descending, we have $\llbracket\omega\in A_i\rrbracket=1$ for some $\omega$ and all $i\in\NN$. 
Certainly, $\omega\in\bigcap_{i\in\NN}[A_i]$. 
Therefore, if $\bigcap_{i\in\NN}[A_i] = \emptyset$, we must have $A_j = \emptyset$ for some $j\in\NN$, hence $\inf_{i\in\NN}[\mu]([A_i]) = [\mu]([A_j]) = \mu(A_j) = \mu(\varnothing) = 0$. 
By taking relative complements, the preceding argument implies that $[\mu]$ is countably additive on~$[\cA]$, hence a premeasure thereon (in fact, if a set $B\in[\cA]$ is a countable union of sets in $[\cA]$, then $B$ is necessarily a finite union of such sets).
By the Carathéodory Extension Theorem, $[\mu]$ admits an extension to a countably-additive measure on~$\sigma[\cA]=\AL$ with total variation $\var(\muL) = \var([\mu]) = \|\mu\|$. 
The extension is unique because $[\mu]$ has finite total variation. 

\item This assertion also follows from Carathéodory's theorem.

\item Let $S\in\AL$.
By part~(2), for each $n\in\NN$ we may choose $A_n,B_n\in\cA$ with $[A_n]\subset S\subset[B_n]$ and $\mu(B_n\setminus A_n)\le 1/(n+1)$. 
Clearly, $\inf_n\mu(B_n\setminus A_n) = 0$.
Without loss of generality, $(A_n)$ is increasing and $(B_n)$ decreasing.
By saturation, there exists $A\in\cA$ such that $A_n\subset A\subset B_n$ for all $n\in\NN$. 
Since $(A_n)$ is increasing and $(B_n)$ decreasing, $U = \bigcup_{n\in\NN}[A_n]\subset[A]\subset \bigcap_{n\in\NN}[B_n] = V$, 
and also $U\subset S\subset V$.
We have $\muL(S\triangle[A])\le\muL(V\setminus U) = \inf_n\mu(B_n\setminus A_n) = 0$, so $A$ is as required.\qedhere
  \end{enumerate}
\end{proof}

\begin{definition}
We call $\AL$ the \emph{Loeb algebra} of the Loeb structure~$\cM$, and $\muL$ the \emph{Loeb measure} on~$[\Omega]$ (i.e., on $\Omega^{\cM}$).
\end{definition}%

\begin{remarks}\hfill
\begin{enumerate}
\item
The proof of Proposition~\ref{prop:Loeb-measure} given above is an adaptation of the classical construction of Loeb measures~\cite{Loeb:1975} (see also the articles by Cutland~\cite{Cutland:1980, Cutland:2000}, and Ross~\cite{Ross:1997}, on which we base our approach). 
Our context differs from the classical one in the sense that we do not need to use nonstandard universes or hyperreals, and our measures need not be probability measures.
\item
Our definition of Loeb algebra differs from the classical definition, according to which the Loeb algebra is the completion $\overline{\AL}$ of~$\AL$ relative to~$\muL$ and every subset $U$ of an $\muL$-null set is declared to be $\overline{\AL}$-measurable and null. 
While the classical definition has the advantage that the converse statements to~(2) and~(3) of Proposition~\ref{prop:Loeb-measure}  hold, we prefer to avoid completing $\AL$ so its definition is independent of~$\mu$.
\end{enumerate}
\end{remarks}

\section{Integration structures}
\label{sec:prob-struct}

In this section, $L$ will denote a signature for Loeb structures (with sorts $\RR$, $\Omega$, $\cA$) such that, in addition to 
all the  constant and function symbols required for Loeb structures, $L$ includes a sort $\LO$ (whose interpretation will be a Banach lattice-algebra) and the following function symbols:
\begin{itemize}
\item a monadic symbol $\nLO{\cdot}:\LO\to\RR$ (to be interpreted as the Banach norm of sort~$\LO$);
\item function symbols $\sup : \LO\to\RR$ and $\inf : \LO\to\RR$; 
\item binary symbols 
\begin{align*}
+_{\LO}&:\LO\times\LO\to\LO,\\
-_{\LO} &:\LO\times\LO\to\LO,\\
\cdot_{\LO}&:\LO\times\LO\to\LO, 
\end{align*}
and $\centerdot_{\LO}:\RR\times\LO\to\LO$
(to be interpreted as the algebra operations of $\LO$ and the scalar multiplication of the interpretation of $\LO$, respectively);
\item binary symbols
\[
\wedge_{\LO}:\LO\times\LO\to\LO,\qquad \vee_{\LO}:\LO\times\LO\to\LO
\]
(to be interpreted as the lattice operations of  the interpretation of  $\LO$);
\item monadic symbols $\sup_{\Omega}:\LO\to\RR$ and $\inf_{\Omega}:\LO\to\RR$; 
\item a binary symbol $\ev_{\Omega} : \LO\times\Omega\to\RR$; 
\item a monadic symbol $\chi : \cA \to \LO$;

\item $I : \LO \to \RR$;
\item $\LO$-valued constant symbols $\mathtt{0}_{\LO}$ and $\mathtt{1}_{\LO}$.
\end{itemize}
A signature including the sorts and symbols above will be called a \emph{signature for integration structures.}
Such a signature may include many other sorts, function symbols, and constant symbols. 

\begin{notation}
If $L$ is a signature for integration structures, $\cM$ is an $L$-structure, $f\in(\LO)^{\cM}$, $\omega\in\Omega^{\cM}$ and $A\in\cA^{\cM}$, we write $\ev_{\Omega}^{\cM}(f,\omega)$ for $f(\omega)$,  $\chi_A$ for $\chi(A)$, and $If$ for $I(f)$.

If $\cM$ is an $L$-structure and $\cM$ is fixed by the context, in order to simplify the notation, we will denote the sort corresponding to $\LO$ in $\mathcal{M}$ as $\LO$, rather than $(\LO)^{\cM}$. Carrying this simplification one step further, we will remove the $\cM$-superscript from the interpretations of the function symbols in $\cM$; thus, for instance, we write $\nLO{\cdot}$ and $+_{\LO}$ instead of $(\nLO{\cdot})^{\cM}$ and  $(+_{\LO})^{\cM}$, respectively. 
\end{notation}

\begin{definition}[Integration structure]
Let $L$ be a signature for integration structures. 
A \emph{positive pre-integration structure} is an $L$-structure $\cM$ such that:
\begin{itemize}
\item The interpretation of 
($\LO$, $+_{\LO}$, $-_{\LO}$, $\centerdot_{\LO}$, $\cdot_{\LO}$, $\wedge_{\LO}$, $\vee_{\LO}$,  $\nLO{\cdot}$, $0_{\LO}$)  in $\cM$ is a Banach algebra and Banach lattice with anchor~$0_{\LO}$ (the \emph{zero function});
\item for all $f\in\LO$ and $\omega\in\Omega$, we have:
  \begin{itemize}
  \item the metric on~$\LO$ is that induced by the norm~$\nLO{\cdot}$; 
  \item $\nLO{f} = \sup_{\omega'\in\Omega}|f(\omega')|$; 
  \item $\ev_{\Omega}(\cdot,\omega)$ is an algebra and lattice homomorphism; 
  \item $\sup f = \sup_{\omega'\in\Omega}f(\omega')$ and $\inf f = \inf_{\omega'\in\Omega}f(\omega')$;
  \item $0_{\LO}(\omega) = 0$ and $1_{\LO}(\omega) = 1$;
  \end{itemize}
\item each $f\in\LO$ is \emph{approximately $\cA$-measurable} (\emph{$\Aapp$measurable}) in the following sense: 
For any reals $u<v$ there exists $A\in\cA$ such that $f(\omega)\le v$ if  $\llbracket\omega\in A\rrbracket = 1$, and $f(\omega)\ge u$ if $\llbracket\omega\in A\rrbracket = 0$;
\item $\chi_A(\omega) = \llbracket \omega\in A \rrbracket$ for all $A\in\cA$ and $\omega\in\Omega$;
\item for some $C\ge 0$: 
  \begin{itemize}
  \item $\cM$ is a pre-Loeb finite measure structure satisfying $\mu(A) = I\chi_A$ for all $A\in\cA$, and $\|\mu\| \coloneqq \mu(\Omega)\le C$ (we also define $\|I\| = \|\mu\|$);
  \item the \emph{integration operation}~$I$ is a $\|\mu\|$-Lipschitz linear functional~$\LO\to\RR$;
  \item $\|\mu\|\inf f \le If \le \|\mu\|\sup f$ for all $f\in\LO$
(in particular, $I$ is a positive functional: $If \ge 0$ if $f \ge 0$).
  \end{itemize}
\end{itemize}

In a \emph{signed pre-integration structure}, the last axiom becomes:
\begin{itemize}
\item For some $C\ge 0$:
 \begin{itemize}
  \item $\cM$ is a pre-Loeb signed measure structure satisfying  $\|\mu\| \le C$ and $\mu(A) = I\chi_A$ for all $A\in\cA$ (we define $\|I\| \coloneqq \|\mu\|$);
  \item $I$~is a $\|\mu\|$-Lipschitz linear functional~$\LO\to\RR$;
 \item $|If| \le \|\mu\|\nLO{f}$ for all $f\in\LO$.
 \end{itemize}
\end{itemize}

A \emph{probability pre-integration structure} is a positive integration structure with $\|\mu\| = 1$. 
A \emph{probability \emph{(resp., \emph{positive}, \emph{signed})} integration structure} is any \lamp-saturated probability (resp., positive, signed) pre-integration structure, where $\lambda=\card(L)$.
\end{definition}

For $f\in\LO$, we let $f_+ = (f\vee 0)$, $f_- = (-f)_+$ (so $f = f_++f_-$ with $f_+,f_-\ge 0$), and $|f| = f_++f_-$ (so $|f|(\omega) = |f(\omega)|$ for all $\omega\in\Omega$). 
For syntactic convenience, we treat the order of the lattice $\LO$ as part of the language, so we write $f\le g$ to mean $\nLO{(f\wedge g) - f}\le 0$. 

\begin{remark}
The usual definition of measurability suggests postulating that for every $f\in\LO$ and every interval $J\subset\RR$ there shall exist $A\in\cA$ such that $\omega\in\cA$ if and only if $f(\omega)\in J$.  
However, exact measurability in this sense is not axiomatizable by positive bounded $L$-sentences.  
On the other hand, approximate measurability is axiomatizable as shown below in Proposition~\ref{prop:Loeb-measurability}.  
Perhaps surprisingly, this postulate fails even in (saturated) probability structures.  
By working in saturated probability spaces and externally enlarging $[\cA]$ to a $\sigma$-algebra $\sigma[\cA]$ of subsets of~$\Omega$, the $\Aapp$measurability axiom implies the exact $\sigma[\cA]$-measurability of all functions $[f]:\Omega\to\RR$ for $f\in\LO$.  
See Proposition~\ref{prop:Loeb-measurability} below.
\end{remark}

\begin{proposition}\label{prop:meas-struc-1st-ord}
Let~$L$ be a language for pre-measure structures. 
For every $C\ge 0$, the class of all (probability, positive, or signed) pre-measure $L$-structures $\cM$ with $\|\boldsymbol{\mu}^{\cM}\|\le C$ is axiomatizable in the logic of approximate satisfaction of positive bounded formulas.
\end{proposition}
\begin{proof}
It is a routine exercise to verify that the axioms in the definition of integration structure can be written as sets of positive bounded $L$-sentences. 
As an example, 
the $\Aapp$measurability condition amounts to the axiom schema
\begin{equation*}
  (\forall_{\!D}\!f)(\exists A)(\forall \omega)
(r\chi_{\compl{A}} - D\chi_A \le f(\omega) \le r\chi_A + D\chi_{\compl{A}})
\end{equation*}
for all rational $D>0$ and~$r$, as easily seen from the semantics of approximate satisfaction. 
The remaining axioms are handled similarly. 
\end{proof}

\begin{notation}
  Given any function $G:\mathbb{R}\to\mathbb{R}$, an element $f\in\LO$, and a real number $t\in\RR$, we let
\begin{equation*}
  \{G(f) \le t\} \coloneqq \{\omega\in\Omega : G(f(\omega))\le t\},
\end{equation*}
with similar definitions for $\{G(f) < t\}$, $\{G(f) \ge t\}$, and $\{G(f) > t\}$.
\end{notation}

\begin{proposition}\label{prop:Loeb-measurability}
  Let $\cM$ be an integration $L$-structure.
  For every $f\in\LO$, the function $[f] = [f]^{\cM} : \Omega\to\RR$ defined by $\omega \mapsto f(\omega)$ is $\AL$-measurable.
\end{proposition}
\begin{proof}
Since $\AL$ is a $\sigma$-algebra, it suffices to show that $\{f \le t\}\in\AL$ for any fixed real number~$t$. 

Let $(u_n:n\in\NN)$ be a strictly decreasing sequence of rational numbers such that $\inf_nu_n = t$. 
By $\Aapp$measurability and saturation, for each $n\in\NN$ there exists $A_n\in\cA$ such that $f(\omega)\le u_n$ if $\omega\in[A_n]$, and $f(\omega)\ge u_n$ if $\omega\notin[A_n]$. 
Clearly, $\{f\le t\} = \bigcap_{n\in\NN}[A_n] \in \AL$.
\end{proof}

\begin{theorem}
[Riesz Representation Theorem for integration structures]\label{thm:Riesz}
  Let $\cM$ be an integration $L$-structure (positive or signed).
For every $f\in\LO$, $[f]$ is Loeb-integrable and
\begin{equation*}
  If = \int [f]\,\dd\muL.
\end{equation*}
\end{theorem}
\begin{proof}
We assume that $I$ is positive, leaving the signed case to the reader. 
Let $C = \|I\| = \|\mu\|$ and $D = \nLO{f}$.
$[f]$ is Loeb integrable because it is  $\AL$-measurable (Proposition~\ref{prop:Loeb-measurability}) and bounded (by $D$). 
The Loeb measure $\muL$ is also positive with total variation $\var(\muL) = \muL(\Omega)= \mu(\Omega) = \|\mu\| = C$.
Let us write $\int\!F$ for $\int\!F\dd\muL$. 
If $A\in\cA$, we have $I\chi_A = \mu(A) = \muL([A]) = \int[\chi_A]$.
Fix $\epsilon > 0$ and let 
$(J_i : i<k)$ be any finite collection of disjoint intervals, each having length at most~$\epsilon$, such that $\bigcup_{i<k}J_i \supset [-D,D]$. 
For $i<k$, let $S_i = \{f\in J_i\}\subset\Omega$.
By measurability of $[f]$ (Proposition~\ref{prop:Loeb-measurability}), the collection $(S_i : i<k)$ is an $\AL$-measurable disjoint cover of~$\Omega$, hence $\sum_{i<k}\muL(S_i) = \muL(\Omega) = C$.
Choose rational numbers $r_i$ ($i<k$) such that $J_i\subset [r_i,r_i+2\epsilon]$.
The $\AL$-simple function $F_{\epsilon} = \sum_{i<k}r_i\chi_{S_i }: \Omega\to\RR$ satisfies $F_{\epsilon}\le [f]\le F_{\epsilon}+2\epsilon$, hence $\int\!F_{\epsilon} \le \int[f] \le \int\!F_{\epsilon} + 2C\epsilon$. 
By internal approximability (Proposition~\ref{prop:Loeb-measure}(3)), there exist $A_i\in\cA$ ($i<k$) such that $\muL([A_i]\triangle S_i) = 0$; thus $\sum_{i<k}\mu(A_i) = \sum_{i<k}\muL(S_i) = C$. 
Without loss of generality we may assume that the sets $A_i$ are pairwise disjoint. 
Let $B = \compl{\left( \bigcup_{i<k} A_i \right)}$, so $\mu(B) = 0$.
Let $f_{\epsilon} = f\chi_B + \sum_{i<k}r_i\chi_{A_i}$.
Certainly, $f_{\epsilon} \le f\le f_{\epsilon}+2\epsilon$, so $I f_{\epsilon} \le If \le If_{\epsilon}+2C\epsilon$. 
Since $\mu(B)=0$, we have $I f_{\epsilon} = \sum_{i<k} r_i\mu(A_i) = \sum_{i<k} r_i\muL(S_i) = \int\!F_{\epsilon}$.
Thus, $If$ and $\int[f]$ both lie in $[If_{\epsilon},If_{\epsilon}+2C\epsilon]$, so $\left| If - \int[f] \right|\le 2C\epsilon$. 
Since $C$ is fixed and this holds for all $\epsilon>0$, we conclude that $If = \int[f]$.
\end{proof}

\section{Directed integration structures}
\label{sec:DCT}

\begin{definition}[Directed integration structure]
Fix a pointed directed set $(\cD,\le,j_0)$. 
Let $L$ be both a signature for integration structures (with sorts $\RR$, $\Omega$, $\cA$, $\LO$) and also for $\cD$-directed structures (with directed sort~$\DD$). 
In addition, assume that $L$ has a function symbol 
\begin{equation*}
\boldsymbol{\varphi} : \DD\to\LO.
\end{equation*}
Such $L$ will be called a \emph{signature for directed integration structures}
(and it may include any other sorts and symbols than those named).

A \emph{$\cD$-directed pre-integration structure} is an $L$-structure $\cM$ that is both a $\cD$-directed $L$-structure and a pre-integration $L$-structure (positive or signed).
A \emph{$\cD$-directed integration structure} is an \lamp-saturated $\cD$-directed pre-integration structure, where $\lambda=\card(L)$.

Given any $\cD$-directed pre-integration structure~$\cM$, let $\varphi = \boldsymbol{\varphi}^{\cM}$. 
Note that the definition of $L$-structure implies that $\varphi$ is uniformly bounded on the discrete sort~$\DD^{\cM}$. 
We define
\begin{equation*}
 \|\varphi\| \coloneqq \sup_{j\in\DD}\nLO{\varphi(j)}.
\end{equation*}
For $j\in \DD$, we will denote $\varphi(j)$ by $\varphi_j$. 
Each $\omega\in\Omega$ defines an external $\cD$-net $\fb(\omega) \coloneqq (\varphi_j(\omega):j\in\cD)$ in~$\LO$. 
We also have a real-valued $\cD$-net $I\fb \coloneqq (I(\varphi_j):j\in\cD)$. 
\end{definition}

\begin{proposition}\label{prop:int-struc-1st-ord}
Let~$L$ be a language for directed integration structures and $C\ge 0$.
The class of all (positive or signed) $\cD$-directed pre-integration $L$-structures $\cM$ such that $\|\varphi\|\le C$ and $\|\mu\|\le C$ is axiomatizable in the logic of approximate satisfaction.
\end{proposition}
\begin{proof}
This follows from Propositions~\ref{prop:Loeb-1st-order} and~\ref{prop:meas-struc-1st-ord}, upon remarking that the condition $\|\varphi\|\le C$ is equivalent to the approximate satisfaction of the axioms%
\footnote{The value $D = \|\varphi\|$ is characterized by the approximate satisfaction of the formula
\begin{equation*}
(\forall j)
\big(\nLO{\boldsymbol\varphi_j} \le D\big) 
\wedge (\exists j)\big(\nLO{\boldsymbol\varphi_j} \ge D\big).
\end{equation*}}
\begin{equation*}
(\forall j)\big(\nLO{\boldsymbol\varphi_j} \le r\big)\qquad
\text{for all rational $r>C$.\qedhere}
\end{equation*}
\end{proof}

\begin{proposition}[Dominated Convergence Theorem]\label{prop:DCT}
Assume that the directed set~$\cD$ is countable. 
Let $\cM$ be a (saturated) $\cD$-directed integration $L$-structure, for a suitable language~$L$. 
Then
\begin{equation*}
  \osc(I\fb) \le \|I\|\sup_{\omega\in\Omega}\osc(\fb(\omega))\qquad
\text{for all $\omega\in\Omega$.}
\end{equation*}
In particular, $I\fb$ is convergent if $\fb(\omega)$ is convergent for each $\omega\in\Omega$.
\end{proposition}
\begin{proof}
We prove the inequality $\osc(I\fb) \le \sup_{\omega\in\Omega}\osc(\fb(\omega))$ in the case when $\cM$ is a probability integration structure (with $\|I\| = 1$), leaving the general case to the reader. 
Let $r = \displaystyle\sup_{\omega\in\Omega}\osc(\fb(\omega))$.

\vspace{-1ex}
For $\epsilon > 0$
and $i,j,j'\in\cD$, let
\begin{equation*}
  \Omega^{j,j'}_{\epsilon} 
= \{\omega\in\Omega : |\varphi_{j'}(\omega)-\varphi_j(\omega)| \le r+\epsilon\}
\end{equation*}
and
\begin{equation*}
  \Omega^i_{\epsilon} = \bigcap_{j,j'\in\cD_{\ge i}}\Omega^{j,j'}_{\epsilon} =  \{\omega\in\Omega : |\varphi_{j'}(\omega)-\varphi_j(\omega)| \le r+\epsilon \quad\text{for all $j,j'\in\cD_{\ge i}$}\}.
\end{equation*}
The functions $\varphi_j(\cdot)$ and $\varphi_{j'}(\cdot)$ are $\AL$-measurable, by Proposition~\ref{prop:Loeb-measurability}. 
It is clear that each set $\Omega_{\epsilon}^{j,j'}$ is also $\AL$-measurable, and so is the (countable) intersection~$\Omega_{\epsilon}^i$. 
Clearly, $\Omega_{\epsilon}^i\subset\Omega^{i'}_{\epsilon}$ for $i\le i'$.
Since $\omega\in\Omega^i_{\epsilon}$ implies $\osc_{\eta_i}(\varphi_{\bullet}(\omega)) \le r+\epsilon$, while $\osc(\fb(\omega)) \le r$ for all $\omega$ by hypothesis, part~(2) of Proposition~\ref{lem:osc-metastab} gives:
\begin{equation*}
  \bigcup_{i\in\cD}\Omega_{\epsilon}^i = \Omega.
\end{equation*}
Since $\cD$ is countable and $\muL$ is a probability measure, we have $\sup_{i\in\cD}\muL(\Omega_{\epsilon}^i) = \muL(\Omega) = 1$, hence $\mu(\Omega^{i_0}_{\epsilon})\ge 1-\epsilon/(\|\varphi\|+1)$ for some $i_0\in\cD$. 
For $j,j'\ge i_0$, we have:
\begin{equation*}
  \begin{split}
    | I(\varphi_{j'})-I(\varphi_j) | &\le I\big(|\varphi_{j'}-\varphi_j|\big) = \int\!\left|\varphi_{j'}-\varphi_j\right|\chi_{\Omega_{\epsilon}^{i_0}}
    + \int\!\left|\varphi_{j'}-\varphi_j\right|\chi_{\compl{(\Omega_{\epsilon}^{i_0})}}\\
    &\le (r+\epsilon)\muL(\Omega_{\epsilon}^{i_0}) + 2\|\varphi\|\muL(\compl{(\Omega_{\epsilon}^{i_0})}) \le  r + 3\epsilon.
  \end{split}
\end{equation*}
Therefore, $\osc(I\fb)\le \osc_{\eta_i}(I\fb) \le r+3\epsilon$ (again, by part~(2) of Proposition~\ref{lem:osc-metastab}). 
Since $\epsilon$ is an arbitrary positive number, $\osc(I\fb)\le r = \sup_{\omega\in\Omega}\osc(\fb(\omega))$.
\end{proof}
Note that the proof above uses only the standard theory of integration.

\begin{corollary}[Metastable Dominated Convergence Theorem]
\label{prop:meta-int}
Let~$(\cD,\le,j_0)$ be a directed set with $\cD$ countable. 
Let $\cM$ be a (not necessarily saturated) $\cD$-directed pre-integration $L$-structure, for a suitable signature~$L$. 
Fix a real number $s\ge 0$. 
Let $T$ be any uniform theory including the axiom $\|I\|\le s$ (that is, the axioms $\|I\|\le u$ for all rationals $u>s$) and extending the theory of $\cD$-directed pre-integration $L$-structures. 
Given $r\ge 0$ and any collection $\Eb^r = (E_{\epsilon,\eta} : \eta\in\Smpl(\cD), \epsilon>r)$ in $\PfinD$, there exists another collection $\widetilde{\Eb}{}^{\!\!rs} = (\widetilde{E}_{\epsilon,\eta} : \eta\in\Smpl(\cD), \epsilon>rs)$ such that every model $\cM$ of~$T$ satisfies the following property: 

\emph{If every external $\cD$-net in the collection $\cC = (\fb(\omega) : \varphi\in({\LO})^{[1]}, \omega\in\Omega)$ is $\Eb^r$-uniformly $r$-metastable, then $\widetilde{\Eb^{rs}}$ is a rate of $rs$-metastability for the collection $(I\fb : \varphi\in(\LO)^{[1]})$.}

In fact, one such $\widetilde{\Eb^{rs}}$ may be found depending only on $r,s$, and~$\Eb^r$.

\end{corollary}
\begin{proof}
Let $r, s \ge 0$ and $\Eb^r$ be given.
Restrict the signature $L$ so it only names the sorts and symbols strictly required for a signature of $\cD$-directed pre-integration structures. 
Let $\TT = T[\boldsymbol{\sigma},\mathtt{a}]$ where $\mathtt{a} : \DD\to\RR$ and $\boldsymbol{\sigma} : \LO\times\Omega\times\DD \to \RR$ are new function symbols, and let $T$ be the theory of pre-integration $L$-structures augmented with the following $\LL$-axioms: 
\begin{itemize}
\item $\|\mathtt{I}\|\le C$, for all rational $C>s$;
\item $(\forall_{\!C}\varphi)(\forall j)(\forall\omega)(|\boldsymbol{\sigma}(\varphi,\omega,j) - \varphi_j(\omega)| \le 0)$, for each rational $C\ge 0$;
\item $(\exists_1\varphi)(\exists \omega)(\forall j)(|\mathtt{a}(j)-\boldsymbol{\sigma}(\varphi,\omega,j)| \le 0)$.
\end{itemize}
Clearly, $T$ is a uniform $\LL$-theory, and $\sigma$ is $(T\restriction L[\boldsymbol{\sigma}])$-definable in~$L$; 
in fact, every model $\cM$ of $T\restriction L$ admits a unique expansion to a model $(\cM,\sigma)$ of~$T\restriction L[\boldsymbol{\sigma}]$, and any such model admits some expansion $(\cM,\sigma,a)$ to a model of~$T$ via any $a : \DD^{\cM}\to\RR$ that approximately admits $\sigma$-parameters of size $\lesssim 1$. 
Using Proposition~\ref{prop:metastab-1st-ord}, extend $T$ to a (necessarily uniform) $\LL$-theory $\TT$ such that $\ab$ is $\Eb^r$-uniformly $r$-metastable for all models $(\cM,\sigma,a)$ of~$\TT$. 
Henceforth all $L'$-structures for a language $L'\subset \LL$ are assumed to be models of $\TT\restriction L'$. 
If $(\cM,\sigma)\appmod \TT\restriction L[\boldsymbol{\sigma}]$, every external sequence $\fb(\omega)$ with $\|\varphi\|\le 1$ admits $\sigma$-parameters $(\varphi,\omega)$ of size $\le 1$, and every such $\fb(\omega)$ is of the form $\ab$ for some expansion $(\cM,\sigma,a)\appmod \TT$ of $(\cM,\sigma)$.
By Proposition~\ref{lem:osc-metastab}, $\osc(\ab) \le r$, hence $\osc(\fb(\omega)) \le r$ whenever $\|\varphi\|\le 1$. 
Since $\|I\|\le s$, we have $\osc(I\fb) \le rs$ by Proposition~\ref{prop:DCT}.
Let $\mathtt{b} : \DD\to\RR$ and $\boldsymbol{\tau} : \LO\times\DD\to\RR$ be new function symbols. 
Expand $\TT$ to an $\LL[\boldsymbol{\tau},\mathtt{b}]$-theory $\TT'$ by adding the axioms
\begin{itemize}
\item $(\forall_{\!C}\varphi)(\forall j)(|\boldsymbol{\tau}(\varphi,j) - I\varphi_j| \le 0)$, for each rational $C\ge 0$;
\item $(\exists_1\varphi)(\forall j)(|\mathtt{b}(j)-\boldsymbol{\tau}(\varphi,j)| \le 0)$.
\end{itemize}
Clearly, $\TT'$ is uniform, $\tau$ is $(\TT'\restriction \LL[\boldsymbol{\tau}])$-definable in~$\LL$, and an expansion $(\cM,\tau,b)$ of a model $(\cM,\tau)$ of~$\TT'\restriction\LL[\boldsymbol{\tau}]$ to a model of~$\TT'$ is via $b : \DD^{\cM}\to\RR$ that approximately admits $\tau$-parameters of size $\lesssim 1$. 
Henceforth, all $L'$-structures for any signature $L'\subset\LL[\boldsymbol{\tau},\mathtt{b}]$ are assumed to be models of~$\TT'\restriction L'$. 
Since $\TT$ implies that $\osc(I\fb)\le rs$ whenever $\|\varphi\|\le 1$, $\TT'$ implies%
\footnote{\emph{A priori,} $\TT'$ only implies that $\osc(\bb)\le rs$ when $\bb$ admits $\tau$-parameters of size~$\le 1$, so $\bb : j\mapsto\tau(\varphi,j) = I\varphi_j$ for some $\|\varphi\|\le 1$. 
If this is the case, however, then any $\bb$ that approximately admits $\tau$-parameters of size $\lesssim 1$ must still satisfy $\osc(\bb)\le rs$. 
Alternatively, the rest of the proof may proceed applying Corollary~\ref{cor:UMP-paramd} in the stronger version that only assumes metastability when $\bb$ admits exact parameters of size $\le 1$.} 
that $\osc(\bb)\le rs$.
By uniform metastability of parametrized sequences (Corollary~\ref{cor:UMP-paramd}), 
there exists a collection~$\widetilde{\Eb^{rs}}$ depending only on~$\TT'$ such that all sequences $\bb$, and hence all sequences $I\fb$ for $\|\varphi\|\le 1$, are $\widetilde{\Eb^{rs}}$-uniformly $rs$-metastable. 
Moreover, $\widetilde{\Eb^{rs}}$ depends only on $\TT'$, hence only on $r,s$ and the given rate~$\Eb^r$.
\end{proof}

\section{Background on metric model theory}
\label{sec:metric-structures}
This section describes the general framework for the classes of structures that are the focus of the paper. 
We refer to these structures as \emph{metric structures}.  

\subsection*{Henson metric structures}

Recall that a \emph{pointed metric space} is a triple $(M,d,a)$, where $(M,d)$ is a metric space and $a$ is a distinguished element of $M$ called the \emph{anchor} of $M$. 
If $(M,d,a)$ is a pointed metric space, the closed ball of radius $r$ around the anchor point $a$ will be denoted 
$B_M[r]$, or simply $B[r]$ if the ambient space $M$ is clear from the context; 
the corresponding open ball will be denoted  $M^{(r)}$
or $B(r)$.
If $(M_1,d_1,a_1), \dots, (M_n,d_n,a_n)$ are pointed metric spaces, we regard the product $\prod_{i=1}^n (M_i,d_i,a_i)$ tacitly as a pointed metric space by  taking $(a_1,\dots,a_n)$ as its anchor and using the supremum metric. 

\begin{definition}
A \emph{metric \emph{(or \emph{Henson})} structure}~$\mathcal{M}$ (often just called a structure in this manuscript) consists of the following items:
\begin{itemize}
\item
A family $(M^{(s)} \mid s\in \mathbf{S})$ of pointed metric spaces,
\item
A collection of functions of the form
\[
F: M^{(s_1)}\times\dots \times M^{(s_n)} \to M^{(s_0)},
\]
each of which is locally uniformly continuous, i.e., uniformly continuous on each bounded subset of its domain.
\end{itemize}
The spaces $M^{(s)}$  are called the \emph{sorts} of $\mathcal M$. 
We say that $\mathcal{M}$ \emph{is based on} the collection $(M^{(s)} \mid s\in\mathbf{S})$ of its sorts. 

We do require that every metric structure contain the set $\mathbb{R}$ of real numbers, equipped with the usual distance and $0$ as an anchor point, as a distinguished sort. 
We also require that the given metric on each sort of $\mathcal M$ be included in the list of functions of $\mathcal M$, and that the anchor of each sort be included as a (constant) function.

If $\mathcal{M}$ is based on $(M^{(s)} \mid s\in\mathbf{S})$ an element of $M^{(s)}$ will be called an \emph{element} of $\mathcal{M}$ of sort~$s$.  
The \emph{cardinality} of $\mathcal{M}$, denoted $\card(\mathcal{M})$, is defined as $\sum_{s\in S}\card (M^{(s)})$.

Some of the sorts $M$ of a structure may be discrete metric spaces, with the respective metric $\dd : M\times M \to \{0,1\}$ taking the value~$1$ at every pair of distinct points.
If all the sorts of $\mathcal M$ are discrete, we will say that $\mathcal{M}$ is a discrete structure. 
Similarly, if the sorts of $\mathcal M$ are bounded, we will say that $\mathcal M$ is a bounded metric structure.

Some of the functions of a structure $\mathcal M$ may have arity $0$. 
Such functions correspond to distinguished elements of $\mathcal{M}$. 
We will call these elements the \emph{constants} of the structure. 
If $F$ is a $\{0,1\}$-valued function of $\mathcal M$, we will identify $F$ with a subset of its domain, namely, $F^{-1}(1)$.  
Such a function will be called a \emph{relation}, or a \emph{predicate}, of $\mathcal M$.

We will require that the special sort $\mathbb{R}$ should come equipped with the field operations of $\mathbb{R}$, the order relation and the lattice operations ($\max(x,y)$ and $\min(x,y)$), plus a constant for each rational number.

\end{definition}

If a structure $\mathcal{M}$ is based on $(M^{(s)} \mid s\in\mathbf{S})$ and $(F^i \mid i\in I)$ is a list of the functions of $\mathcal M$, we write
\[
\mathcal{M}=(M^{(s)}, F_i \mid s\in\mathbf{S}, i\in I).
\]
For notational simplicity, the real sort $\mathbb{R}$, the metrics on the sorts of $\mathcal{M}$, and their respective  anchors need not be listed explicitly in this notation. 
We will only list them when needed for emphasis.

The structures that we will be dealing with are ``hybrid'' in the sense that some of their sorts are discrete, while others are genuine metric spaces. 
Typically, the nondiscrete structures will be Banach algebras or Banach lattices; in these the natural anchor point is $0$. 
The discrete sorts that we will encounter include partial orders and purely algebraic structures; in structures of this type, the particular choice of anchor point is often inconsequential.

\subsection{Henson signatures and structure isomorphisms}
\label{sec:signatures}
We will need a formal way to index the sorts and functions of any given structure $\mathcal{M}$. 
This is accomplished through the concept of \emph{signature} of a metric structure.
\begin{definition}
Let $\mathcal{M}$ be a structure based on $(M^{(s)} \mid s\in\mathbf{S})$. 
A \emph{Henson signature}~$L$ for $\mathcal{M}$ consists of the following items:
\begin{itemize}
\item
A sort index set $\mathbf S$,
\item
A special element $s_\mathbb{R}\in\mathbf{S}$ such that $M^{(s_\mathbb{R})}=\mathbb{R}$,
\item
For each function $F: M^{(s_1)}\times\dots \times M^{(s_n)} \to M^{(s_0)}$, a triple of the form 
\[
((s_1,\dots,s_n), f,s_0),
\]
where $f$ is a purely syntactic symbol called a \emph{function symbol} for $F$. 
We write $F=f^{\mathcal{M}}$ and call $F$ the \emph{interpretation} of $f$ in $\mathcal{M}$. 
We call $s_1\times\dots\times s_n$ and $s_0$ the \emph{domain} and \emph{range} of $f$, respectively. 
We express this by writing (purely formally)
\[
f:s_1\times\dots\times s_n \to s_0.
\]
 The number $n$ is called the \emph{arity} of the function symbol $f$.  
If $n=0$ and the constant value of $f^\mathcal M$ in $M^{(s_0)}$ is $c$, we call $f$ a \emph{constant symbol} for $c$. 

\end{itemize}
We express the fact that $L$ is a signature for $\mathcal{M}$ by saying that $\mathcal{M}$ is an \emph{$L$-structure}. 
A structure~$\cM$ for some Henson signature~$L$ will also be called a \emph{Henson structure.}
The \emph{cardinality} of a signature $L$, denoted $\card(L)$, is defined as
\[
\card(\bS)+\card(\{\, f \mid\text{$f$ is a function symbol of $L$}\,\})+\aleph_0,
\]
where $\mathbf S$ is the sort index set of~$L$.
\end{definition}

\begin{definition}
	\label{D:signatures}
If $L$ and $L'$ are signatures, we say that $L$ is a \emph{subsignature} of $L'$ (or that $L'$ is an \emph{extension} of $L$), and  write $L\subseteq L'$, if the following conditions hold:
\begin{itemize}
\item
The sort index set of $L$ is a subset of the sort index set of $L'$,
\item
Every triple of the form $((s_1,\dots,s_n), f,s_0)$ that is in $L$ is also in $L'$.
\end{itemize}
If $L,L'$ are signatures, we say that $L'$ is an \emph{extension by constants} of $L$ if $L$ and $L'$ have the same sort index set and every function symbol of $L'$ that is not in $L$ is a constant symbol. 
If the set of such constant symbols is $C$, we denote $L'$ as $L[C]$.
\end{definition}

\begin{definition}
	Let $L$ be a signature and let $\mathcal{M}$ and $\mathcal{N}$ be $L$-structures based on $(M^{(s)}\mid s\in\mathbf{S})$ and $(N^{(s)}\mid s\in\mathbf{S})$, respectively.
	\begin{enumerate}
	\item
	$\mathcal{M}$ is a \emph{substructure} of $\mathcal{N}$ if $M^{(s)}\subseteq N^{(s)}$ and, for each function symbol~$f$, the interpretation $f^\mathcal{N}$ of $f$ in $\mathcal N$ is an extension of $f^\mathcal{M}$.
	\item
	$\mathcal{M}$ and $\mathcal{N}$ are \emph{isomorphic} if there exists a family $\mathcal{I}=(\mathcal{I}^{(s)}\mid s\in\mathbf{S})$ of maps (called an \emph{isomorphism} from $\mathcal{M}$ into $\mathcal{N}$) such that for each $s\in\mathbf{S}$,  $\mathcal{I}^{(s)}:M^{(s)}\to N^{(s)}$ is a bijection that commutes with the interpretations of the function symbols of $L$, in the sense that if $f:s_1\times\dots\times s_n \to s_0$, then $\mathcal{I}^{(s_0)}(f^\mathcal{M}(a_1),\dots,f^\mathcal{M}(a_n))=f^\mathcal{N}(\mathcal{I}^{(s_1)}(a_1),\dots,\mathcal{I}^{(s_n)}(a_n))$. 
If  $a$ is an element of $M^{(s)}$ and  the sort index $s$ need not be made specific, we may write $\mathcal{I}(a)$ instead of $\mathcal{I}^{(s)}(a)$.
	\item 
	An \emph{automorphism} of $\mathcal{M}$ is an isomorphism between $\mathcal{M}$ and $\mathcal{M}$.
	\end{enumerate}
\end{definition}

\subsection{Uniform classes and ultraproducts of metric structures}
\label{sec:ultraproducts}
Recall that a \emph{filter} on a nonempty set~$\Lambda$ is a collection~$\mathcal{F}$ of subsets of~$\Lambda$ such that (\textit{i})~$\Lambda\in \mathcal{F}$ and $\emptyset\notin \mathcal{F}$, (\emph{ii})~$A\cap B\in \mathcal{F}$ if $A,B\in \mathcal{F}$, and (\emph{iii})~$A\in \mathcal{F}$ if $B\in \mathcal{F}$ and $A\supset B$. 
An \emph{ultrafilter} on~$\Lambda$ is a maximal filter $\cU$ on~$\Lambda$; equivalently, $\cU$ is a filter such that (\emph{iv})~$A\in\cU$ or $\Lambda\setminus A\in\cU$ for all $A\subset \Lambda$. 
If $\Lambda$ is an index set and $\mathcal{F}$ is a filter on $\Lambda$, we will say that a subset of $\Lambda$ is \emph{$\mathcal{F}$-large} if it is in $\mathcal{F}$. 
An ultrafilter $\cU$ on~$\Lambda$ is \emph{principal} if there exists $\lambda_0\in\Lambda$ such that $A\in\cU$ iff $A\ni\lambda_0$ for all $A\subset\Lambda$; otherwise, $\cU$ is \emph{nonprincipal}.
If $X$ is a topological space, $(x_\lambda)_{\lambda\in \Lambda}$ is a family of elements of $X$, and $\mathcal{F}$ is a filter on $\Lambda$, we will say that $(x_\lambda)_{\lambda\in \Lambda}$ \emph{converges to} an element $y\in X$ with respect to $\mathcal{F}$  if for every neighborhood $U$ of~$y$, the set $\{\lambda\in\Lambda\mid x_\lambda\in U\}$ is $\mathcal{F}$-large. 
If $X$ is compact Hausdorff, then for every family $x_{\bullet} = (x_\lambda)_{\lambda\in \Lambda}$ and every ultrafilter $\mathcal{U}$ on $\Lambda$ there exists a unique $y\in X$ such that $(x_\lambda)_{\lambda\in \lambda}$ converges to $y$ with respect to $\mathcal{U}$; 
this element $y$ is called the \emph{$\mathcal{U}$-limit} of  $(x_\lambda)_{\lambda\in \lambda}$ and is denoted $\Ulim x_{\bullet}$ or $\Ulim_{\lambda} x_\lambda$.

Let $(X_\lambda,d_\lambda)_{\lambda\in\Lambda}$ be a family of metric spaces and let $\mathcal{U}$ be an ultrafilter on $\Lambda$. 
The \emph{$\mathcal{U}$-ultraproduct} of $(X_\lambda,d_\lambda)_{\lambda\in \Lambda}$ is the metric space defined in the following manner. 
Let $\lixd \coloneqq \ell^\infty(X_\lambda,d_\lambda\mid\lambda\in\Lambda)$ be the set of all elements of $\prod_{\lambda\in\Lambda}X_\lambda$ that are bounded (when regarded as families indexed by~$\Lambda$ in the natural way). 
For $x=(x_\lambda)_{\lambda\in \Lambda}$, $y=(y_\lambda)_{\lambda\in \Lambda}$ in $\lixd$, and an ultrafilter $\cU$ on~$\Lambda$, define
\[
d(x,y)= \Ulim_{\lambda} d_\lambda(x_\lambda, y_\lambda).
\]
Since elements of $\lixd$ are bounded families, it is clear that $d$ is well defined. 
It is also easy to verify that $d$ is a pseudometric on $\lixd$. 
Now we can turn $d$ into a metric in the usual way, namely by identifying any two elements $x,y\in\lixd$ such that $d(x,y)=0$. 
For $x\in\lixd$, we let $(x)_{\mathcal U}$ denote the equivalence class of $x$ under this identification, and for any two equivalence classes $(x)_{\mathcal U},(y)_{\mathcal U}$, we define  $d((x)_{\mathcal U},(y)_{\mathcal U})$ as $d(x,y)$. 
The resulting metric space is called the \emph{$\mathcal{U}$-ultraproduct} of the family $(X_\lambda,d_\lambda)_{\lambda\in \lambda}$. 
It will be denoted $(\prod_{\lambda\in\Lambda} X_\lambda)_{\mathcal U}$. 

If the spaces $(X_\lambda,d_\lambda)$ are identical to the same space $(X,d)$ for $\lambda\in\Lambda$, the $\mathcal{U}$-ultraproduct $(\prod_{\lambda\in\Lambda} X_\lambda)_{\mathcal U}$ is called the \emph{$\mathcal{U}$-ultrapower} of $(X,d)$, denoted $(X)_{\mathcal U}$. 
Note that the map from $X$ into $(X)_{\mathcal U}$ that assigns to each $x\in X$ the equivalence class of the constant family $(x\mid\lambda\in\Lambda)$ is an isometric embedding. 
When the ultrafilter $\mathcal{U}$ is principal or the space $(X,d)$ is compact this map is  surjective, though it is not so in general. 
The verification of these statements is left to the reader. 

In the definition of ultraproduct, we lifted the metrics from the family $(X_\lambda,d_\lambda)_{\lambda\in \Lambda}$ to $(\prod_{\lambda\in\Lambda} X_\lambda)_{\mathcal U}$ by taking $\mathcal{U}$-limits. 
Doing the same for more general functions requires additional hypotheses. 
Let us introduce the concept of \emph{uniform family of functions}.

\begin{definition}
Suppose that $(X, d, a)$ and $(Y,\rho, b)$ are pointed pseudometric spaces,  $B$ is a subset of $X$, and $F:X\to Y$ is uniformly continuous and bounded on $B$.
\begin{enumerate}
\item
A \emph{bound for $F$ on $B$} is a number $\Omega\ge 0$ such that
\[
x\in B \quad\Rightarrow\quad F(x)\in B_Y(\Omega).
\]
\item
A \emph{modulus of uniform continuity for $F$ on $B$} is a function $\Delta:(0,\infty)\to[0,\infty)$ such that, for all $x,y\in B$ and $\epsilon > 0$,
\[
d(x,y) < \Delta(\epsilon) \quad\Rightarrow\quad \rho(F(x),F(y))\le\epsilon.
\]
\end{enumerate}

\end{definition}

\begin{definition}
\label{D:uniform class}
Let $L$ be a signature and let  $\sC$  be a class of $L$-structures. 
We will say that $\sC$  is a \emph{uniform class} if the following two conditions hold for every function symbol $f:s_1\times\dots\times s_n \to s_0$ of $L$ and every $r>0$:
\begin{enumerate}
\item
(Local equiboundedness condition for $\sC$.) 
There exists $\Omega=\Omega_{f,r}\in[0,\infty)$ such that, for every structure $\mathcal M$ of $\sC$, the number $\Omega$ is a bound for $f^{\mathcal{M}}$ on  $B_{M_\lambda^{(s_1)}}(r)\times\dots\times B_{M_\lambda^{(s_n)}}(r)$.

\item
(Local equicontinuity condition for $\sC$.) 
There exists $\Delta=\Delta_{f,r}:(0,\infty)\to[0,\infty)$ such that for every structure $\mathcal M$ of $\sC$, the function $\Delta$ is a modulus of uniform continuity for $f^{\mathcal{M}}$ on $B_{M_\lambda^{(s_1)}}(r)\times\dots\times B_{M_\lambda^{(s_n)}}(r)$.

Any collection $(\Omega_{r,f},\Delta_{r,f} \mid r>0)$ will be called a family of \emph{moduli of local uniform continuity} for~$f$. 
A collection $\bU = (\Omega_{r,f},\Delta_{r,f} \mid r>0, f\in \mathbf{F})$, with $f$ ranging over the collection $\mathbf{F}$ of function symbols of~$L$, will be called a \emph{modulus of uniformity} for $L$-structures.
\end{enumerate}
\end{definition}

\begin{remark} Clearly, any single $L$-structure $\cM$ admits some modulus of uniformity~$\bU$; 
however, no single such~$\bU$ is a modulus of uniformity for every $L$-structure. 
This is quite analogous to the fact that every Cauchy sequence is metastable with some uniform rate~$\Eb$, but no single such rate of uniform metastability applies to all Cauchy sequences (refer to the discussion in Section~\ref{sec:metastability}-\ref{sec:motivation}).
\end{remark}

\label{P:ultraproducts of functions}
Let $\sC$  be a uniform class of $L$-structures. 
Let $(\mathcal{M}_\lambda)_{\lambda\in \Lambda}$ be family of structures in $\sC$ such that  $\mathcal{M}_\lambda$ is based on $(M_\lambda^{(s)} \mid s\in\mathbf{S})$ for each $\lambda\in \Lambda$. 
If $f:s_1\times\dots\times s_n \to s_0$ is a function symbol of $L$, then for any ultrafilter $\mathcal{U}$ on $\Lambda$ we define a function
\begin{equation*}
\Big( \prod_{\lambda\in\Lambda} f^{\mathcal{M}_\lambda} \Big)_{\mathcal U} : 
\Big( \prod_{\lambda\in\Lambda} M^{(s_1)}_\lambda \Big)_{\mathcal U} 
\times\dots\times
\Big( \prod_{\lambda\in\Lambda} M^{(s_n)}_\lambda \Big)_{\mathcal U} 
\to
\Big( \prod_{\lambda\in\Lambda} M^{(s_0)}_\lambda \Big)_{\mathcal U} 
\end{equation*}
naturally as follows:  
If $(x_\lambda^{i})_{\lambda\in\Lambda}\in\ell^\infty(M_\lambda^{(s_i)},d_\lambda^{(s_i)})_{\Lambda}$ for $i = 1,\dots,n$, we let
\begin{equation}
\label{D:ultraproduct extension}
\Big( \prod_{\lambda\in\Lambda} f^{\mathcal{M}_\lambda} \Big)_{\mathcal U}
\big(\, \big((x_\lambda^{1})_{\lambda\in\Lambda}\big)_{\mathcal U},\dots,\big((x_\lambda^{m})_{\lambda\in\Lambda}\big)_{\mathcal U}\, \big)=
\big(\,(f(x_\lambda^{1},\dots,x_\lambda^{m}))_{\lambda\in\Lambda}\,\big)_{\mathcal U}.
\end{equation}
The uniformitwy of $\sC$ implies that if $(x_\lambda^{i})_{\lambda\in\Lambda} \in B_r\big(\ell^\infty(M_\lambda^{(s_i)},d_\lambda^{(s_i)})_{\Lambda}\big)$, then $(f(x_\lambda^{1},\dots,x_\lambda^{m}))_{\lambda\in\Lambda} \in B_{\Omega}\big(\ell^\infty(M_\lambda^{(s_i)},d_\lambda^{(s_i)})_{\Lambda}\big)$ for some $\Omega>0$, hence the right-hand side of~\eqref{D:ultraproduct extension} is an element of $(\prod_{\lambda\in\Lambda} X_\lambda)_{\mathcal U}$. 
Thus, if $\Omega$ is a uniform bound for $f^{\mathcal{M}_\lambda}$ on $B_{M_\lambda^{(s_1)}}(r_1)\times\dots\times B_{M_\lambda^{(s_n)}}(r_n)$ for all $\lambda\in\Lambda$, then $\Omega$ is also a bound for $( \prod_{\lambda\in\Lambda} f^{\mathcal{M}_\lambda} )_{\mathcal U}$ on $B_{(\prod_{\lambda\in\Lambda} M^{(s_1)}_\lambda)_{\mathcal U} }(r_1)\times\dots\times B_{(\prod_{\lambda\in\Lambda} M^{(s_n)}_\lambda)_{\mathcal U} }(r_n)$.
It is also trivial to verify that, if $\Delta$ is a uniform continuity modulus for $f^{\mathcal{M}_\lambda}$ on $B_{M_\lambda^{(s_1)}}(r_1)\times\dots\times B_{M_\lambda^{(s_n)}}(r_n)$ for all $\lambda\in\Lambda$, then  $\Delta$ is also a modulus of uniform continuity for $( \prod_{\lambda\in\Lambda} f^{\mathcal{M}_\lambda} )_{\mathcal U}$ on $B_{( \prod_{\lambda\in\Lambda} M^{(s_1)}_\lambda)_{\mathcal U} }(r_1)\times\dots\times B_{( \prod_{\lambda\in\Lambda} M^{(s_n)}_\lambda)_{\mathcal U} }(r_n)$.  
Thus, equation~\eqref{D:ultraproduct extension} defines the interpretation of functions in $(\prod_{\lambda\in\Lambda} X_\lambda)_{\mathcal U}$ well. 
The following proposition summarizes the preceding discussion.

\begin{proposition}
\label{prop:ultraproduct}
Let $\sC$ be a uniform class of $L$-structures and let  $(\mathcal{M}_\lambda)_{\lambda\in \Lambda}$ be a family of structures in $\sC$ such that for each $\lambda\in \Lambda$ the structure $\mathcal{M}_\lambda$ is based on $(M_\lambda^{(s)} \mid s\in\mathbf{S})$. 
If $\mathcal{U}$ is an ultrafilter on $\Lambda$, we obtain an $L$-structure $(\prod_{\lambda\in\Lambda} \mathcal{M}_\lambda)_{\mathcal U}$ based on  $(\,(\prod_{\lambda\in\Lambda} M^{(s)}_\lambda)_{\mathcal U} \mid s\in\mathbf{S}\,)$ by interpreting any function symbol $f$ of~$L$ in $( \prod_{\lambda\in\Lambda} \mathcal{M}_\lambda)_{\mathcal U}$ as $(\prod_{\lambda\in\Lambda} f^{\mathcal{M}_\lambda} )_{\mathcal U}$.

Furthermore, any modulus of uniformity for~$\mathscr C$ is also a modulus of uniformity for $(\prod_{\lambda\in\Lambda} \mathcal{M}_\lambda)_{\mathcal U}$.
\end{proposition}

\begin{definition}
  \label{D:ultraproduct}
The structure $(\prod_{\lambda\in\Lambda} \mathcal{M}_\lambda)_{\mathcal U}$ in Proposition~\ref{prop:ultraproduct} is called the \emph{$\cU$-ultraproduct} of the family $(\mathcal{M}_\lambda)_{\lambda\in \Lambda}$.
\end{definition}

Note that the hypothesis that a class~$\sC$ is uniform asserts that the collection $(f^{\cM}\mid\cM\in\sC)$ of interpretations of a given function symbol $f$ is an equicontinuous and equibounded family, precisely as in the statement of the Arzelà-Ascoli theorem. 
This classical result is thus subsumed under the fact that any ultraproduct of a uniform family of $L$-structures is itself an $L$-structure.
Just as the hypotheses of equicontinuity and equiboundedness are both necessary for the conclusion of the Arzelà-Ascoli theorem to hold, it should be clear that ultraproduct structures are defined, in general, only for subfamilies of some \emph{uniform} class of structures---otherwise, the right-hand side of equation~\eqref{D:ultraproduct extension} may fail to define an element lying at finite distance from the anchor $\big((a_{\lambda})_{\lambda\in\Lambda}\big)_{\cU}$, or else the function so defined may fail to be continuous. 

\subsection{Henson languages and semantics: Formulas and satisfaction}
\label{sec:formulas}
We now focus our attention on the precise connection between metric structures and their ultrapowers and, more generally, between families of metric structures and their ultraproducts. 
These connections are intimately connected to notions from model theory, a branch of mathematical logic.

In the current literature, there are two formally different but equivalent logical frameworks to study metric structures from a model-theoretic perspective. 
One of these frameworks is that of \emph{continuous model theory}~\cite{Ben-Yaacov-Usvyatsov:2010,Ben-Yaacov-Berenstein-Henson-Usvyatsov:2008}, which uses real-valued logic, and the other is the logic of approximate truth, introduced in the 1970s by C.~W.~Henson~\cite{Henson:1975,Henson:1976} and developed further by Henson and the second author~\cite{Henson-Iovino:2002,Iovino:2014}. 
We have adopted the latter because, despite its less widespread use in the current literature, it has strong syntactic advantages, as it allows dealing with unbounded metric structures such as Banach spaces in a natural fashion, without having to replace the metric by an equivalent bounded metric.%
\footnote{For a proof of the equivalence among various formulations, see~\cite{Iovino:2001,Iovino:2009,Caicedo-Iovino:2014}.}

Let $L$ be a fixed signature. 
In analogy with languages of traditional (discrete) first-order logic, we construct a language, called a \emph{Henson language,} which is suitable for discussing properties of metric structures. 
The language consists of syntactic expressions called \emph{positive bounded formulas} of~$L$, or $L$-formulas. 
These are strings or symbols built from a basic alphabet that includes the following symbols:
\begin{itemize}
\item
The function symbols of the signature $L$,
\item
For each sort index $s\in\mathbf{S}$ of $L$, a countable collection of symbols called \emph{variables of sort $s$}, or variables \emph{bound to the sort $s$}.
\item
Logical connectives $\lor$ and $\land$, and for each positive rational number $r$, quantifiers $\forall\!_r$ and $\exists_r$.
\item
Parentheses and the comma symbol. 
\end{itemize}

First we define the concept of $L$-term.  
Intuitively, a term is a string of symbols that may be interpreted by elements of $L$-structures.  Since elements of structures occur inside sorts, each term must have a sort associated with it. 
Thus  we define the concept of \emph{$s$-valued term}:

\begin{definition}
An \emph{$s$-valued $L$-term} is any finite string of symbols that can be obtained by finitely many applications of the following rules of formation:
\begin{enumerate}
\item
Every variable of sort $s$ is an $s$-valued term,
\item
If $f$ is a function symbol with $f:s_1\times\dots\times s_n \to s$ and $t_1,\dots,t_n$ are  such that $t_i$ is an $s_i$-valued for $i=1,\dots, n$, then $f(t_1, \dots, t_n)$ is an $s$-valued term.
\end{enumerate}

If $t$ is a term and $x_1,\dots, x_n$ is a list of variables that contains all the variables occurring in $t$, we write $t$ as $t(x_1,\dots, x_n)$.

A \emph{real-valued term} is an $s_\mathbb{R}$-valued term. 
A \emph{term} is string that is an $s$-valued term for some $s\in\mathbf{S}$.

\end{definition}

\begin{definition}
Let $\mathcal{M}$ be an $L$-structure based on $(M^{(s)} \mid s\in\mathbf{S})$ and let $t(x_1,\dots,x_n)$ be an $L$-term, where $x_i$ is a variable of sort $s_i$, for $i=1,\dots, n$. 
If $a_1,\dots, a_n$ are elements of $ \mathcal{M}$ such that $a_i$ is of sort $s_i$, for $i=1,\dots, n$, the \emph{evaluation} of $t$ in $\mathcal M$ at $a_1,\dots,a_n$, denoted $t^{\mathcal{M}}[a_1,\dots, a_n]$, is defined by induction on the length of $t$ as follows:
\begin{enumerate}
\item
If $t$ is $x_i$, then $t^{\mathcal{M}}[a_1,\dots, a_n]$ is $a_i$,
\item
If $t$ is $f(t_1,\dots, t_n)$, where $f$ is a function symbol and $t,\dots, t_n$ are terms of lower length, then $t^{\mathcal{M}}[a_1,\dots, a_n]$ is 
\[
f^{\mathcal{M}}(t_1^{\mathcal{M}}[a_1,\dots, a_n],\dots, t_n^{\mathcal{M}}[a_1,\dots, a_n]).
\]
\end{enumerate}
As an addendum to~(2), by a slight abuse of notation, if $f$ is a nullary $M^{(s)}$-valued function symbol (i.e., a constant symbol) we usually interpret $f$ as be the (unique) element $a\in M^{(s)}$ in the range of the function~$f^{\cM}$.
\end{definition}

\begin{notation}
\label{R:polynomial terms}
Recall from the definition of signature that every signature $L$ must include a special sort index  $s_\mathbb{R}$ and constant symbol for each rational number. 
Informally we will identify each rational number with its constant symbol in~$L$. 
More generally, since $L$ includes function symbols for the addition and multiplication in $\mathbb{R}$, for every polynomial $p(x_1,\dots,x_n)\in\mathbb{Q}[x_1,\dots,x_n]$ there exists a real-valued $L$-term $t(x_1,\dots,x_n)$ such that $t^{\mathcal{M}}[a_1,\dots,a_n] =p(a_1, \dots,a_n)$ for any $L$-structure $\mathcal{M}$ and $a_1,\dots,a_n\in\mathbb{R}$. 
We will identify $t$ and $p$. 
Thus, if $t_1,\dots, t_n$ are $L$ terms and $p(x_1,\dots,x_n)\in\mathbb{Q}[x_1,\dots,x_n]$, we may refer to the $L$-term $p(t_1,\dots, t_n)$.
\end{notation}

\begin{definition}
A \emph{positive bounded $L$-formula} (or simply an \emph{$L$-formula}) is any finite string of symbols that can be obtained by finitely many applications of the following rules of formation:
\begin{enumerate}
\item
If $t$ is a real-valued $L$-term and $r$ is a rational number, then the expressions
\[
t \le r\qquad\text{and} \qquad t\ge r
\]
are $L$-formulas.  
These are the \emph{atomic} $L$-formulas.
\item
If $\varphi$ and $\psi$ are positive bounded $L$-formulas, then the expressions
\[
(\varphi\land\psi) \qquad\text{and}\qquad (\varphi\lor\psi)
\] 
are positive bounded $L$-formulas. 
These are the \emph{conjunction} and \emph{disjunction}, respectively, of $\varphi$ and $\psi$.
\item 
If $\varphi$ is positive bounded $L$-formula, $r$ is a positive rational, and $x$ is a variable, then the expressions
\[
\exists_r  x\,\varphi \qquad\text{and}\qquad \forall\!_r x\,\varphi
\] 
are positive bounded $L$-formulas. 
\end{enumerate}
\end{definition}

\begin{notation}
Whenever possible, we shall omit parentheses according to the usual syntactic simplification rules.
If $t$ is a real-valued term and $r_1,r_2$ are rational numbers, we will write $r_1\le t \le r_2$  as an abbreviation of the conjunction $(r_1\le t\land t\leq r_2)$. 
Similarly, we regard $t=r$ as an abbreviation of the conjunction $(t\leq r\land t\geq r)$.  
If $t_1$ and $t_2$ are real-valued terms, we regard $t_1\le t_2$ as an abbreviation of $0\le t_2-t_1$ and, if $t_1,t_2$ are $s$-valued terms, we regard the expression $t_1=t_2$ as an abbreviation of $d(t_1,t_2)\le 0$, where $d$ is the function symbol designating the metric of the sort indexed by $s$. 
If $\varphi_1,\dots,\varphi_n$ are formulas, we may write $\bigwedge_{i=1}^n \varphi_i$ and $\bigvee_{i=1}^n \varphi_i$ as abbreviations of  $\varphi_1\land\dots\land\varphi_n$ and $\varphi_1\lor\dots\lor\varphi_n$, respectively. 
If $t$ is an $s$-valued term and $d,a$ are the designated function symbol and constant symbol, respectively, for the metric and the anchor of this sort, we shall regard the expression $t\in B_r$ as an abbreviation of the formula $d(t,a)\le r$.
\end{notation}

\begin{definition}
A \emph{subformula} of a formula $\varphi$ is a substring of $\varphi$ that is itself a formula. 
If $\varphi$ is a formula and $x$ is a variable, we say that $x$ occurs \emph{free} in $\varphi$ if there is at least one occurrence of $x$ in $\varphi$ that is not within any subformula of the form $\forall\!_r x\,\varphi$ or $\exists_r x\,\varphi$.  
If $x_1,\dots, x_n$ are variables, we write $\varphi$ as $\varphi(x_1,\dots, x_n)$ if all the free variables of $\varphi$ are among $x_1,\dots, x_n$. 
A positive bounded $L$-\emph{sentence} is a positive bounded formula without any free variables.
\end{definition}

The definition below introduces the most basic concept of model theory, namely, the satisfaction relation $\models$ between structures and formulas. Intuitively, if $\mathcal{M}$ is an $L$ structure,  $\varphi(x_1,\dots, x_n)$ is an $L$- formula and $a_1,\dots,a_n$ are elements of $\mathcal{M}$,
\[
\mathcal{M}\models\varphi[a_1,\dots, a_n]
\]
means that $\varphi$ is true $\mathcal{M}$ if $x_1,\dots x_n$ are interpreted as $a_1,\dots, a_n$, respectively. 
Evidently, for this to be meaningful, the variable $x_i$ must be of the same sort as the element $a_i$, for $i=1,\dots, n$.

\begin{definition}
Let $\mathcal{M}$ be an $L$-structure based on $(M^{(s)} \mid s\in\mathbf{S})$ and let $\varphi(x_1,\dots,x_n)$ be an $L$-formula, where $s_i$ is a variable of sort $s_i$, for $i=1,\dots, n$. 
If $a_1,\dots, a_n$ are elements of $ \mathcal{M}$ such that $a_i$ is of sort $s_i$, for $i=1,\dots, n$, the \emph{(discrete) satisfaction relation} $\mathcal{M}\models\varphi[a_1,\dots, a_n]$ is defined inductively as follows:
\begin{enumerate}
\item
If $\varphi(x_1,\dots,x_n)$ is $t\le r$, where $t=t(x_1,\dots, x_n)$ is a real-valued term and $r$ is rational, then
$\mathcal{M}\models\varphi[a_1,\dots, a_n]$ if and only if $t^{\mathcal{M}}[a_1,\dots, a_n]\le r$.
\item
If $\varphi(x_1,\dots,x_n)$ is $t\ge r$, where $t=t(x_1,\dots, x_n)$ is a real-valued term and $r$ is rational, then $\mathcal{M}\models\varphi[a_1,\dots, a_n]$ if and only if $t^{\mathcal{M}}[a_1,\dots, a_n]\ge r$.
\item
If $\varphi(x_1,\dots,x_n)$ is $(\psi_1\land\psi_2)$, where $\psi_1$ and $\psi_2$ are $L$-formulas, then
$\mathcal{M}\models\varphi[a_1,\dots, a_n]$ if and only if 
\[
\mathcal{M}\models\psi_1[a_1,\dots, a_n]
\quad\text{and}\quad
\mathcal{M}\models\psi_2[a_1,\dots, a_n].
\]
\item
If $\varphi(x_1,\dots,x_n)$ is $(\psi_1\lor\psi_2)$, where $\psi_1$ and $\psi_2$ are $L$-formulas, then
$\mathcal{M}\models\varphi[a_1,\dots, a_n]$ if and only if 
\[
\mathcal{M}\models\psi_1[a_1,\dots, a_n]
\quad\text{or}\quad
\mathcal{M}\models\psi_2[a_1,\dots, a_n].
\]
\item
If $\varphi(x_1,\dots,x_n)$ is $\exists_r x \, \psi(x, x_1,\dots,x_n)$, where $r$ is a positive rational, $x$ is a variable of sort $s$, and $\psi(x, x_1,\dots,x_n)$ is an $L$-formula, then
$\mathcal{M}\models\varphi[a_1,\dots, a_n]$ if and only if 
\[
\mathcal{M}\models\psi[a,a_1,\dots, a_n]
\quad\text{for some $a\in B_{M^{(s)}}[r]$}.
\]
\item
If $\varphi(x_1,\dots,x_n)$ is $\forall\!_r x\, \psi(x, x_1,\dots,x_n)$, where $r$ is a positive rational, $x$ is a variable of sort $s$, and $\psi(x, x_1,\dots,x_n)$ is an $L$-formula, then
$\mathcal{M}\models\varphi[a_1,\dots, a_n]$ if and only if 
\[
\mathcal{M}\models\psi[a,a_1,\dots, a_n]
\quad\text{for every $a\in B_{M^{(s)}}(r)$}.
\]
\end{enumerate}
If $\mathcal{M}\models\varphi[a_1,\dots, a_n]$, we say that $\mathcal{M}$ \emph{satisfies} $\varphi$ at $a_1,\dots, a_n$.
\end{definition}
Note that universal (resp., existential) quantification is interpreted in open (resp., closed) balls.

\begin{definition}
If $\Phi$ is a set of formulas, we denote it by $\Phi(x_1,\dots, x_n)$ if all the free variables of all the formulas in $\Phi$ are among $x_1,\dots, x_n$. 
If  $\mathcal{M}$ is a structure and $a_1,\dots, a_n$ are elements of $\mathcal{M}$, we write $\mathcal{M}\models\Phi[a_1,\dots, a_n]$ if $\mathcal{M}\models\varphi[a_1,\dots, a_n]$ for every $\varphi\in\Phi$.

\end{definition}

\subsection{Approximations and approximate satisfaction}
\label{sec:approximations}
We begin this subsection by defining a strict partial ordering of positive bounded $L$-formulas, namely the relation \emph{``$\psi$ is an approximation of~$\varphi$'',} denoted $\varphi\appby\psi$ (or $\psi\appto\varphi$). 
Roughly speaking, this means that $\psi$ arises when every estimate occurring in~$\varphi$ is relaxed. 
The formal definition of the approximation relation is by induction on the complexity of $\varphi$, as given by the following table.

\begin{center}
  \renewcommand{\arraystretch}{1.50}
\begin{tabular}{lll}
\underline{If $\varphi$ is:} &\qquad 
&\underline{The approximations of $\varphi$ are:}\\
$t \le r$ &\qquad  &$t \le r'$\quad where $r'>r$\\
$t \ge r$ &\qquad  &$t \ge r'$\quad where $r'<r$\\
$(\xi\land\psi)$ &\qquad  &$(\xi'\land\psi')$\quad
where $\xi \appby \xi'$ and $\psi \appby \psi'$\\
$(\xi\lor\psi)$ &\qquad  &$(\xi'\lor\psi')$\quad
where $\xi \appby \xi'$ and $\psi \appby \psi'$\\
$\exists_r x\, \psi$ &\qquad
&$\exists_{r'} x\, \psi'$\quad where $\psi \appby \psi'$ and $r'>r$\\
$\forall\!_r x\, \psi$ &\qquad
&$\forall\!_{r'} x\, \psi'$\quad where $\psi \appby \psi'$ and $r'<r$
\end{tabular}
\end{center}

\begin{definition}
Let $\mathcal{M}$ be an $L$-structure based on $(M^{(s)} \mid s\in\mathbf{S})$ and let $\varphi(x_1,\dots,x_n)$ be an $L$-formula, where $x_i$ is a variable of sort $s_i$, for $i=1,\dots, n$. 
If $a_1,\dots, a_n$ are elements of $ \mathcal{M}$ such that $a_i$ is of sort $s_i$, for $i=1,\dots, n$, we say that $\mathcal{M}$ \emph{approximately satisfies}  $\varphi$ at $a_1,\dots, a_n$, and write
\begin{equation*}
\mathcal{M}\appmod\varphi[a_1,\dots, a_n],
\end{equation*}
if
\begin{equation*}
\mathcal{M}\models\varphi'[a_1,\dots, a_n], \quad
\text{for every $\varphi'\appto\varphi$}.
\end{equation*}
If $\Phi(x_1,\dots,x_n)$ is a set of formulas, we say that $\mathcal{M}$ \emph{approximately satisfies}  $\Phi$ at $a_1,\dots, a_n$, and write $\mathcal{M}\appmod\Phi[a_1,\dots, a_n]$, if $\mathcal{M}\appmod\varphi[a_1,\dots, a_n]$ for every $\varphi\in\Phi$.
\end{definition}

Clearly, approximate satisfaction is a weaker notion of truth than discrete satisfaction.
For nondiscrete metric space structures, the approximate satisfaction relation $\appmod$ is the ``correct'' notion of truth, in the sense that for these structures it is not the notion of discrete satisfaction, but rather that of approximate satisfaction, that yields a well-behaved model theory. 
(For discrete structures, the two relations are clearly equivalent.)

The negation connective (``$\neg$'') is not allowed in positive bounded formulas.  
However, for every positive bounded formula $\varphi$ there is a positive bounded formula $\wneg\varphi$, called the \emph{weak negation of $\varphi$}, that plays a role analogous to that played by the negation of $\varphi$.  

\begin{definition}
  The unary pseudo-connective $\wneg$ of \emph{weak negation} of $L$-formulas is defined recursively as follows.
\begin{center}
  \renewcommand{\arraystretch}{1.50}
  \begin{tabular}{lll}
    \underline{If $\varphi$ is:} &\qquad 
    &\underline{$\wneg\varphi$ is:}\\
    $t\le r$ &\qquad  &$t \ge r$\\
    $t \ge r$ &\qquad  &$t \le r$\\
    $(\xi\land\psi)$ &\qquad  &$(\wneg\xi\lor\wneg\psi)$\\
    $(\xi\lor\psi)$ &\qquad  &$(\wneg\xi\land\wneg\psi)$\\
    $\forall\!_r x \, \psi$ &\qquad
    &$\exists_r x \wneg\psi$\\
    $\exists_r x \, \psi$ &\qquad
    &$\forall\!_r x \wneg\psi$.\\
  \end{tabular}
\end{center}
\end{definition}

\begin{remarks}\hfill
\label{R:approximations}
\begin{enumerate}
\item If $\varphi,\varphi'$ are positive bounded formulas, then $\varphi\appby\varphi'$ if and only if $\wneg\varphi'\appby\wneg\varphi$. 
\item Although languages for metric structures do not include a connective interpreted as the implication ``\emph{if $\varphi$ then $\psi$}'', a formula of the form ``$\wneg\varphi \vee \psi$'' may be regarded as a weak conditional.
\item If $\mathcal{M}$ is an $L$-structure and $\varphi(x_1,\dots,x_n)$ is a positive bounded $L$-formula such that $\mathcal{M} \not\models \varphi[a_1,\dots,a_n]$, then $\mathcal{M} \models \wneg\varphi[a_1,\dots,a_n]$. If $\varphi'$ is an approximation of $\varphi$ such that $\mathcal{M} \models \wneg\varphi'[a_1,\dots,a_n]$, then $\mathcal{M} \not\models \varphi[a_1,\dots,a_n]$.
\end{enumerate}
\end{remarks}

\begin{proposition}
\label{P:negation}
Let $\mathcal{M}$ be an $L$-structure based on $(M^{(s)} \mid s\in\mathbf{S})$, let $\varphi(x_1,\dots,x_n)$ be an $L$-formula, where $x_i$ is a variable of sort $s_i$, for $i=1,\dots, n$, and let $a_1,\dots, a_n$ be elements of $\mathcal{M}$ such that $a_i$ is of sort $s_i$, for $i=1,\dots, n$. 
Then, $\cM\not\appmod \varphi[a_1,\dots,a_n]$ if and only there exists a formula $\varphi'\appto\varphi$ such that $\cM\appmod\wneg\varphi'[a_1,\dots,a_n]$.
\end{proposition}

\begin{proof}
  In order to simplify the nomenclature, let us suppress the lists $x_1,\dots,x_n$ and $a_1,\dots,a_n$ from the notation.
    
  If $\mathcal{M}\not\appmod \varphi$, there exists $\varphi'\appto\varphi$ such that $\mathcal{M}\not\models \varphi'$.  
We have $\mathcal{M}\models \wneg\varphi'$ and hence $\mathcal{M}\appmod \wneg\varphi'$.  
Conversely, assume that there exists $\varphi'\appto\varphi$ such that $\mathcal{M} \appmod \wneg\varphi'$.  
Take sentences $\psi,\psi'$ such that $\varphi\appby\psi\appby\psi'\appby\varphi'$.  
Then $\mathcal{M}\models\wneg\psi'$ (by Remark~\ref{R:approximations}-(3)) and hence $\mathcal{M} \not\models\psi$, so $\mathcal{M} \not\appmod \varphi$.
\end{proof}

\subsection{Theories, elementary equivalence, and elementary substructures}
\label{sec:substructures}
In this section we include some basic definitions from model theory.

\begin{definition}
	\label{D:model}
	Let $L$ be a signature. 
	
	\begin{enumerate}
		\item
		An \emph{$L$-theory} (or simply a \emph{theory}) is a set of $L$-sentences. 
		\item
		If $T$ is a theory and $\mathcal{M}\appmod T$, we say that $\mathcal{M}$ is a \emph{model} of $T$. 
A theory is \emph{satisfiable} (or \emph{consistent}) if it has a model.
The class of all models of a theory~$T$ is denoted $\Mod(T)$.
If $\mathscr{C}$ is any class of $L$-structures, the class $\mathscr{C}\cap\Mod(T)$ of models of~$T$ in~$\mathscr{C}$ is denoted $\Mod_{\mathscr{C}}(T)$.
              \item An \emph{axiomatizable class $\sC$} (or \emph{elementary class}) is one that consists of all the models of a fixed theory~$T$. 
We say that $\sC$ \emph{is $L$-axiomatizable}, or \emph{$\sC$ is axiomatizable by~$T$,} when the language or the theory need to be specified.

              \item An $L$-theory~$T$ is \emph{uniform} if the class of all models of~$T$ is uniform.
(See Proposition~\ref{prop:uniformity-axiomatizable}.)

		\item
		The \emph{complete $L$-theory of a structure~$\mathcal{M}$}, denoted $\thry{(\mathcal{M})}$, is the set of all $L$-sentences $\varphi$ such that $\mathcal{M}\appmod\varphi$. 
A \emph{complete $L$-theory} is the complete $L$-theory of any $L$-structure~$\cM$.
\item 		The \emph{complete $L$-theory of class $\sC$ of $L$-structures} is $\thry{(\sC)} = \bigcap_{\cM\in\sC}\thry(\cM)$. 
		\item
		Two $L$-structures $\mathcal{M},\mathcal{N}$ are \emph{elementarily equivalent}, written $\mathcal{M}\equiv\mathcal{N}$, if they have the same complete theory, i.e., if
		\[
			\mathcal{M}\appmod\varphi \quad\Leftrightarrow\quad\mathcal{N}\appmod\varphi,\qquad
\text{for every $L$-sentence $\varphi$.}
		\]
		
		\item
		 If $\mathcal{M}$ and $\mathcal{N}$ are  $L$-structures and  $\mathcal{M}$ is a substructure of $\mathcal{N}$, we say that $\mathcal{M}$ is an \emph{elementary substructure} of $\mathcal{N}$, and we write $\mathcal{M}\prec\mathcal{N}$, if whenever $a_1,\dots, a_n$ are elements of $\mathcal{M}$ and $\varphi(x_1,\dots, x_n)$ is an $L$-formula such that $a_i$ is of the same sort as $x_i$, for $i=1,\dots, n$, we have
		\[
		\mathcal{M}\appmod\varphi[a_1,\dots, a_n] \quad\Leftrightarrow\quad\mathcal{N}\appmod\varphi[a_1,\dots, a_n].
		\]
	\end{enumerate}
\end{definition}

\begin{remarks}\label{rem:ThM-uniform}
  \begin{enumerate}
  \item When $\cM$ is an $L$-structure, the interpretation of each function symbol is locally bounded and locally uniformly continuous by definition, hence $\thry(\cM)$ is necessarily a uniform theory.
  \item Any satisfiable theory $T$ admits some extension to a uniform theory $T'=\thry(\cM)$, where $\cM$ is any model of~$T$. 
Neither the extension nor a modulus of uniformity thereof is uniquely determined by~$T$ in general.
  \item If $T'$ extends a $\bU$-uniform theory~$T$, then $T'$ is also $\bU$-uniform. 
  \end{enumerate}
\end{remarks}

\begin{proposition}
  \label{prop:uniformity-axiomatizable}
Let $L$ be a signature. 
Given a modulus of uniformity $\bU$, the class of $L$-structures  $\cM$ such that $\bU$ is a modulus of uniformity for~$\cM$ is axiomatizable.
\end{proposition}
\begin{proof}
Let $\bU = (\Omega_{r,f},\Delta_{r,f} \mid r>0, f\in \mathbf{F})$ be a modulus of uniformity.
For a  function symbol~$f : s_1\times\dots\times s_n \to s_0$ and rational numbers $u,v,w>0$, consider the $L$-sentences
  \begin{align*}
    \chi_{f,u,v} \ : \ \ &\forall\!_ux_1\dots\forall\!_ux_n \big(d(f(x_1,\dots,x_n),a_0) \le v\big),\\
\xi_{f,u,v,w} \ : \ \ &\forall\!_ux_1\forall\!_uy_1\dots\forall\!_ux_n\forall\!_uy_n 
\big( \overline{d}(\overline{x},\overline{y}) \ge v \,\vee\, d(f(\overline{x}),f(\overline{y})) \le w \big),
  \end{align*}
where $\overline{x},\overline{y}$ denote the $n$-tuples $x_1,\dots,x_n$ and $y_1,\dots,y_n$, respectively,  $\overline{d}$ denotes the supremum distance on $\cM^{(s_1)}\times\dots\times\cM^{(s_n)}$, and $a_0$ the constant symbol for the anchor of the sort indexed by~$s_0$.
It should be clear that  for any signature $L$, the class of $L$-structures  $\cM$ such that $\bU$ is a modulus of uniformity for~$\cM$ is axiomatized by the union of the following sets of sentences:
\begin{align*}
   & \{\, \chi_{f,u,v} \mid \text{$f\in \mathbf{F}$, $u,v\in\QQ_+$, $u<r$, $v>\Omega_{r,f}$ for some $r\in(0,\infty)$}\, \}, \\
    &\{\, \xi_{f,u,v,w} \mid \text{$f\in \mathbf{F}$, $u,v\in\QQ_+$, $u<r$, $v<t$, $w>\Delta_{r,f}(t)$ for some $r,t\in(0,\infty)$}\, \}. \qedhere
\end{align*}
\end{proof}

It is easy to construct examples showing that the meta-property ``\emph{the class~$\sC$ is uniform}'' (without specifying a modulus of uniformity) is not axiomatizable in general.

The following proposition gives simple conditions to verify $\equiv$ and $\prec$. 

\begin{proposition}
\label{P:elementary}
Let $\mathcal{M}$ and $\mathcal{N}$ be $L$-structures.
\begin{enumerate}[\normalfont(1)]
\item
\label{I:elementary equivalence}
$\mathcal M\equiv\mathcal{N}$ if and only if for every $L$-sentence $\varphi$,
\[
\mathcal M\appmod \varphi \quad\Rightarrow\quad \mathcal{N}\appmod \varphi.
\]
\item
\label{I:Tarski-Vaught}
\emph{(Tarski-Vaught test for $\prec$).}
A substructure $\mathcal{M}$ of the structure $\mathcal{N}$ is an elementary substructure if and only if the following condition holds: 
If $\varphi(x_1,\dots, x_n,y)$ is an $L$-formula, $a_1,\dots, a_n$ are elements of $\mathcal{M}$ with $a_i$ of the same sort as $x_i$  for $i=1,\dots, n$
such that $\mathcal{N}\appmod\exists_ry\,\varphi[a_1,\dots, a_n]$ for some $r>0$, then 
$\mathcal{M}\appmod\exists_ry\,\varphi[a_1,\dots, a_n]$.
\end{enumerate}
\end{proposition}

\begin{proof}
For part (\ref{I:elementary equivalence}), the direct implication follows by definition of elementary equivalence, while the converse follows from Proposition~\ref{P:negation}. 
For (\ref{I:Tarski-Vaught}), the direct implication is trivial, and the converse follows by induction on the complexity of~$\varphi$.
\end{proof}

The following is an immediate consequence of Proposition~\ref{P:elementary}-(\ref{I:Tarski-Vaught}):

\begin{proposition}[Downward L\"owenheim-Skolem Theorem]
\label{P:downward lowenheim-skolem}
Let $\mathcal M$ be a structure and let $A$ be a set of elements of $\mathcal M$. 
Then there exists a substructure $\mathcal{M}_0$ of $\mathcal{M}$ such that
\begin{itemize}
\item
$\mathcal{M}_0\prec\mathcal{M}$,
\item
Every element of $A$ is an element of $\mathcal{M}_0$,
\item
$\card(\mathcal{M}_0)\le\card(A)+\card(L)$.
\end{itemize}
\end{proposition}

\begin{definition}
An \emph{elementary chain} is a family $(\mathcal{M}_i)_{i\in I}$ of structures, indexed by some linearly ordered set~$(I,<)$, such that $\mathcal{M}_i\prec\mathcal{M}_j$ if $i<j$.
\end{definition}

Another useful consequence of Proposition~\ref{P:elementary} is the elementary chain property:

\begin{proposition}[The Elementary Chain Property]
\label{P:elementary chain}
If $(\mathcal{M}_i)_{i\in I}$ is an elementary chain, then $\bigcup_{i\in I} \mathcal{M}_i$ is an elementary extension of $\mathcal{M}_j$ for every $j\in I$.
\end{proposition}

\begin{proof}
By Proposition~\ref{P:elementary}.
\end{proof}

\subsection{\Los' Theorem}
\label{sec:Los}
The following fundamental theorem, proved by J.~\Los\ in the 1950s~\cite{Los:1955}, intuitively states that a formula $\varphi$ is satisfied by an ultraproduct of a family  $(\mathcal{M}_\lambda)_{\lambda\in \Lambda}$ of structures if and only if every approximation of $\varphi$ is satisfied by almost all of the structures $\mathcal{M}_\lambda$. 
\Los\ proved the theorem for discrete structures (i.e., traditional first-order logic), where approximations are not needed; 
however, essentially the same argument holds for arbitrary metric structures.

\begin{theorem}[\Los' Theorem for Metric Structures]
\label{T:Los}
Let $L$ be a signature and let  $(\mathcal{M}_\lambda)_{\lambda\in \Lambda}$ be a family of $L$-structures in a uniform class such that for each $\lambda\in \Lambda$ the structure $\mathcal{M}_\lambda$ is based on $(M_\lambda^{(s)} \mid s\in\mathbf{S})$. 
Let $(a_{1,\lambda})_{\lambda\in\Lambda},\dots, (a_{n,\lambda})_{\lambda\in\Lambda}$ be such that of $(a_{i,\lambda})_{\lambda\in\Lambda}\in\ell^\infty(M_\lambda^{(s_i)})_{\Lambda}$ for $i=1,\dots, n$    and let $\varphi(x_1,\dots, x_n)$ be an $L$-formula such that $x_i$ is of sort $s_i$. 
Then, for any ultrafilter $\mathcal{U}$ on $\Lambda$, 
\[
(\, \prod_{\lambda\in\Lambda} \mathcal{M}_\lambda \,)_{\mathcal U} \appmod
\varphi[\,((a_{1,\lambda})_{\lambda\in\Lambda})_{\mathcal U},\dots, ((a_{n,\lambda})_{\lambda\in\Lambda})_{\mathcal U}\,]
\]
if and only if for every approximation $\varphi'\appto\varphi$, the set
\[
\{\,\lambda\in\Lambda \mid \mathcal{M}_\lambda\models\varphi'[a_{1,\lambda},\dots, a_{n,\lambda}]\,\}
\]
is $\mathcal{U}$-large.

\end{theorem}

\begin{proof}
Using the definition of $\cU$-ultraproduct of structures and the interpretation of function symbols therein (Proposition~\ref{prop:ultraproduct} and Definition~\ref{D:ultraproduct}), \Los' Theorem follows by induction on the complexity of~$\varphi$. 
\end{proof}

An important corollary of \Los' theorem is the special case when all the structures $\mathcal M_\lambda$ equal the same structure $\mathcal M$. 
In this case, the $\mathcal{U}$-ultraproduct $( \prod_{\lambda\in\Lambda} \mathcal{M}_\lambda)_{\mathcal U}$ is the $\mathcal{U}$-ultrapower of $\mathcal{M}$. 
Hence we have the following:

\begin{corollary}
	Every metric structure is an elementary substructure of its ultrapowers.
\end{corollary}

\subsection{Compactness} 
\label{sec:compactness}
The Compactness Theorem is arguably the most distinctive theorem of first-order logic.

For a set $\varphi$ of formulas, let $\varphi_\approx$ denote the set of all approximations of formulas in~$\varphi$.

\begin{theorem}[Compactness Theorem]
	\label{T:compactness}
Let $\sC$ be a uniform class of structures and let $T$ be an $L$-theory. 
If every finite subset of $T_\approx$ has a model in~$\sC$ in the semantics of discrete satisfaction, then $T$ has a model in the semantics of approximate satisfaction. 
This model may be taken to be an ultraproduct of structures in $\sC$ that admits the same modulus of uniformity as the class~$\sC$.
\end{theorem}

\begin{proof}
	Let $\Lambda$ be the set of finite subsets of $T_\approx$, and for each $\lambda$ in $\Lambda$, let $\mathcal{M}_\lambda$ be a model in $\sC$ of all the sentences in~$\lambda$ in the semantics of discrete satisfaction. 
For every finite subset $\varphi$ of $T_\approx$, let $\Lambda_{\supseteq\varphi}$ be the set of all $\lambda\in\Lambda$ such that $\lambda\supseteq \varphi$. 
Then $\mathcal{M}_\lambda\models\varphi$ for every $\lambda\in\Lambda_{\supseteq\varphi}$. 
Note that the collection of subsets of $\Lambda$ of the form $\Lambda_{\supseteq\varphi}$ is closed under finite intersections since $\Lambda_{\supseteq\varphi}\cap\Lambda_{\supseteq\Psi} = \Lambda_{\supseteq\varphi\cup\Psi}$. 
Let $\mathcal{U}$ be an ultrafilter on $\Lambda$ extending this collection. 
Then, by \Los' Theorem (Theorem~\ref{T:Los}), we have $(\prod_{\lambda\in\Lambda} \mathcal{M}_\lambda)_{\mathcal U} \appmod T$. 
Furthermore, $(\prod_{\lambda\in\Lambda} \mathcal{M}_\lambda)_{\mathcal U}$ admits the same modulus of uniformity as the family $(\mathcal{M}_\lambda)_{\lambda\in\Lambda}$, by Proposition~\ref{prop:ultraproduct}.
\end{proof}

The following corollary amounts to a restatement of the Compactness Theorem that does not explicitly mention approximation of formulas.

\begin{corollary}\label{cor:compactness}
If $\Upsilon$ is a uniform theory and $T$ is any collection of sentences such that every finite subset of~$T$ is approximately satisfied by a model of~$\Upsilon$, then $\Upsilon\cup T$ has a model.
\end{corollary}
\begin{proof}
Every finite subset $\varphi$ of $(\Upsilon\cup T)_{\approx}$ must be a subset of $(\Upsilon\cup T')_{\approx}$ for some finite subset $T'\subseteq T$. 
By hypothesis, $\Upsilon\cup T'$ has a model, which is thus also a model of $\varphi$ in the semantics of discrete satisfaction. 
By the Compactness Theorem (Theorem~\ref{T:compactness}), $\Upsilon\cup T$ has a model.
\end{proof}

\begin{remarks}
	\label{R:topology}
Fix a signature~$L$ and a class $\sC$ of $L$-structures.
For any set $\Phi$ of $L$-formulas, let $\Mod_{\sC}(\Phi) = \sC\cap\Mod(\Phi)$ be the subclass of $\sC$ consisting of models of~$\Phi$. 
We topologize $\sC$ by taking as a basis of closed sets the collection of all classes of the form $\Mod_{\sC}(\varphi)$ as $\varphi$ varies over~$L$-formulas; thus, closed sets are of the form $\Mod_{\sC}(\Phi)$.
We refer to this as the \emph{logic topology on structures.}

Corollary~\ref{cor:compactness} is equivalent to the following statement:
If $\sC = \Mod(\Upsilon)$ is the class of models of a uniform $L$-theory~$\Upsilon$, and $T$ is any collection of $L$-sentences such that the family
\[
\left(\ \Mod_{\sC}(\Phi) \mid \text{$\Phi$ is a finite subset of $T$}\ \right) 
\]
has the finite intersection property, then 
\[
\bigcap_{\overset{\Phi\subset T}{\text{$\Phi$ finite}}}\Mod_{\sC}(\varphi)\neq\emptyset.
\]
Thus, an axiomatizable uniform class $\sC$ is compact in the logic topology. 
Moreover, if $(\mathcal{M}_\lambda)_{\lambda\in\Lambda}$ is any family in $\sC$ and $\mathcal{U}$ is an ultrafilter on $\Lambda$, then a $\mathcal{U}$-limit of $(\mathcal{M}_\lambda)_{\lambda\in\Lambda}$ in this topology is given by the ultraproduct $(\prod_{\lambda\in\Lambda} \mathcal{M}_\lambda)_{\mathcal U}$.
This explains why the Compactness Theorem is so named, as well as the naturality of the ultraproduct construction. 
More generally, if $\sC$ is any uniform class endowed with the above topology, then $\sC$ is a relatively compact (dense) subset of the compact class $\overline{\sC} = \Mod(\thry(\sC))$.
We emphasize that uniformity is a necessary condition for precompactness. 

To conclude these remarks, we note that the logic topology on any uniform class~$\sC$ of $L$-structures is regular (although usually not Hausdorff); 
this follows from the fact that $\sC$ is a subspace of the compact (hence regular) space $\overline{\sC} = \Mod(\thry(\sC))$. 
\end{remarks}

Now we present three useful applications of the Compactness Theorem. 
The first one (Corollary~\ref{C:complete theories}) gives a finitary condition for a theory to be of the form $\thry(\mathcal{M})$ for some structure $\mathcal{M}$. 
The second one (Corollary~\ref{C:elementary equivalence}), states that any two models of a complete theory $T$ can be jointly elementarily embedded in a single model of $T$. 

Recall (Definition~\ref{D:model}) that the complete theory of a structure $\mathcal{M}$ is denoted $\thry(\mathcal{M})$.

\begin{corollary}
	\label{C:complete theories}
	The following conditions are equivalent for a theory $T$:
	\begin{enumerate}[\normalfont(1)]
		\item 
		\label{I:complete theory}
The theory~$T$ is complete, i.e., there exists a structure $\mathcal{M}$ such that $T=\thry(\mathcal M)$.
		\item
		\label{I:complete theory conditions}
		\begin{enumerate}[\normalfont(a)]
			\item There exists a uniform theory $\Upsilon$ such that every finite subset of $\Upsilon\cup T$ is satisfiable, and
			\item
			For every $L$-sentence $\varphi$, if $\varphi\notin T$, then there exists $\varphi'\appto\varphi$ such that $\wneg\varphi'\in T$.
		\end{enumerate} 
	\end{enumerate}
\end{corollary}

\begin{proof}
	We only have to prove  (\ref{I:complete theory conditions})$\Rightarrow$(\ref{I:complete theory}). 
By part~(a) of~(\ref{I:complete theory conditions}) and Corollary~\ref{cor:compactness}, $T$ has a model $\mathcal{M}$ whose theory $\thry(\mathcal M)$ extends~$T$. 
Part~(b) of~(\ref{I:complete theory conditions}) gives $\thry(\cM)\subseteq T$, proving~(\ref{I:complete theory}).
\end{proof}

\begin{notation}
	If $\mathcal{M}$ is an $L$-structure and $A=(a_i)_{i\in I}$ is an indexed family of elements of $\mathcal{M}$, we denote by $(\mathcal{M},a_i \mid i\in I)$ the expansion of $\mathcal{M}$  that results from adding a distinct constant for $a_i$, for each $i\in I$. 
In particular, if $A$ is a  set of elements of $\mathcal{M}$, we denote by 
	\[
	(\mathcal{M},a \mid a\in A)
	\]
	 the expansion of $\mathcal{M}$  that results from adding a constant for each $a\in A$. 
If, in this case, $C$ is a set of constant symbols not already in $L$ that includes one constant symbol designating each constant $a\in A$, we informally refer to the signature $L[C]$ (see Definition~\ref{D:signatures})  as~$L[A]$. 
Thus, we informally call the preceding expansion of $\mathcal{M}$ an ``$L[A]$-structure''. 
If $A$ consists of all the elements of $\mathcal{M}$, we write $(\mathcal{M}, a \mid a\in A)$ as $(\mathcal{M}, a\mid a\in \mathcal{M})$ and denote the expanded signature by~$L[\cM]$.
\end{notation}

\begin{definition}
	Let $\mathcal{M}$ and $\mathcal{N}$ be $L$-structures. 
An \emph{elementary embedding}  of $\mathcal M$ into $\mathcal N$ is a map $e$ that assigns to each element $a$ of $\mathcal M$ an element $e(a)$ of $\mathcal N$ such that, whenever $a_1,\dots, a_n$ are elements of $\mathcal{M}$ and $\varphi(x_1,\dots, x_n)$ is an $L$-formula such that $a_i$ is of the same sort as $x_i$, for $i=1,\dots, n$, we have
	\[
	\mathcal{M}\appmod\varphi[a_1\dots, a_n] \quad\Leftrightarrow\quad\mathcal{N}\appmod\varphi[e(a_1)\dots, e(a_n)].
	\]
\end{definition}

Note that $e$ is an elementary embedding of $\mathcal{M}$ into $\mathcal{N}$ if and only if the $L[A]$-structures $(\mathcal{M},a \mid a\in A)$ and $(\mathcal{N}, e(a) \mid a\in A)$ are elementarily equivalent.

\begin{corollary}
	\label{C:elementary equivalence}
	The following conditions are equivalent for two $L$-structures $\cM,\cM'$. 
	\begin{enumerate}[\normalfont(1)]
		\item
		\label{I: elementary equivalence}
		$\cM\equiv\cM'$.
		\item
		\label{I:elementary embedding}
		There exists a structure $\cN$ such that $\cM\prec\cN$ and there is an elementary embedding of $\cM'$ into $\cN$. 
Moreover, $\cN$ can be taken to be an ultrapower of $\cM$.
	\end{enumerate}
\end{corollary}

\begin{proof}
It suffices to prove the direct implication, since the inverse is clear.
Assume $\cM\equiv\cM'$. 
Let $T$ be the complete $L[\cM]$-theory of $\cM$ and $T'$ the complete $L[\cM']$-theory of~$\cM'$. 
Both $T$ and $T'$ are uniform theories, by Remark~\ref{rem:ThM-uniform}.
First we show that every finite subset of $T'_{\approx}\cup T$ has a model. 
Since $T'$ is closed under conjunctions, it suffices to show that $T\cup\{\psi\}$ has a model whenever $\psi\in T'_{\approx}$. 
Any given formula in $T'_\approx$ is of the form $\varphi'(c_{b_1},\dots,c_{b_n})$, where $\varphi'$ is an approximation of a formula $\varphi$ in $T'$ and $c_{b_i}$ is a constant for an element $b_i$ of~$\cM'$, for $i=1,\dots, n$.
Since $\cM'\appmod\varphi[b_1,\dots,b_n]$ by assumption, then there exists $r>0$ such that $\cM'\appmod\exists_r x_1\dots \exists_r x_n\,\varphi(x_1,\dots,x_n)$.
Now, since $\cM\equiv\cM'$ and $\varphi'\appto\varphi$, the semantics of approximate satisfaction ensure the existence of elements $a_1,\dots, a_n$ of $\cM$ such that $\cM\appmod\varphi'[a_1,\dots,a_n]$; 
hence, $(\cM,a \mid a\in\cM)$ admits an expansion to a model $\widetilde{\cM}$ of $T\cup\{\psi\}$ simply by letting $b_1^{\widetilde{\cM}} \coloneqq a_1$, \dots, $b_n^{\widetilde{\cM}} \coloneqq a_n$.
By the Compactness Theorem~\ref{T:compactness}, $T\cup T'$ has a model. 
Let
	\[
	(\,\mathcal N, \widetilde{a}, \widetilde{b}\mid {a\in\cM,b\in\cM'})
	\]
	be a model of this theory. 
The maps $a\mapsto \widetilde{a}$ and $b\mapsto \widetilde{b}$ are elementary embeddings of $\cM$ and $\cM'$, respectively, into $\cN$. 
Without loss of generality, we may assume both that $\cN$ is an ultrapower of~$\cM$, and $\widetilde{a}=a$ for all elements $a$ of~$\cM$. 
This proves (\ref{I:elementary embedding}).
\end{proof}

Essentially the same argument used to prove Corollary~\ref{C:elementary equivalence} proves the following result:

\begin{corollary}
	\label{C:families of equivalent structures}
	Given any family $(\mathcal{M}_i)_{i\in I}$ of elementarily equivalent $L$-structures there exists an $L$-structure $\mathcal N$ such that, for each $i\in I$, there is an elementary embedding of $\mathcal{M}_i$ into $\mathcal N$. 
\end{corollary}

\begin{proposition}
\label{P:isomorphic elementary extensions}
If $\mathcal{M},\mathcal{N}$ are elementarily equivalent structures, then there exist structures $\Hat{\mathcal{M}},\Hat{\mathcal{N}}$ such that $\cM\prec\hcM$, $\cN\prec\hcN$, and $\hcM$ is isomorphic to $\hcN$.

Furthermore, if $(a_i)_{i\in I}$, $(b_i)_{i\in I}$ are elements of $\mathcal{M}$ and $\mathcal{N}$  such that
\[
(\mathcal{M}, a_i \mid i\in I)\equiv(\mathcal{N}, b_i\mid i\in I),
\]
then there exist elementary extensions  $\Hat{\mathcal{M}}\succ\mathcal{M}$ and $\Hat{\mathcal{N}}\succ \mathcal{N}$ and an isomorphism $\mathcal{I}$ from $\Hat{\mathcal{M}}$ into $\Hat{\mathcal{N}}$ such that $\mathcal{I}(a_i)=b_i$ for all $i\in I$. 

\end{proposition}

\begin{proof}

Inductively, for every ordinal $n<\omega$, use Corollary~\ref{C:elementary equivalence} to construct structures $\mathcal{M}_n, \mathcal{N}_n$, and maps $e_n, f_n$, such that
\begin{enumerate}[(i)]
\item
$\mathcal{M}_0=\mathcal{M}$, $\mathcal{N}_0=\mathcal{N}$,
\item
$e_n$ is an elementary embedding of $\mathcal{M}_n$ into $\mathcal{N}_{n+1}$ and $f_n$ is an elementary embedding of $\mathcal{N}_{n+1}$ into $\mathcal{M}_{n+1}$,
\item
$f_{n+1}(e_n(a))= a$ for every element $a$ of the universe of $\mathcal{M}_n$.
\end{enumerate}
Let $\Hat{\mathcal{M}}= \bigcup_{n<\omega} \mathcal{M}_n$,  $\Hat{\mathcal{N}}= \bigcup_{n<\omega} \mathcal{M}_n$, and $e= \bigcup_{n<\omega} e_n$. 
The $e$ is an isomorphism from  $\Hat{\mathcal{M}}$ into $\Hat{\mathcal{N}}$.

The second part of the statement is given by the preceding construction, since $e_0$ can be chosen to map $a_i$ to $b_i$ for each $i\in I$. 
Alternatively, let $C = (c_i\mid i\in I)$ be a collection of new constants and apply the result just proved to the $L[C]$-structures $\widetilde{\cM} = (\mathcal{M}, a_i \mid i\in I)$ and $\widetilde{\cN} = (\mathcal{N}, b_i\mid i\in I)$.
\end{proof}

\subsection{Types}
\label{S:types}
We begin this subsection defining the notion of \emph{finite satisfiability} of a set of formulas.

\begin{definition}\label{def:fin-satisf}
	If $\varphi(x_1,\dots, x_n)$ is a set of $L$-formulas and $\mathcal{M}$ is an $L$-structure, we say that $\varphi$ is \emph{finitely satisfiable} in $\mathcal{M}$ if there exists $r$ such that for every finite $\varphi_0\subseteq\varphi$,
	\[
		\tag{*}
	\mathcal{M}\appmod\exists_r x_1\dots \exists_r x_n\, \bigwedge_{\varphi\in\varphi_0}\varphi(x_1,\dots, x_n).
	\]
\end{definition}

Note that if $(*)$ holds, then for every finite $\varphi_0\subseteq\varphi_{\approx}$ there exist $c_1,\dots, c_n$ in $\mathcal{M}$ such that $\mathcal{M}\models\bigwedge_{\varphi\in\varphi_0}\varphi[c_1,\dots, c_n]$ (discrete satisfaction). 
However, the tuple $(c_1, \dots, c_n)$ depends on~$\varphi_0$.

We now introduce one of the central concepts of model theory: that of \emph{type}.

\begin{definition}
\label{D:type}
	If $\mathcal{M}$ is an $L$-structure,  $A$ is a set of elements of $\mathcal{M}$, and $(c_1,\dots,c_n)$ is a tuple of elements of $\mathcal{M}$, the \emph{type of  $(c_1,\dots,c_n)$ over $A$}, denoted $\tp_A(c_1,\dots, c_n)$, is the set of all $L[A]$-formulas $\varphi(x_1,\dots, x_n)$ such that $(\mathcal{M},a \mid a\in A)\appmod \varphi[c_1,\dots, c_n]$. 
        We denote $\tp_\emptyset(c_1,\dots, c_n)$ by  $\tp_L(c_1,\dots,c_n)$.

	If $T$ is a complete $L$-theory and $t=t(x_1,\dots, x_n)$ is a set of $L[C]$-formulas, where $C$ is a set of constant symbols not in $L$, we say that $t$ is an \emph{$n$-type of} $T$ if there exists an $L[C]$-model $\mathcal{M}$ of~$T$ and elements $c_1,\dots,c_n$ in $\mathcal{M}$ such that $t = \tp_A(c_1,\dots, c_n)$ where $A$ is the subset of~$\cM$ consisting of elements interpreting the constants in~$C$. 
In this case, we say $t$ \emph{is realized} in $\mathcal{M}$, and that the $n$-tuple $(c_1,\dots, c_n)$ \emph{realizes} $t$ in $\mathcal{M}$. 
\end{definition}

\begin{remarks}\hfill
	\label{R:types}

	\begin{enumerate}
	
		\item
		If $T$ is a complete theory and $t$ is a type for $T$, then $T\subseteq t$. 
In fact, if $\mathcal{M}$ is a model of $T$ and $t$ is a type over a set $A$ of elements of $\mathcal M$, then the set of sentences in $t$ is precisely the complete $L[A]$-theory $\thry(\mathcal{M}, a\mid a\in A)$, which extends~ $T=\thry(\cM)$.
		\item
		\label{I:types realized in extension}
		The notation $\tp_A(c_1,\dots, c_n)$ is imprecise in the sense that it does not make reference to the structure $\mathcal{M}$ where the elements $c_1,\dots, c_n$ and the set $A$ ``live''. 
However, since $T$ is a complete $L$-theory by assumption, precise knowledge of $\mathcal{M}$ is to a large extent unnecessary. 
In fact, if we are given a family $(t_i)_{i\in I}$ of types of $T$ such that, for $i\in I$, the type $t_i$ is over a set of parameters $A_i$ realized in an $L$-structure $\mathcal{M}_i$, then by Corollary~\ref{C:families of equivalent structures} there exists a single structure $\mathcal{N}$ in which all the structures $\mathcal{M}_i$ are elementary embedded. 
In fact, as we shall see in Proposition~\ref{P:Keisler-Shelah} below, given any cardinal $\kappa$ we can fix a ``big'' model $\cN$ of $T$ that is an elementary extension of every $L$-structure $\cM$ with cardinality $\card(\cM)<\kappa$, and such that $\cN$ realizes every type over any of its subsets $A$ with $\card(A)<\kappa$. 
Furthermore, the model $\mathcal{N}$ can be taken with the following additional homogeneity property: 
If $c_1,\dots, c_n$ and $c_1',\dots, c_n'$ are elements of $\mathcal{N}$ and $A$ is a set of elements  of $\mathcal{N}$ with $\card(A) < \kappa$, then $\tp_A(c_1,\dots, c_n)=\tp_A(c'_1,\dots, c'_n)$ if and only if there is an automorphism of $\mathcal{N}$ carrying $c_1,\dots, c_n$ to $c'_1,\dots, c'_n$ and fixing $A$ pointwise. 
This allows viewing types as orbits on the big model under the action of its group of automorphisms, and enables a Galois-theoretic viewpoint of complete theories. 
It also explains the use of the word ``type''. We will discuss this in more detail in part~\ref{sec:spaces-types}.
		\item
		If $T$ is a complete theory, $\mathcal{M}$ is an arbitrary model of $T$, $A$ is a set of elements of $\mathcal{M}$, and $t$ is a type of $T$ over~$A$, there is no guarantee that $t$ is realized in $\mathcal{M}$. 
However:
		\begin{enumerate}
			\item
			The equivalence between (\ref{I:type elementary equivalence}) and (\ref{I:type elementary extension}) of Proposition~\ref{P:types} below shows that $\mathcal{M}$ has an elementary extension where $t$ is realized. 
			In particular, every elementary extension of $\mathcal{M}$ has a further elementary extension where $t$ is realized.  
			
			\item $t$ is finitely satisfiable in every model of $\thry(\mathcal{M},a \mid a\in A)$.
			\end{enumerate}
		\end{enumerate}
\end{remarks}

\begin{proposition}
	\label{P:types}
Let $T$ be a complete theory and let $t(x_1,\dots,x_n)$ be a set of $L[C]$-formulas, where $C$ is a set of constant symbols not in $L$. 
The following conditions are equivalent.
\begin{enumerate}[\normalfont(1)]
	\item 
	\label{I:type elementary equivalence}
	There exists a model $\mathcal M$ of $T$, a set $A$ of elements of $\mathcal M$, 
	and elements $c_1,\dots, c_n$ of $\mathcal{M}$ such that $t(x_1,\dots, x_n)=\tp_A(c_1,\dots, c_n)$.

	\item 
	\label{I:type elementary extension}
	For every model $\mathcal{N}$ of $T$ there exists an elementary extension $\mathcal{N}'$ of $\mathcal{N}$, a set $A$ of elements of $\mathcal{N}'$, 
	and elements $c_1,\dots, c_n$ of $\mathcal{N}'$ such that $t(x_1,\dots, x_n)=\tp_A(c_1,\dots, c_n)$.
	
		\item 
		\label{I:type complete}
		\begin{enumerate}[\normalfont(a)]
			\item 
			There exists a positive rational $r$ such that, for every finite subset $\varphi\subset t$, the formula $\exists_r x_1\dots \exists_r x_n\, \bigwedge_{\varphi\in\varphi}\varphi(x_1,\dots, x_n)$ is in~$t$.
			\item
			\label{I:weak negation}
			For every $L[C]$-formula $\varphi(x_1,\dots, x_n)$, if $\varphi\notin t$, then there exists $\varphi'\appto\varphi$ such that $\wneg\varphi'\in t$.
		\end{enumerate}
\end{enumerate}
\end{proposition}

\begin{proof}
The equivalence between (\ref{I:type elementary equivalence}) and (\ref{I:type complete}) is given by Corollary~\ref{C:complete theories} (by replacing the variables $x_1,\dots, x_n$ with constant symbols not already in $L[A]$). 
The equivalence between (\ref{I:type complete}) and (\ref{I:type elementary extension}) is given by Corollary~\ref{C:elementary equivalence}.
\end{proof}

\begin{remark}
	Part (\ref{I:type complete}) of Proposition~\ref{P:types} gives a purely syntactic criterion for a given set of formulas to be a type. 
\end{remark}

\subsection{Saturated and homogeneous models}
\label{sec:saturation}
Strictly for reasons of notational simplicity, we will focus our attention on one-sorted structures throughout this subsection. 
When a structure  $\mathcal{M}$ has only one sort, this sort is called the \emph{universe} of $\mathcal{M}$. 
The \emph{cardinality} of a one-sorted structure is defined as the cardinality of its universe.

The observations in Remark~\ref{P:types} lead naturally to the concept of \textit{saturated structure}:

\begin{definition}
	Let $\mathcal{M}$ is a structure with universe $M$ and let $\kappa$ be an infinite cardinal with $\kappa\le\card(M)$. 
We say that $\mathcal{M}$ is \emph{$\kappa$-saturated} if whenever $t$ is a type of $\thry(\mathcal{M})$ over a subset of $M$ of cardinality strictly less than $\kappa$, there is a realization of $t$ in $\mathcal{M}$.

In this context, we may abuse notation and use $\omega$ as a synonym for~$\aleph_0$, the cardinality of (infinite) countable sets. 
Thus, an $\omp$-saturated structure is one that realizes types over any countable subset of the universe.
The informal terminology ``countable saturation'' shall mean $\omp$-saturation, and not $\omega$-saturation.
\end{definition}

Notice that an $L$-structure $\mathcal{M}$ is $\kappa$-saturated if and only if, whenever $A$ is a subset of the universe of $\mathcal{M}$ of cardinality less than $\kappa$ and $\varphi(x_1,\dots, x_n)$ is a set of $L[A]$-formulas that is finitely satisfiable in $\mathcal{M}$, there exist $c_1,\dots, c_n$ in $\mathcal{M}$ such that $\varphi(x_1,\dots, x_n) \subseteq \tp_A(c_1,\dots, c_n)$.

\begin{proposition}
\label{P:Keisler-Shelah}
	If $\mathcal{M}$ is $\kappa$-saturated and $\mathcal{N}$ is a structure of cardinality less than $\kappa$ such that $\mathcal{N}\equiv\mathcal{M}$, then $\mathcal{N}$ can be elementarily embedded in $\mathcal{M}$.
\end{proposition}

\begin{proof}
	Assume that $\mathcal{M}$ is $\kappa$-saturated, $\mathcal{N}\equiv\mathcal{M}$, and $\card(\cN)<\kappa$. 
If follows easily by induction that, if  $\alpha$ is an ordinal satisfying $\alpha<\kappa$ and $(a_i)_{i<\alpha}$ is a list of elements of $\mathcal{N}$, then there exists $(a'_i)_{i<\alpha}$ in $\mathcal{M}$ such that
	$(\mathcal{N}, a_i \mid i<\alpha)\equiv(\mathcal{M}, a_i'\mid i<\alpha)$. 
Thus, in the case when $(a_i)_{i<\alpha}$ lists all the elements of $\mathcal{N}$, the map $a_i\mapsto a_i'$ is an elementary embedding of $\mathcal{N}$ into $\mathcal{M}$.
\end{proof}

\begin{proposition}
If $\mathcal{M}$ is an $\aleph_1$-saturated $L$-structure, then $\appmod$ and $\models$ coincide on $\mathcal{M}$, i.e., for an $L$-formula $\varphi(x_1,\dots, x_n)$ and elements $a_1,\dots, a_n$ of suitable sorts, we have
\[
\mathcal{M}\appmod \varphi[a_1,\dots, a_n]
\quad\Leftrightarrow\quad
\mathcal{M}\models \varphi[a_1,\dots, a_n].
\]
\end{proposition}

\begin{proof}
By induction on the complexity of~$\varphi$.
\end{proof}

Recall that, if $\kappa$ is a cardinal, then $\kappa^+$ denotes the smallest cardinal larger than $\kappa$.

\begin{proposition}
If $\kappa$ is an infinite cardinal, then every structure has a $\kappa^+$-saturated elementary extension.
\end{proposition}

\begin{proof}
Fix an infinite cardinal $\kappa$  and  an $L$-structure $\mathcal{M}$. 
Applying  Remark \ref{R:types}-(\ref{I:types realized in extension}), we construct inductively, for every ordinal $i<\kappa^+$, a structure $\mathcal{M}_i$, such that
\begin{enumerate}[(i)]
\item
$\mathcal{M}_0=\mathcal{M}$,
\item
$\mathcal{M}_i\prec \mathcal{M}_{i+1}$ and every type over a subset of the universe of $\mathcal{M}_i$ of cardinality less than $\kappa$ is realized in $\mathcal{M}_{i+1}$,
\item
If $j<\kappa^+$ is a limit ordinal, then $\mathcal{M}_j= \bigcup_{i<j} \mathcal{M}_i$.
\end{enumerate}
It follows from the Elementary Chain Property (Proposition~\ref{P:elementary chain}) that $(\mathcal{M}_i)_{i<\kappa^{+}}$ is an elementary chain, and that $\bigcup_{i<\kappa^+} \mathcal{M}_i$ is an elementary extension of~$\mathcal{M}$, which is clearly $\kappa^+$-saturated.
\end{proof}

Suppose that $\mathcal{M}$ is $\kappa$-saturated and let $\alpha$ be an ordinal with $\alpha<\kappa$. 
It follows directly from the $\kappa$-saturation of $\mathcal{M}$ that if $(a_i)_{i<\alpha}$, $(a'_i)_{i<\alpha}$ are families in $\mathcal M$ such that $(\mathcal{M}, a_i \mid i<\alpha)\equiv(\mathcal{M}, a_i'\mid i<\alpha)$, then for every element $b$ of $\mathcal{M}$ there exists an element $b'$ such that 
\[
(\mathcal{M}, b, a_i \mid i<\alpha)\equiv(\mathcal{M}, b', a_i'\mid i<\alpha).
\]
 A structure $\mathcal{M}$ that has this extension property for all pairs of families  $(a_i)_{i<\alpha}$, $(a'_i)_{i<\alpha}$ with $\alpha<\kappa$ is said to be \emph{$\kappa$-homogeneous}. 
Note that if $\mathcal{M}$ is $\kappa$-homogeneous with $\kappa=\card(\mathcal{M})$, then the $\kappa$-homogeneity can be used iteratively to extend the map $a_i\mapsto a_i$ ($i<\alpha$) to an automorphism of $\mathcal{M}$. 
This suggests the following definition.

\begin{definition}
Let $\kappa$ be an infinite cardinal. 
A structure $\mathcal{M}$ is \emph{strongly $\kappa$-homogeneous} if whenever $\alpha$ is an ordinal with $\alpha<\kappa$ and $(a_i)_{i<\alpha}$, $(a'_i)_{i<\alpha}$ are families in the universe of $\mathcal M$ such that $(\mathcal{M}, a_i \mid i<\alpha)\equiv(\mathcal{M}, a_i'\mid i<\alpha)$, there exists an automorphism $\mathcal{I}$ of $\mathcal{M}$ such that $\mathcal{I}(a_i)=a_i'$ for all $i<\alpha$.
\end{definition}

 The following theorem shows that, for arbitrarily large $\kappa$, every structure has elementary ultrapowers that are $\kappa$-saturated and $\kappa$-homogeneous.

\begin{theorem}
	\label{T:keisler-shelah}
	For every infinite cardinal $\kappa$ there exists an ultrafilter $\mathcal{U}$ with the following property: Whenever $\mathcal{M}$ is a metric structure of cardinality at most $2^\kappa$, the $\mathcal{U}$-ultrapower of $\mathcal{M}$ is both $\kappa^+$-saturated and $\kappa^+$-homogeneous.
\end{theorem}

We omit the proof of Theorem \ref{T:keisler-shelah}. 
The construction of the ultrafilter is due to S.~Shelah~\cite{Shelah:1971}, and it builds on ideas of H.~J.~Keisler and K.~Kunen. 
Shelah's epochal proof is for traditional first-order (i.e., discrete structures), but his argument was adapted by C.~W.~Henson and the second author for structures based on Banach spaces \cite[Corollary 12.3]{Henson-Iovino:2002}. 
The proof for Banach structures applies to general metric structures without significant changes.

\subsection{Extending the language}
\label{sec:extensions}

Given a signature $L$, a sort index set $\mathbf S$, indices $s_0,s_1,\dots, s_n\in\mathbf S$, and a new function symbol $f:s_1\times\dots\times s_1\to  s_0$, we will denote by $L[f]$ the signature that results from adding $f$ to $L$. 
If $\mathcal{M}$ is an $L$ structure based on $(M^{(s)}\mid s\in\mathbf{S})$ and $F$ is a function from $M^{(s_1)}\times\dots\times M^{(s_n)}$ into $M^{(s_0)}$, then will denote by $(\mathcal {M}, F)$ the expansion of $\mathcal{M}$ to $L[f]$ that results from defining $f^{\mathcal M}$ as $F$. 

Let $L$ be a signature and  consider the signature $L[f]$, where $f$ is a function symbol not in~$L$. 
Fix also an $L[f]$-theory $T$ (which could be empty). 
In standard mathematical practice there are cases where, in models of $T$, the new symbol can be dispensed with because it is already definable through $L$-formulas. 
This leads to the notion of \emph{definability}, which is the main concern of this subsection. 
There are several ways to formalize the concept of $f$ being $T$-\emph{definable}; for instance, one could say that
\begin{quote}
If  $(\mathcal{M},F)\appmod T$, where $\mathcal{M}$ is an $L$-structure, then $F$ is determined uniquely by $T$, i.e., if $(\mathcal{M},F),(\mathcal{M},F')\appmod T$, then $F=F'$.
\end{quote}
or, alternatively,
\begin{quote}
In all models of $T$, formulas involving $f$ can be approximated by $L$-formulas.
\end{quote}
Below, we  prove that these two conditions are equivalent (see Theorem~\ref{T:beth}); but first we must formalize what we mean by ``approximated'' in the preceding statement.

Recall that if $(M,d,a)$ is a pointed metric space, the open ball of radius $r$ around the anchor point $a$ is denoted or $B_M(r)$, or $B(r)$ if $M$ is given by the context. 
If $\mathcal{M}$ is a metric structure and $a$ is in $B_{M^{(s)}}(r)$, where $M^{(s)}$ is one of the sorts of $\mathcal{M}$, we may informally say that $a$ is an element of $B_{\mathcal M}(r)$.

\begin{definition}\label{def:explicit-def}
Let $L$ be a signature, let $f$ be a real-valued $n$-ary function symbol, and let $T$ be a uniform $L[f]$-theory. 
We will say that $f$ is \emph{explicitly defined by $T$ in $L$} if the following condition holds for every $r\in\mathbb{R}$ and every pair of nonempty intervals $I\subset J \subseteq \mathbb{R}$ with $I$ closed and $J$ open:  
There exists an $L$-formula $\varphi_{r,I,J} = \varphi(x_1,\dots, x_n)$ such that, whenever $(\mathcal{M},F)\appmod T$  and $a_1,\dots, a_n$ are elements of $B_{\mathcal M}(r)$, we have
\begin{align*}
F(a_1,\dots, a_n)\in I
\quad&\Rightarrow\quad
\mathcal{M}\appmod\varphi[a_1,\dots, a_n],\\
\mathcal{M}\appmod\varphi[a_1,\dots, a_n]
\quad&\Rightarrow\quad
F(a_1,\dots, a_n)\in J.
\end{align*}
The collection $(\varphi_{r,I,J} \mid r\in\RR, \emptyset\ne I\subset J)$ is a \emph{definition scheme} for~$F$ (modulo~$T$). 
\end{definition}

\begin{remark}\label{rem:explicit-def}
  A definition scheme $\Sigma = (\varphi_{r,I,J} : r\in\RR, \emptyset\ne I\subset J)$ ($I$ closed, $J$ open) characterizes~$F$ uniquely in any structure $(\cM,F)\appmod T$.  
Namely, if $a_1,\dots,a_n\in B_{\cM}(r)$, then $F(a_1,\dots,a_n)$ is the unique real number $t$ with the following property: $\cM\appmod\varphi_{r,I,J}[a_1,\dots,a_n]$ whenever $J\supset I\ni t$.  
Certainly, $F(a_1,\dots,a_n)$ is one such number~$t$ (since $\Sigma$ is a definition scheme for~$F$).  
Conversely, if $t$ has the stated property, we have $F(a_1,\dots,a_n)\in J$ whenever $J\supset I\ni t$, hence $t = F(a_1,\dots,a_n)$.
\end{remark}

The theorem below is the most fundamental result about first-order definability. 
The topological proof we give here is not commonly known. 
We thank Xavier Caicedo for pointing out an error in an earlier draft of this manuscript and suggesting a correction.

\begin{theorem}[Beth-Svenonius Definability Theorem]
\label{T:beth}
Let $L$ be a signature, let $f$ be a real-valued function symbol, and let $T$ be a uniform $L[f]$-theory. 
The following conditions are equivalent.
\begin{enumerate}[\normalfont(1)]
\item
\label{I:explicit defn}
$f$ is explicitly defined by $T$ in $L$.
\item
\label{I:implicit defn}
If $\mathcal{M}$ is an $L$-structure, then
\[
(\mathcal{M},F), (\mathcal{M},F') \appmod T
\quad\Rightarrow\quad
F=F'.
\]

\end{enumerate}

\end{theorem}

Before proving Theorem~\ref{T:beth}, let us make the following observation about general topological spaces.
If $Z$ is a topological space, two points $x,y\in Z$ are said to be \emph{topologically indistinguishable}, denoted $x\equiv y$, if every neighborhood of $x$ contains $y$ and every neighborhood of $y$ contains~$x$. 
Now, we observe that if $X,Y$ are regular topological spaces with $X$ compact and $g:X\to Y$ is a continuous bijection, then $g$ is a homeomorphism if and only if
\begin{equation}
g(x)\equiv g(y)
\quad\Rightarrow\quad
x\equiv y.\label{eq:g-mod-equiv}
\end{equation}
Indeed, any such $g$ maps indistinguishable points to indistinguishable points (by continuity), so $g$~induces a continuous map $\overline{g}:\overline{X}\to \overline{Y}$ between the spaces~$\overline{X}=X/{\equiv}$ and~$\overline{Y}=Y/{\equiv}$. 
Since topologically distinguishable points of a normal space have disjoint neighborhoods, $\overline{X}$ and $\overline{Y}$ are Hausdorff spaces, with $\overline{X}$ compact; 
moreover, $\overline{g}$ is a bijection, by~\eqref{eq:g-mod-equiv}, and hence a homeomorphism. 
Clearly, $g$ is a homeomorphism also.

\begin{proof}
 $(\ref{I:explicit defn})\Rightarrow(\ref{I:implicit defn})$: 
This is an immediate consequence of Remark~\ref{rem:explicit-def}.

As a preliminary step to showing $(\ref{I:implicit defn})\Rightarrow(\ref{I:explicit defn})$, we prove the following:
\begin{claim}
Assume that (\ref{I:implicit defn}) holds. 
If $L'$ is a signature extending~$L$, and $\mathcal{M}$, $\mathcal{N}$ are $L'$-structures admitting $L'[f]$-expansions $(\mathcal{M},F)$, $(\mathcal{N},G)$ that are models of~$T$, then
\[
\mathcal{M}\equiv_{L'}\mathcal{N}
\quad\Rightarrow\quad
(\mathcal{M},F)\equiv_{L'[f]}(\mathcal{N},G).
\]
\end{claim}

To prove the claim, let $L'$ extend $L$ and let $\mathcal{M}$, $\mathcal{N}$ be elementarily equivalent $L'$-structures admitting expansions $(\mathcal{M},F)$, $(\mathcal{N},G)$ that are models of~$T$. 
Using Theorem~\ref{T:keisler-shelah}, fix an ultrafilter $\mathcal{U}$ such that the $\mathcal U$-ultrapowers $(\mathcal{M})_{\mathcal{U}},(\mathcal{N})_{\mathcal{U}}$ are isomorphic $L'$-structures. 
By \Los's Theorem~\ref{T:Los}, the $\cU$-ultrapower $(\cM,F)_{\cU} = ((\cM)_{\cU},(F)_{\cU})$ is an $L'[f]$-structure elementarily equivalent to~$(\cM,F)$; similarly, $(\cN,G)_{\cU} = ((\cN)_{\cU},(G)_{\cU}) \equiv_{L'[f]}(\cN,G)$. 
Let $\mathcal{I}$ be an $L'$-isomorphism from $(\mathcal{M})_{\mathcal{U}}$ into $(\mathcal{N})_{\mathcal{U}}$.  
\emph{A fortiori,} $\mathcal{I}$ is an $L$-isomorphism between (the $L$-reducts of)~$(\cM)_{\cU}$ and~$(\cN)_{\cU}$. 
By~(\ref{I:implicit defn}), we have $\mathcal{I}((F)_{\mathcal{U}})=(G)_{\mathcal{U}}$; thus, $(\mathcal{M},F)_{\mathcal{U}}$ and $(\mathcal{N},G)_{\mathcal{U}}$ are isomorphic $L'[f]$-structures, so $(\mathcal{M},F)\equiv_{L'[f]}(\mathcal{N},G)$. This proves the claim.

$(\ref{I:implicit defn})\Rightarrow(\ref{I:explicit defn})$: 
Assume that $(\ref{I:implicit defn})$ holds.
The signature of $L[f]$ specifies that $f:s_1\times\dots\times s_n\to  s_{\mathbb R}$ for some sorts $s_1,\dots,s_n$. 
Fix a new tuple $\overline{c} = c_1,\dots, c_n$ of constant symbols such that $c_i$ is of sort $s_i$ for $i=1,\dots, n$ and let $L' = L[\overline{c}]$. 
For any fixed rational $r>0$, let $T'_r$ be the $L'[f]$-theory obtained by adding to~$T$ the sentences $\dd(c_i,a_{s_i})\le r$, for $i=1,\dots,n$, where $a_{s_i}$ is (the symbol for) the anchor of sort~$s_i$. 
Clearly, $T'_r$ is a uniform $L'[f]$-theory. 
Let ${\sC}$ be the class of all models of~$T'_r$, i.e., of $L'[f]$-structures of the form $(\cM, b_1, \dots, b_n, F)$, where $(\cM, F)$ is a model of~$T$, and $b_i$ is an element of sort~$s_i$ of~$\cM$ that satisfies $\dd(b_i,a_{s_i})\le r$ for $i=1,\dots,n$.
Let $\mathscr D$ be the class of all $L'$-structures $\cN$ that are $L'$-reducts of some $L'[f]$-structure $(\cN, F)\in\mathscr C$ and let $g:\mathscr C\to\mathscr D$ be the map $(\cN, F)\mapsto \cN$. 
Note that $g$ is surjective by definition.

We regard the classes $\mathscr C$ and $\mathscr D$ as topological spaces as follows. 
A basis for the closed classes of~$\mathscr{C}$ is given by all subclasses of the form $\Mod_{\mathscr{C}}(\varphi)$ where $\varphi$ is an $L'[f]$-sentence (the union of finitely many such basic closed subclasses is of the same form, since the logic is closed under disjunction); 
in other words, the closed subclasses of $\mathscr C$ are those of the form $\Mod_{\mathscr{C}}(\Phi)$, where $\Phi$ is an $L'[f]$-theory.
Similarly, the closed subclasses of $\mathscr{D}$ are those of the form $\Mod_{\mathscr{D}}(\Phi)$, where $\Phi$ is an $L'$-theory.%
\footnote{Later, in Subsection~\ref{sec:spaces-types}, we shall refer to this as the \emph{logic topology}.}
Both $\mathscr C$ and $\mathscr D$ are regular, with $\mathscr{C}$ compact, by Remark~\ref{R:topology}, and the surjection $g$ is clearly continuous. 

Let $\cN\in\mathscr{D}$, so $\cN$ is the $L'$-reduct of some model $(\cN,F)\appmod T'_r$. 
Let $\cM = \cN\restriction L$ be the $L$-reduct of~$\cN$.
Clearly, $(\cM,F)$ is a model of $T'_r\restriction L[f] = T$; 
moreover, by~$(\ref{I:implicit defn})$, $(\cM,F)$ is the unique expansion of $\cM$ to a model of~$T$. 
Since $T'_r$ extends~$T$, $(\cN,F)$ is necessarily the unique preimage of~$\cN$ under~$g$, so $g$ is injective; thus, $g$ is a continuous bijection. 
By the claim and the topological observation immediately following the statement of the theorem, we conclude that $g$ is a homeomorphism. 

Fix intervals $\emptyset\ne I = [p,q]\subset J = (u,v)$ with $p,q,u,v$ rational, and let
\begin{align*}
  K &= \Mod_{\sC}(p\le f(\overline{c}) \le q), &
  K' &= \Mod_{\sC}(f(\overline{c})\le u \vee f(\overline{c})\ge v).
\end{align*}
Since $K$ and $K'$ are closed and disjoint, so are the subsets $g(K)$ and $g(K')$ of the homeomorphic image $\sD = g(\sC)$, which is compact since $\sC$ is. 
Since any compact regular space is normal, there exists a closed neighborhood~$Q$ of $g(K)$ disjoint from~$g(K')$. 
By compactness of~$g(K')$, $Q$ may be taken of the form $\Mod_{\sD}(\varphi(\overline{c}))$ for some $L$-formula $\varphi_{I,J,r} = \varphi(\overline{x})$.
Since the interpretation of $\overline{c}$ is arbitrary in $B_{\cM}(r)$, the scheme $(\varphi_{I,J,r} : r\in\RR, \emptyset\ne I\subset J)$ defines $f$ explicitly by~$T$.
\end{proof}

\begin{definition}
\label{D:definable real-valued}
If $L$ is a signature, $f$ is a real-valued function symbol, and $T$ is a uniform $L[f]$-theory,
we will say that $f$ is $T$-\emph{definable} in $L$ if it satisfies the equivalent conditions of Theorem~\ref{T:beth}.
When $L$ is given by the context, we may simply say that $f$ is $T$-\emph{definable};
furthermore, if $L$ and $T$ are given by the context, we may say that $f$ is ``definable''.
\end{definition}

If $T$ is a uniform theory, Theorem~\ref{T:beth} allows us to see the real-valued functions that are $T$-definable as those that are left fixed by automorphisms of sufficiently saturated models of~$T$. 
This observation yields Corollary~\ref{C: definability closure properties} below.

For notational convenience, if $f$ is a function symbol for a given function $F:\mathbb{R}^n\to\mathbb{R}$, we will liberally identify $f$ with its interpretation $F$.
\begin{corollary}
\label{C: definability closure properties}
Let $T$ be a uniform theory. 
\begin{enumerate}[\normalfont(1)]
\item
\label{I:composition definable}
A composition of functions that are $T$-definable is $T$-definable.
\item
\label{I:continuous definable}
Every continuous function $F:\mathbb{R}^n\to\mathbb{R}$ is $T$-definable.
\item
\label{I:sup definable}
If $F:\mathbb{R}^{n+1}\to\mathbb{R}$ is $T$-definable by $T$, so are the functions $G_r:\mathbb{R}^n\to\mathbb{R}$ and $H_r:\mathbb{R}^n\to\mathbb{R}$ defined by
\begin{align*}
G_r(x_1,\dots, x_n)&=\sup_{y\in B(r)} F(x_1,\dots, x_n, y),\\
\qquad
H_r(x_1,\dots, x_n)&=\inf_{y\in B(r)} F(x_1,\dots, x_n, y).
\end{align*}
\end{enumerate}
\end{corollary}

\begin{proof}
Clauses  (\ref{I:composition definable}) and (\ref{I:sup definable}) follow from the definitions. 
To prove~(\ref{I:continuous definable}), note that, since every signature comes equipped with the ordered field and lattice structure for $\mathbb{R}$ plus a constant for each rational,  all polynomials functions with rational coefficients are definable in any theory. 
Therefore,~(\ref{I:continuous definable}) follows from the Stone-Weierstrass theorem.
\end{proof}

\begin{definition}
\label{D:definable}
If $L$ is a signature, $f$ is an $n$-ary function symbol (not necessarily real-valued), and $T$ is a uniform $L[f]$-theory,
we will say that $f$ is $T$-\emph{definable in~$L$} if the real-valued function $d(f(x_1,\dots, x_n),y)$ is $T$-definable in $L$ in the sense of Definition~\ref{D:definable real-valued}.
As in Definition~\ref{D:definable real-valued}, if $L$ or $T$ are given by the context, we may omit them from the nomenclature.
\end{definition}

\begin{remark}
It is not difficult to verify that, if $f$ is a real-valued function symbol, there is no conflict between the notions of definability for $f$ given by Definitions~\ref{D:definable real-valued} and~\ref{D:definable}. 
(This is because if $f$ is $n$-ary and $F$ is an interpretation of $f$, then for any real number $r$,  $d(F(a_1,\dots, a_n),b)=r$ if and only if $b = \pm F(a_1,\dots, a_n)$.)
\end{remark}

\subsection{Spaces of types and the monster model}
\label{sec:spaces-types}

In this section, $T$ will denote a fixed complete $L$-theory and we will denote by $\mathscr T$ the set of types that are realized in a model of~$T$. 
Note that, if $(t_i)_{i\in I}$ is a family of types in $\mathscr T$ and $t_i$ is realized in $\mathcal M_i$ then, by Corollary~\ref{C:families of equivalent structures}, there exists a model $\mathcal N$ of $T$ such that $\mathcal M_i\prec \mathcal{N}$ for every $i\in I$. 
Thus, each $t_i$ is realized in $\mathcal{N}$.

By Remark~\ref{R:types}-(\ref{I:types realized in extension}), if $t=t(x_1,\dots, x_n)$ is a type in $\mathscr T$, then there exists a positive rational $r$ such that
\[
\exists_r x_1\dots \exists_r x_n\, \bigwedge_{\varphi\in\varphi}\varphi(x_1,\dots, x_n) \in t,
\qquad
\text{for each finite $\varphi\subseteq t$}.
\]
For each choice of $n$ and  $r$, we will denote by $\mathscr{T}_n^{(r)}$ be the set of all $t(x_1,\dots, x_n)\in\mathscr T$ that satisfy this condition.

If $\varphi(x_1,\dots,x_n)$ is an $L$-formula, let
\[
[\varphi]=\{\, t\in \mathscr T \mid \varphi\in t\, \}.
\]
The collection of sets of this form is closed under finite unions and intersections. We define a topology on $\mathscr T$ by letting these be the basic closed sets. 
We will refer to this topology as the \emph{logic topology}.
We shall always regard $\mathscr T$ as a topological space via the logic topology. 

\begin{remark}
\label{R:compactness and types}
The Compactness Theorem~\ref{T:compactness} says exactly that $\mathscr{T}_n^{(r)}$ is compact for each $n$ and~$r$.  
\end{remark}

The complement of a set of the form $[\varphi]$ is not necessarily of the same form; 
however,  if $t\in \mathscr T$, by Proposition~\ref{P:types}-(\ref{I:weak negation}),  we have $t\notin [\varphi]$ if and only if there exists $\psi\in t$ and $\psi'>\psi$ such that $t\in[\psi']\subseteq[\varphi]^c$. 
Thus, for $t\in \mathscr T$, the sets of the form $[\psi']$ where $\psi'$ is an approximation of a formula $\psi\in t$ form a local neighborhood base around $t$.

Let $\varphi(x_1,\dots,x_n)$ be the formula $\exists_r x_1\dots \exists_r x_n\, \bigwedge_{1\le i \le n} x_i=x_i$. Then, $[\varphi]=\mathscr{T}_n^{(r)}$. 
Every type contains a formula $\varphi$ of this kind, for some $n$ and some $r$. Therefore, the space $\mathscr T$ is locally compact.

Let $A$ be a set of parameters, i.e., $A$ is a set of elements of some fixed model $\mathcal{M}$ of $T$. If $(t_i)_{i\in I}$ is a family in $\mathscr T_n^{(r)}$ such that $t_i$ is a type over $A$ for each $i\in I$ and $\mathcal{U}$ is an ultrafilter of $I$, then, by the compactness of $\mathscr T_n^{(r)}$, the limit  $\mathcal{U}\lim_{i} t_i$ is a type over $A$. 
Conversely, if $t(\bar x)$ is a type over $A$, since $t$ is finitely satisfiable in $\mathcal{M}$ (see Definition~\ref{def:fin-satisf}), there exists a family $(\bar a_i)_{i\in I}$ in $\mathcal{M}$ and an ultrafilter $\mathcal{U}$ on $I$ such that $t=\mathcal{U}\lim_{i} \tp_A(\bar a_i)$. Thus, types over  model $\mathcal{M}$ can be viewed as ultrafilter limits of types realized within $\mathcal{M}$.
  
Let $\kappa$ be a cardinal that is both larger than the cardinality of every model of $T$ needed in our proofs and larger than the cardinality of any indexed family of types mentioned. 
Let now $\mathfrak C$ be a $\kappa$-saturated, $\kappa$-homogeneous model of $T$ (see Theorem~\ref{T:keisler-shelah}). 
Then we can assume that:

\begin{enumerate}
\item
Every model of $T$ occurs as an elementary submodel of $\mathfrak C$ (see Proposition~\ref{P:Keisler-Shelah}).
\item
Every type (over a set of parameters is a model of $T$)  is realized in $\mathfrak C$.
\item
If $(a_1,\dots, a_n)$ and $(b_1,\dots, b_n)$ are $n$-tuples of elements of $\mathfrak{C}$ such that $\tp_A(a_1,\dots, a_n)=\tp_A(b_1,\dots, b_n)$, then there exists an automorphism $f$ of $\mathfrak{C}$ such that $f(a_i)=b_i$ for $i=1,\dots, n$.
\item
A real-valued function $f$ is definable in $T$ over a set of parameters $A$ if and only if $f$ is fixed by every automorphism of $\mathfrak C$ that fixes $A$ pointwise fixes the graph of $f$. (See the remarks preceding Corollary~\ref{C: definability closure properties}.)
\end{enumerate}
We shall henceforth refer to $\mathfrak C$ as a \emph{big model} or a \emph{monster model} for $T$. Monster models are a time-saving device, and what makes them convenient  are the properties listed above. For each complete theory considered, we will always work within a fixed monster model. The particular choice of monster model will be irrelevant for our discussions since any two such models can be  embedded elementarily in a common one. Thus, we may informally refer to ``the'' monster model of $T$.

There is another topology on types that plays an important role in the model theory of metric spaces.   
Let $t(x_1,\dots, x_n)$ and $t'(x_1,\dots, x_n)$ are types consistent with $T$ (i.e., $t$ and $t'$ are realized in the monster model) over the same set of parameters (which can be thought of as a subset of the monster model). 
We define $d(t,t')$ as the infimum of all distances $d(\bar a,\bar a')$ where $\bar a$ realizes $t$ and $\bar a'$ realizes $t'$. 
Note that this infimum is realized by some pair $\bar a,\bar a'$ in the monster model. 
It can be readily verified, using the saturation of the monster model, that $d$ is a metric and that $(\mathscr T_n^{(r)}, d)$ is compact and complete.

\section*{Appendix: On the notion of ``finitary'' properties of metric structures}

We conclude this manuscript with some general remarks on the meaning of the informal term ``finitary'' in our context. 
Describing a certain mathematical property as finitary presupposes, in our view, the existence of a formal language~$\cL$ in which the property can be formulated (here we use the term ``language'' in an abstract sense not restricted to first-order languages or Henson's language of positive bounded formulas). 
There is an inherent tradeoff between the expressive power of logical languages and the strength of their model-theoretic properties. 
On one hand, if the language $\cL$ is rich enough (for example, if $\cL$ admits infinite conjunctions and disjunctions), it may capture complex properties with, say, a single formula; 
however, a powerful model-theoretic property like the Compactness Theorem (Theorem~\ref{T:compactness}) is bound to fail for such~$\cL$. 
On the other hand, compactness of first-order-like logics (including Henson's logic) is, in essence, a reflection of the limited expressive power of the language. 
In fact, for metric structures, there is no logic strictly more expressive than Henson logic satisfying both the Compactness Theorem and the elementary chain property (Proposition~\ref{P:elementary chain})~\cite{Iovino:2001}.

A strong feature of the notion of approximate satisfaction of positive bounded formulas is that it inherently captures ``asymptotically approximable'' properties of $L$-structures: 
By definition, the approximate satisfaction of an $L$-formula $\varphi$ amounts to the discrete satisfaction of the full set $\varphi_+$ of formulas $\psi$ approximating~$\varphi$. 
Now, if $\Phi$ is a set of positive bounded of $L$-formulas such that $\Phi\appmod\varphi$ (i.e., every model of $\Phi$ satisfies $\varphi$ approximately), then every approximation of $\varphi$ admits a finite-length proof from~$\Phi$ plus the axioms for the real numbers and metric spaces.\footnote{This was first observed by C.~Ward Henson in the 1980s. 
He gave lectures on this material, but his lecture notes were not formally published.
The idea of finite provability via approximations can be traced back to Mostowski, who proved~\cite{Mostowski:1961} that  for any rational $r < 1$, the set of sentences of {\L}ukasiewicz logic that take truth value $>r$ in all structures is recursively enumerable. 
Simultaneously, E. Specker's student B. Scarpellini had proved in his 1961 dissertation that the valid sentences of {\L}ukasiewicz logic are not axiomatizable.} 
but $\varphi$ itself may not admit a single such proof. 
A property $P$ captured by the approximate satisfaction of a single formula~$\varphi$ should by all rights be called finitary, although $P$ may not admit a finite-length proof. 
Still, it is enough for each rational $\epsilon>0$ to prove $\varphi_{\epsilon}$, the formula obtained from~$\varphi$ by $\epsilon$-relaxing every inequality and quantifier bound in $\varphi$. 
This approach, though \emph{sensu stricto} infinitary, is much the one used when a proof in analysis starts with the incantation: \emph{``Let $\epsilon>0$ be given.''}
(Of course, it is always desirable that ``the same'' proof works for all $\epsilon>0$ in the sense that $\epsilon$ only appears in formulas used in the proof as a parameter, so that a uniform scheme proves all~$\varphi_{\epsilon}$.)

More generally, any property~$P$ equivalent to the simultaneous satisfaction of a collection~$\Psi$ of positive bounded formulas is finitary in the following sense: 
If $P$ holds for all structures in a class~$\mathcal{C}$ axiomatized by a theory~$T$, then $P$ admits, in principle, a proof scheme (say, a syntactic proof for each $\psi\in\Psi_+$).
Equivalently, the \emph{failure} of~$P$ in a structure~$\cM$ is equivalent to the discrete satisfaction $\cM\models\wneg\psi$ of the weak negation of a single formula $\psi\in\Psi_+$.
Thus, a finitary property is witnessed by the absence of counterexamples having finite-length proofs. 

Tao observes in his paper on ergodic averages~\cite{Tao:2008} that metastability is connected to ideas from proof theory; thanking U.~Kohlenbach for the observation,  he notes that metastability is an instance of Kreisel's no-counterexample interpretation~\cite{Kreisel:1951,Kreisel:1952}, which is in turn a particular case of G\"odel's Dialectica interpretation~\cite{Godel:1958}.  In fact, before Tao's paper, the concept had been used extensively, under different nomenclature, by Avigad, Gerhardy, Kohlenbach, and Towsner in the context of proof mining \cite{Avigad-Gerhardy-Towsner:2010, Kohlenbach-Lambov:2004, Kohlenbach:2005a,Kohlenbach:2008}.
For a more up-to-date survey on metastabillity rates obtained by proof mining, see Kohlenbach's lecture at the 2018 International Congress of Mathematicians~\cite[pp.~68-69]{Kohlenbach:2018}.

The property that a sequence $(a_n)$ in a metric $L$-structure be convergent cannot, in general, be captured by the simultaneous satisfaction of any collection (whether finite or infinite) of positive bounded $L$-formulas. 
However, the property \emph{``$(a_n)$ is convergent''} is equivalent to the disjunction of the properties \emph{``$\Eb$ is a rate of metastability for~$(a_n)$''} over all possible metastability rates~$\Eb$. 
For a specific~$\Eb$, the latter property is equivalent to the conjunction of the (infinitely many) properties \emph{``$E_{\epsilon,\eta}$ is a rate of $[\epsilon,\eta]$-metastability for~$(a_n)$''} for all $\epsilon>0$ and samplings~$\eta$.
Thus, from the perspective discussed above, metastable convergence with rate~$\Eb$ is a finitary property of metric structures, while convergence with no specified rate is not.

\bibliographystyle{alpha}
\bibliography{../iovino}
\end{document}